\documentclass[a4paper]{amsart}

\usepackage[applemac]{inputenc}
\usepackage[T1]{fontenc}
\usepackage[english]{babel}
\usepackage{babelbib}
\usepackage{enumerate}
\usepackage{url}
\usepackage[all,cmtip]{xy}
\usepackage{amssymb}
\usepackage{color}
\usepackage{subcaption}
\usepackage{tkz-fct}
\usepackage{tikz}
\usepackage{tikz-cd}
\usetikzlibrary{matrix}

\newcommand{\HH}{\CohomH}

\DeclareMathOperator{\M}{M}
\DeclareMathOperator{\U}{U}
\DeclareMathOperator{\SU}{SU}
\DeclareMathOperator{\PU}{PU}
\DeclareMathOperator{\SL}{SL}
\DeclareMathOperator{\GL}{GL}
\DeclareMathOperator{\PGL}{PGL}
\DeclareMathOperator{\SO}{SO}

\DeclareMathOperator{\End}{End}

\DeclareMathOperator{\Lie}{Lie}

\DeclareMathOperator{\rad}{rad}
\DeclareMathOperator{\Opp}{Opp}

\DeclareMathOperator{\Spec}{Spec}
\DeclareMathOperator{\Proj}{Proj}
\DeclareMathOperator{\Stab}{Stab}
\DeclareMathOperator{\Trans}{Trans}
\DeclareMathOperator{\Sym}{Sym}
\DeclareMathOperator{\id}{id}
\DeclareMathOperator{\rk}{rk}
\DeclareMathOperator{\CohomH}{H}
\DeclareMathOperator{\RealPart}{Re}

\DeclareMathOperator{\Tr}{Tr}
\DeclareMathOperator{\Gr}{Gr}
\DeclareMathOperator{\Fl}{Flag}
\DeclareMathOperator{\pr}{pr}
\DeclareMathOperator{\interior}{int}
\DeclareMathOperator{\Flag}{Flag}
\DeclareMathOperator{\Gal}{Gal}
\DeclareMathOperator{\Res}{Res}
\DeclareMathOperator{\gen}{gen}

\renewcommand{\ss}{\textup{ss}}
\newcommand{\s}{\textup{s}}
\renewcommand{\sp}{\textup{sp}}

\newcommand{\op}{\textup{op}}

\newcommand{\A}{\mathbb{A}}
\newcommand{\C}{\mathbb{C}}

\newcommand{\G}{\mathbb{G}}
\newcommand{\Gm}{\mathbb{G}_m}
\renewcommand{\H}{\mathbb{H}}

\renewcommand{\P}{\mathbb{P}}

\newcommand{\R}{\mathbb{R}}
\renewcommand{\SS}{\mathbb{S}}

\newcommand{\Z}{\mathbb{Z}}

\renewcommand{\O}{\mathcal{O}}

\renewcommand{\epsilon}{\varepsilon}

\newcommand{\df}{:=}
\newcommand{\too}{\longrightarrow}
\newcommand{\iso}{\simeq}
\newcommand{\ol}{\overline}
\renewcommand{\Re}{\RealPart}

\renewcommand{\int}{\interior}

\theoremstyle{plain}
\newtheorem{theorem}{Theorem}[section]
\newtheorem{lemma}[theorem]{Lemma}
\newtheorem{proposition}[theorem]{Proposition}
\newtheorem{corollary}[theorem]{Corollary}

\theoremstyle{definition}
\newtheorem{definition}[theorem]{Definition}
\newtheorem{example}[theorem]{Example}

\theoremstyle{remark}
\newtheorem{remark}[theorem]{Remark}

\numberwithin{equation}{subsection}

\begin{document}

\title[]{Configurations of flags in orbits of real forms}

\author[Elisha Falbel]{Elisha Falbel}
\author[Marco Maculan]{Marco Maculan}
\author[Giulia Sarfatti]{Giulia Sarfatti}

\address[Elisha Falbel, Marco Maculan]{Institut de Math\'ematiques de Jussieu, Universit\'e Pierre et Marie Curie, 4 place Jussieu, F-75252 Paris, France}
\email{\{elisha.falbel, marco.maculan\}@imj-prg.fr}

\address[Elisha Falbel]{INRIA, Ouragan, Paris-Rocquencourt}

\address[Giulia Sarfatti]{Dipartimento di Matematica e Informatica ``U. Dini'', Universit\`a di Firenze, Viale Morgagni 67/A, 50134 Firenze, Italy}
\email{sarfatti@math.unifi.it}

\thanks{The third author is partially supported by INDAM-GNSAGA, by the 2014 SIR grant {\em Analytic Aspects in Complex and Hypercomplex Geometry} and by Finanziamento Premiale FOE 2014 {\em Splines for accUrate NumeRics: adaptIve models for Simulation Environments} of the Italian Ministry of Education (MIUR). Part of this project has been developed while she was an INdAM-COFUND fellow at IMJ-PRG Paris}

\date{}

\begin{abstract}
In this paper we start the study of configurations of flags in closed orbits of real forms using mainly tools of GIT.   As an application, using cross ratio coordinates for generic configurations, we identify boundary unipotent representations
of the fundamental group of  the figure eight knot complement  into real forms of $\PGL(4,\C)$. 

\end{abstract}

\maketitle

\tableofcontents

\section{Introduction}

\subsection{Configurations of flags} Let $n\ge 1$ be a positive integer and $V = \C^n$. Given integers $1 \le d \le n-1$ and $r \ge 2$, what are the possible configurations of $r$ ordered sub-vector spaces of dimension $d$?

A precise way to formulate this question is to consider the Grassmannian $\Gr_d(V)$ of $d$-dimensional subspaces, which is a complex manifold, and study the quotient $\Gr_d(V)^r / \GL(V)$. Yet the quotient topology on $\Gr_d(V)^r / \GL(V)$ is wildly pathological.

The idea behind Geometric Invariant Theory (GIT) is to circumvent this obstacle by excluding some degenerate configurations. To describe the non-degenerate configurations, embed first the Grassmannian $\Gr_d(V)^r$ into the projective space $\P((\textstyle \bigwedge^d V)^{\otimes r})$ by composing the Pl\"ucker and the Segre embedding:
$$
\begin{tikzcd}[row sep=0]
\Gr_d(V)^r \ar[r, "\textup{Pl\"ucker}"] & \P(\textstyle \bigwedge^d V)^r \ar[r, "\textup{Segre}"]& \P((\textstyle \bigwedge^d V)^{\otimes r}) \\
(W_1, \dots, W_r) \ar[r, mapsto]  & (\bigwedge^d W_1, \dots, \bigwedge^d W_r) \ar[r, mapsto]&  \bigwedge^d W_1 \otimes \cdots \otimes \bigwedge^d W_r.
\end{tikzcd}
$$
Then, the non-degenerate configurations, called semi-stable, are defined as follows:

\begin{definition} A point $x \in \Gr_d(V)^r$ is said to be \emph{semi-stable} if there is a homogeneous polynomial  on $(\textstyle \bigwedge^d V)^{\otimes r}$ of positive degree which is invariant under the action of $\SL_n(\C)$ and does not vanish at $x$.
\end{definition}

The subset of semi-stable points $\Gr_d(V)^{r, \ss}$ is open and contains the closure of the orbits of its points. The set $Y$ of the closed orbits of $\Gr_d(V)^{r, \ss}$, which can be seen also as the quotient by the equivalence relation
$$ x \sim x' \Longleftrightarrow \textup{the closure of their orbits meet in $\Gr_d(V)^{r, \ss}$},$$
is a Hausdorff topological space. Even better, $Y$ inherits a structure of complex algebraic variety, even though it is often singular. 

A warning on notations: in the literature, both the symbols $\Gr_d(V)^{r, \ss} / \SL_n(\C)$ and  $\Gr_d(V)^{r, \ss} /\!\!/ \SL_n(\C)$ are used to indicate the quotient $Y$. Throughout this article, we will prefer the first, as the whole set of orbits in $\Gr_d(V)^{r, \ss}$ with its quotient topology will never be considered.

Understanding $Y$ for $n = 2$ and $d = 1$, \emph{i.e.} configurations of $r$ ordered points on the complex projective line $\P^1(\C)$, is a non trivial task already for $r = 5$  and for $r \ge 11$ a complete answer is not known\footnote{The quotient $Y / \mathfrak{S}_r$ is described by the ring of invariants of \emph{binary quantics}: it is the ring of polynomials in the variables $a_0, \dots, a_r$ that are invariant under the action of $\SL_2(\C)$ on $\C[a_0, \dots, a_r]$ defined by  $(g\cdot f)(x,y) = f(g^{-1}(x,y))$
where $g \in \SL_2(\C)$ and $f = \sum_i a_i x^{r-i} y^i$.}.

The question undertaken here is different in two directions. First, instead of a single subspace of $V$, we consider \emph{flags} of subspaces
$$ 0 \subsetneq W_1 \subsetneq W_2 \subsetneq \cdots \subsetneq W_N \subsetneq V.$$
Second, the sub-vector spaces have additional properties, like being totally isotropic with respect to a given hermitian form. The following example are studied in detail:
\begin{enumerate}
\item $r$ couples $(L_i, H_i)$ made of a vector line $L_i$ and a hyperplane $H_i$ containing it, with one of the following additional properties:
\begin{enumerate}
\item $L_i$ and $H_i$ are complexifications of real sub-vector spaces of $\R^n$;
\item $L_i$ is an isotropic line with respect to a hermitian form of signature $(1, n-1)$, and $H_i$ is its orthogonal.
\end{enumerate}
An explicit description is given for $r = 3,4$.
\item Four planes in $\C^4$, with one of the following additional properties:
\begin{enumerate}
\item the planes are complexifications of real sub-vector spaces of $\R^4$;
\item the planes are totally isotropic with respect to a hermitian form of signature $(2,2)$;
\item the planes correspond to quaternionic lines in the quaternionic plane.
\end{enumerate}
\end{enumerate}
The original motivation for this work was the general idea that from decorated triangulations of a manifold one obtains precise informations on representations of its fundamental group and eventually on geometric structures associated to it (see Section \ref{proposition:generic4toinvariants}).  A decorated triangulation of a manifold
is an assignment of a configuration of flags to each simplex of the triangulation.  The flags are thought to be attached to the vertices of the simplices and one imposes that the pairings between the simplices are compatible with the configurations of flags.  If the stabilizer of the face configurations 
for each simplex is trivial one can define a representation of the fundamental group of the triangulated manifold.  This was implemented in the case of surfaces in \cite{FG} and for 3-manifolds in \cite{BFG,GGZ,DGG}.  The most studied case consists of complete flags in $V$ acted upon by 
$\PGL(n,\C)$.  In the case of 3-manifolds, it was observed in certain cases (see \cite{F,FKR,FS}), that if the decoration has values in configurations in an orbit of a real form then the 3-manifold has a geometric structure whose holonomy coincides with the representation obtained from the decorated triangulation.  This justifies the search of decorated triangulations with values in configurations of real flags.

In computations, coordinates for generic configurations of flags \cite{FG,BFG,GGZ} were used to obtain a census of decorated triangulations (see \cite{FKR,C}).  We will describe these coordinates (Section \ref{section:coordiantesgeneric}) defined on an open set of the configuration space and obtain, in these coordinates, the projection onto the previous spaces of partial flags configurations (see Propositions \ref{propositionz>w} and \ref{proposition:generic4toinvariants}).  This allows us to write  equations for configurations to be 
in real orbits in the cases of $\PU(n,1)$ and all real forms of $\PGL(4,\C)$.  In the last section we explain the use of decorated triangulations to obtain representations with values in real forms. 
Finally, from the computation of boundary unipotent generically decorated representations of the fundamental group of the figure eight 
knot obtained in \cite{C}, we obtain that the only one of these representations decorated by configurations in a real orbit (for a real form of  $\PGL(4,\C)$) is a representation into $\SO(3,1)$ (see Proposition \ref{proposition:figureeight}). 

\subsection{The role of real GIT} Let  $G_0$ be a real form of $\SL_n(\C)$. The Cartan involution associated with $G_0$ allows to define a real algebraic group $G$ whose real points $G(\R)$ are naturally identified with $G_0$, and the complex points $G(\C)$ with $\SL_n(\C)$.

Given integers $0 < d_1\le  \dots \le d_N < n$, the group $G_0$ acts on the flag variety $\Fl_{d}(V)$ of flags of $V$ of type $d = (d_1, \dots, d_N)$. A theorem of J.  Wolf (\cite{W}) asserts that there exists a unique
closed orbit $F_0$ of $G_0$ in $\Fl_{d}(V)$, and 
\begin{equation} \label{eq:WolfsInequality}\dim_\R F_0 \ge \dim_\C \Fl_{d}(V). \end{equation}
Moreover, equality comes exactly when there is a flag $x \in \Flag_d(V)$ whose stabilizer $\Stab_{\GL(V)}(x)$ is stable under the Cartan involution defining $G_0$. If this is the case, the stabilizer $\Stab_{\GL(V)}(x)$ is the group of complex points of a parabolic subgroup $P$ of the real algebraic $G$, and the real projective variety $F:= G/P$ verifies $F(\R) = F_0$ and $F(\C) = \Flag_d(V)$.

The map induced between quotients,
$$ F_0^r / G_0 \too \Flag_d(V)^r / \SL_n(\C),$$
is in general far from being injective, and inequality \eqref{eq:WolfsInequality}  propagates to the quotient: there are dense open subsets $U_0 \subset F_0^r / G_0$ and $U \subset  \Flag_d(V)^r / \SL_n(\C)$ that are respectively a real and a complex manifold, and
$$ \dim_\R U_0 \ge \dim_\C U,$$
with equality precisely when $\dim_\R F_0 = \dim_\C \Fl_{d}(V)$.

In this paper we work under the hypothesis $\dim_\R F_0 = \dim_\C \Fl_{d}(V)$. As explained above, the closed orbit $F_0$ is the set of real points of a real projective variety $F$ whose complex points are identified with $\Flag_d(V)$. The problem of studying the quotient $F_0^r / G_0$ can be approached from the point of view of real Geometric Invariant Theory, as developed by Luna, Birkes, Richardson-Slodowy (\cite{LunaRealGIT}, \cite{LunaFonctionsDifferentiables}, \cite{Birkes}, \cite{RichardsonSlodowy}).

Let $X := F^r$. The set of semi-stable points of $X(\C) = \Flag_d(V)^r$ is stable under complex conjugation and defines a Zariski open subset $X^{\ss}$. Unwinding the definitions, a point $x$ of $X(\R) = F_0^r$ is semi-stable if it is so as a point of $\Flag_d(V)^r$ under the action of $\SL_n(\C)$. The GIT quotient of $X^\ss$ by $G$ inherits the structure of a real projective variety $Y$: the set of its complex points is the set of closed orbits of $\SL_n(\C)$ in $X^\ss(\C) = \Flag_d(V)^r$. 

However, the set of real points does not necessarily verify the analogous property for real points. To be more precise let us introduce the set $X^\ss(\R) / G(\R)$ of closed $G(\R)$-orbits in $X(\R)$: following Luna, since the set of all $G(\R)$-orbits is not so interesting, in this paper we reserve the symbol $X^\ss(\R) / G(\R)$ for the set of closed orbits. 

\begin{theorem} With the notations introduced above, the following holds:
\begin{enumerate}
\item for $x \in X^\ss(\R)$ there is a unique closed $G(\R)$-orbit contained in $$\ol{G(\R) \cdot x} \cap X^\ss(\R); $$
\item the set $X^\ss(\R) / G(\R)$, endowed with the quotient topology given by seeing it as the quotient by the equivalence relations
$$ x \sim x' \Longleftrightarrow \ol{G(\R) \cdot x} \cap \ol{G(\R) \cdot x'} \cap X^\ss(\R) \neq \emptyset,$$
is Hausdorff and compact;
\item the natural map  $\theta \colon X^\ss(\R) / G(\R) \too Y(\R) $ is proper, with finite fibres and open onto its image.
\end{enumerate}
\end{theorem}

The map $\theta$ is neither injective nor surjective in general, as the examples discussed in this text show. The openness of $\theta$ is proved in Section \ref{section:GIT} as we are unable to retrieve this statement from the literature.

\subsection{Statement of the results}

\subsubsection{Arrangements of flags line-hyperplane} In the first example we study configurations of flags of the type $(L, H)$ where $L$ is a line of $V$ and $H$ is a hyperplane containing it.

We consider the following two real forms of $\SL_n(\C)$ admitting a parabolic subgroup stabilizing such a flag:
\begin{itemize}
\item[(a)] the split form $\SL_n(\R)$;
\item[(b)] the special unitary group $\SU(1, n-1)$ with respect to a hermitian form on $V$ of signature $(1, n-1)$.
\end{itemize}

A $r$-tuple $\{ (L_i, H_i) \}_{i = 1, \dots, r}$ of flags line-hyperplane is semi-stable if and only if there exists a permutation $\sigma \in \mathfrak{S}_r$ (necessarily without fixed points) such that, for all $i = 1, \dots, r$, the line $L_i$ is not contained in $H_{\sigma(i)}$.

To a fixed-point free permutation $\sigma$, called also derangement, one associates an invariant $s_\sigma$ consisting in evaluating a linear form defining the hyperplane $H_{\sigma(i)}$ on a generator of the line $L_{i}$.

Through the invariants $s_\sigma$, the GIT quotient $Y_{n, r}$ of $\Fl_{1, n-1}(V)^{r, \ss}$, which is a variety of dimension $$r^2 - 3r +1 - \min \{ 0, n-r\}^2,$$ is embedded in the projective space $\P^{!r -1}(\C)$ where $!r$ is the number of derangements in $\mathfrak{S}_r$. Similarly to the factorial of a number, for $r \ge 3$ the number of derangements satisfies the recurrence relation
$$ !r = (r-1)(!(r-1) + !(r-2)).$$ 
In particular $!3 = 2$ (the only derangements being the two $3$-cycles) and $!4 = 9$.
\begin{itemize}
\item $Y_{n, 3}= \P^1(\C)$ and the quotient map is given by the triple ratio;
\item For $r = 4$ and $n \ge 4$, $Y_{n, 4}$ is the complete intersection in $\P^8(\C)$ of the quadrics
\begin{align*}
x_1 x_5 &= x_3 x_4, &
x_1 x_9 &= x_2 x_7, &
x_5 x_9 &= x_6 x_8.
\end{align*}
If $n = 3$ the quotient $Y_{3, 4}$ is a hyperplane section of the previous variety.
\end{itemize}
Let  $\{ (L_i, H_i)\}_{i = 1, \dots, r}$ be a $r$-tuple of flags. If $n \ge r$, it is said to be in \emph{general position} if the following properties are satisfied:
\begin{enumerate}
\item $\dim L_1 + \cdots + L_r = r$;
\item $\dim H_1 \cap \cdots \cap H_r = n-r$;
\item $L_i \cap H_j = 0$ whenever $i \neq j$;
\item $(L_1 + \cdots + L_r) \cap H_1 \cap \cdots \cap H_r = 0$.
\end{enumerate}
If $r > n$, the $r$-tuple is in general position if every subset of $n$ elements verifies the previous conditions. 

When $n = 3$ the flags of the type line-hyperplane are complete flags. The notion of ``general position'' then competes with that of ``generic configurations of flags'' of Goncharov. However, we show in Example \ref{ex:GenericVsGeneralPosition} (see also Proposition \ref{Prop:QuotientLineHyperplaneSingularFiber}) that the orbits of generic configurations are not necessarily closed. In this sense, configurations in general position are better behaved with respect to GIT. 

Consider the real form $G_0 = \SL_n(\R)$ and the associated real algebraic group $G = \SL_{n, \R}$. The set of real points of the associated flag variety $F$ is the set of $(L, H)$ where $L$ is a real line of $\R^n$ and $H$ is a real hyperplane containing it. Let $U$ be the open subset of $X = F^r$ made of configurations of flags in general position.
\begin{theorem} With the notations introduced above, the map
$$ \theta \colon U(\R) / \SL_n(\R) \too \P^{!r - 1}(\R),$$
induced by the invariants $s_\sigma$, is
\begin{itemize}
\item injective if $n$ is odd or $n > r$;
\item two-to-one if $n$ is even and $n \ge r$.
\end{itemize}
\end{theorem}
This lack of injectivity can be prevented by taking the quotient by the group $\GL_n(\R)$ instead of $\SL_n(\R)$.

Consider the special unitary group $\SU(1, n-1)$ with respect to a hermitian form $h$ of signature $(1, n-1)$. The unique closed orbit $F_0$ in $\Fl_{1, n-1}(V)$ is made of couples $(L, H)$ where $L$ is an isotropic line and $H$ its orthogonal: therefore $F_0$ can be identified with the set of isotropic lines in $V$. Consider the open subset $W \subset F_0^{r, \ss}$ made of $r$-tuples of lines $L_1, \dots, L_r$ generating a space of dimension $\ge 3$. 
\begin{theorem} The map $\theta \colon W / \SU(1, n-1) \to \P^{!r - 1}(\R)$ induced by the invariants $s_\sigma$ is injective.
\end{theorem}
As the case $r = 3$ shows, the map $\theta$ on the complement of $W$ is not injective, so that the preceding theorem is sharp. Nonetheless, a more precise statement is proved when $r$ is odd. A $r$-tuple $x = (L_1, \dots, L_r)$ of isotropic lines is seen to be semi-stable if and only if none of them is repeated more than $r/2$ times. Pick for every line $L_i$ a generator $v_i$ and, for a derangement $\sigma$, set
$$ \epsilon_\sigma(x) = \prod_{i = 1}^r h( v_i, v_{\sigma(i)} ).$$
Note that $\epsilon_{\sigma}(x)$ is the complex conjugated of $\epsilon_{\sigma^{-1}}(x)$. Let $D$ be a subset of the derangements of $\mathfrak{S}_r$ with $!r/2$ elements and the property: $ \sigma \in D \Rightarrow \sigma^{-1} \not \in D$.
\begin{theorem} With the notations introduced above, the map 
$$\epsilon \colon F_0^{r, \ss}(\R) / \SU(1, n-1) \too (\C^{!r/2} \smallsetminus \{ 0 \}) / \R_{>0},$$ defined by sending a point $F_0^{r, \ss}(\R)$ to the open ray generated by $(\epsilon_\sigma(x))_{\sigma \in D}$, is injective.
\end{theorem}

\subsubsection{Four planes} The second example studies the real forms $G_0$ of $\SL_4(\C)$ admitting a parabolic subgroup $P$ stabilizing  a plane in $\C^4$. There are three (isomorphism classes) of this kind:
\begin{itemize}
\item[(a)] the split form $\SL_4(\R)$;
\item[(b)] the special unitary group $\SU(2,2)$ with respect to a hermitian form on $\C^4$ of signature $(2,2)$;
\item[(c)] if $\H$ denotes the Hamilton quaternions, the group $\SL_2(\H)$ of $\H$-linear automorphisms of $\H^2$ whose associated complex matrix has determinant $1$\footnote{To a quaternion $q = a + bi + cj +dk \in \H$ one can associate the $2 \times 2$ matrix with complex entries 
$$m(q) := \begin{pmatrix} a + bi & c + di \\ -c +di & a - bi \end{pmatrix}.$$
This induces an isomorphism of non-commutative $\C$-algebras $m \colon H \otimes_\R \C \to \End(\C^2)$. A $2 \times 2$ matrix with quaternionic coefficients,
$$ \begin{pmatrix} q_1 & q_2 \\ q_3 & q_4\end{pmatrix} \in \GL_2(\H)$$
belongs to $\SL_2(\H)$ if and only if the associated $4 \times 4$ matrix with complex coefficients,
$$ \begin{pmatrix} m(q_1) & m(q_2) \\ m(q_3) & m(q_4) \end{pmatrix},$$
has determinant $1$. If $q_1 \neq 0$ this is equivalent to saying that the norm of $q_1 q_4 - q_1 q_3 q_1^{-1} q_2$ is $1$.}.
\end{itemize}

A quadruple $(W_1, W_2, W_3, W_4)$ of complex planes of $\C^4$ is semi-stable with respect to $\SL_4(\C)$ if and only $W_{\sigma(1)} \cap W_{\sigma(2)} = 0$ and $ W_{\sigma(3)} \cap W_{\sigma(4)} = 0$ for some permutation $\sigma \in \mathfrak{S}_4$. 

The GIT quotient of $\Gr_2(\C^4)^\ss$ by $\SL_4(\C)$ is the complex projective plane $\P^2(\C)$, and the quotient map is given by $3$ invariants $s_{1234}$, $s_{1324}$ and $s_{1423}$: for a semi-stable quadruple of planes $(W_1, W_2, W_3, W_4)$ and for basis $w_{i1}, w_{i2}$ of $W_i$, the invariant $s_{1234}$ is defined as the product of determinants
$$ s_{1234} = \det(w_{11}, w_{12}, w_{21}, w_{22}) \cdot \det(w_{31}, w_{32}, w_{41}, w_{42}).$$
The invariants $s_{1324}$ and $s_{1423}$ are defined similarly.

For a real form $G_0$ of $\SL_4(\C)$ among the ones mentioned above, let $G$ the be associated real reductive groupe, $F$ the flag variety $G / P$ and $X = F^4$. The GIT quotient $Y$ of $X^{\ss}$ by $G$ is a form of the projective space, \emph{i.e.} it becomes isomorphic to $\P^2(\C)$ as soon as one extends scalars to $\C$. 

In other words, it is a Severi-Brauer variety: it is the ``trivial'' one, that is, the one isomorphic to $\P^2_\R$ as a real algebraic variety. In particular, $Y(\R) = \P^2(\R)$.

\begin{theorem} With the notations introduced above, consider the natural map
$$\theta \colon F_0^{4, \ss} / G(\R) \too \P^2(\R). $$

Then, in cases (a) and (b) it is bijective (thus a homeomorphism). In case (c) it is injective and its image is the set of points $[t_0:t_2 : t_2] \in \P^2(\R)$ satisfying the inequality
$$t_0^2 + t_1^2 + t_2^2 \le 2(t_0 t_1 + t_1 t_2 + t_0 t_2). $$
\end{theorem}

The results are better resumed in the following table:
\renewcommand{\arraystretch}{2}

\begin{center}
\begin{tabular}{c|c|c}
$G_0$ & $F_0$ & $F_0^{4, \ss} / G_0$ \\
\hline
$\SL_4(\R)$ & $\{ \textup{real planes in $\R^4$} \}$ & $\P^2(\R)$ \\
$\SU(2,2)$ & $\{ \textup{totally isotropic planes} \}$ & $\P^2(\R)$ \\
$\SL_2(\H)$ & $\{ \textup{quaternionic lines in $\H^2$} \}$ & 
$
\renewcommand{\arraystretch}{1}
\begin{matrix} \textup{closed interior in $\P^2(\R)$ of the conic} \\ t_0^2 + t_1^2 + t_2^2 \le 2(t_0 t_1 + t_1 t_2 + t_0 t_2) \end{matrix}
$
\end{tabular}
\end{center}
\renewcommand{\arraystretch}{1}

In case (c), if the four quaternionic lines are pairwise distinct, the projection on the quotient makes intervene the quaternionic cross-ratio studied in \cite{BisiGentili, GwynneLibine}.

\subsection{Organization of the paper} 

In Section \ref{sec:conventions} we recall some basic facts on real algebraic groups and flag varieties. In Section \ref{section:GIT} we pass in review GIT over the real numbers. In Section \ref{section:Sequences} we treat the example of flags line-hyperplane. In Section \ref{sec:QuadruplesOfPlanes}, the case of four plane is treated. In Section \ref{section:coordiantesgeneric} the link with generic configurations of flags is studied, while in the last section we discuss the relation with decorated representations.

\subsection{Acknowledgements} We thank Claudio Gorodski, Antonin Guilloux, Greg Kuperberg and Maxime Wolff for several discussions.

\section{Conventions and reminders} \label{sec:conventions}

\subsection{Real algebraic varieties} In this paper we try not to confuse algebraic varieties with their set of points: while it may seem pedantic in the complex case, it is worth it in the real one. If $X$ is a complex algebraic variety, its set complex points is denoted by $X(\C)$; if $X$ is real the set of its real points is denoted by $X(\R)$, its complexification $X_\C$ and its complex points $X(\C)$. The set of complex (resp. real) points will be always endowed with its natural complex (resp. real) topology.

For a positive integer $n \ge 1$ the $n$-dimensional affine space is denoted $\A^n$, the $n$-dimensional projective space $\P^n$. The associated real/complex points are
\begin{align*}
\A^n(\R) &= \R^n, &\A^n(\C)&= \C^n, \\ \P^n(\R)&= (\R^n \smallsetminus \{ 0\})/\R^\times, & \P^n(\C)&= (\C^n \smallsetminus \{ 0\})/\C^\times .
\end{align*}
 If we want to insist on the field of definition we write $\A^n_\R$, $\A^n_\C$, $\P^n_\R$ or $\P^n_\C$. 

The datum of a quasi-projective real algebraic variety $X$ is equivalent to the datum of a complex algebraic variety $X_\C$ (its \emph{complexification}) together with an anti-holomorphic involution $\tau_X$ such that there exists a locally closed embedding $\epsilon \colon X \to \P^n_\C$ verifying, for all $x \in \P^n(\C)$,
$$ \epsilon(\tau_X(x)) = \overline{\epsilon(x)},$$
where, for a complex number $z \in \C$, $\bar{z}$ stands for its conjugate and, for a point $x = [x_0 : \cdots : x_n] \in \P^n(\C)$, 
$$ \bar{x} = [\bar{x}_0 : \cdots : \bar{x}_n]. $$
This is equivalent to asking that the polynomials defining $X_\C$ in $\P^n_\C$ have real coefficients. 

A map $f \colon X \to Y$ between real algebraic varieties is a complex algebraic map $f_\C \colon X_\C \to Y_\C$ compatible with the anti-holomorphic involutions $\tau_X, \tau_Y$ of $X, Y$, that is, $\tau_Y \circ f_\C = f_\C \circ \tau_X$.

\subsection{Real algebraic groups and forms} A real algebraic group is an algebraic variety $G$ together with maps
\begin{align*}
\textup{mult} &\colon G \times_\R G \to G, \\
\textup{inv} &\colon G  \to G,
\end{align*}
defining respectively the group law and the inverse, and a point $e \in G(\R)$. The datum is supposed to verify the usual relations \cite[Chapter I, \S1.5]{Borel}.

Given a complex algebraic group $G$, a real form of $G$ is a real algebraic group $H$ whose complexification $H_{\C}$ is isomorphic to $G$. It is well known that real forms do not need to be isomorphic. For instance, the following real algebraic groups defined as closed subvarieties of $\A^2_\R$,
\begin{align*}
\Gm &: xy = 1, &\U(1) &: x^2 + y^2 = 1,
\end{align*}
are isomorphic over $\C$ (this is seen taking the change of coordinates $x \mapsto x + iy$, $y \mapsto x - iy$) but not over $\R$. 

Some words on notation. The multiplicative group is denoted $\Gm$. At the level of real or complex points one has
\begin{align*}
\Gm(\R) &= \R^\times, & \Gm(\C) &= \C^\times.
\end{align*}
If the field of definition is not clear from the context, we specify it by writing $\mathbb{G}_{m, \R}$ or $\mathbb{G}_{m, \C}$.  Similarly for the general linear group: $\GL_n$ denotes the algebraic group, $\GL_n(\R)$ and $\GL_n(\C)$ the groups over matrices, $\GL_{n, \R}$ and $\GL_{n, \C}$ respectively the real and complex algebraic group. For a real algebraic group $G$, we denote by $G_0$ the set of its real points $G(\R)$.

A connected complex algebraic group is said to be reductive if every representation is completely reducible: for every algebraic representation $G \to \GL(V)$ and every $G$-stable subspace $W \subset V$  there exists a $G$-stable complement $W'$. This is equivalent to the fact that the complex Lie group $G(\C)$ contains a Zariski-dense compact subgroup. A semi-simple group is a reductive group whose center is finite.

A real algebraic group is said to be reductive (resp. semi-simple) if its complexification is reductive (resp. semi-simple).

\subsection{Parabolic subgroups} A subgroup $P$ of a complex reductive group is said to be parabolic if the quotient variety $G/P$ is projective. 

A parabolic subgroup of a real reductive group $G$ is an algebraic subgroup $P$ whose complexification $P_\C$ is a parabolic subgroup of $G_\C$. The quotient is then a projective real algebraic variety. A complex parabolic subgroup $Q$ of $G_\C$ is the complexification of a real parabolic subgroup of $G$ if and only if the dimension of $Q$ equals the dimension of the real Lie group $Q(\C) \cap G(\R)$. Over both real and complex numbers, two parabolic subgroups are said to be opposite if their intersection is a common Levi factor. This turns out to be equivalent to say that their intersection is of minimal dimension.

Let $G$ be a real or complex reductive group and let $P$ be a parabolic subgroup of $G$. Then there is a parabolic subgroup $P_\op$ opposite to $P$. Moreover the subset $\Opp(P)$ of $X = G / P_\op$ made of points $x \in X$ whose stabilizer is a parabolic subgroup opposite to $P$ is Zariski open. 
Let $x_\op \in X$ be point associated to $P_\op$. The natural map
\begin{eqnarray*}
P &\too & \Opp(P) \\
g & \longmapsto &  g x_\op
\end{eqnarray*}
induces an isomorphism
$$ \rad^u P \stackrel{\sim}{\too} \Opp(P).$$

\begin{example} Let $V$ be a finite-dimensional real or complex vector space. A parabolic subgroup of $\GL(V)$ is the subgroup stabilizing a flag of vector subspaces 
$$ F_\bullet : 0 \subset F_1 \subset F_2 \subset \cdots \subset F_r,$$
(the inclusions are strict). Conversely the stabilizer of a flag is a parabolic subgroup of $\GL(V)$. Let $F_\bullet = (F_i)_{i = 1, \dots, r}$ be a flag and $P = \Stab(F_\bullet)$ its stabilizer. Set $d_i = \dim F_i$ for all $i = 1, \dots, r$. Then the projective variety $\GL(V) / P$ is the variety of flags $\Flag(d_1, \dots, d_r)$. 
\end{example}

\begin{example}[Incidence variety] Let $V = \C^n$. The \emph{incidence variety} is the flag variety $\Flag_{(1, n-1)}$ of couples $(L, H)$ made of a line $L \subset V$ and of a hyperplane $H \subset $ containing $L$.

It can be seen as a subvariety of $\P(V) \times \P(V^\ast)$ sending a couple $(L, H)$ to the couple $(L, H^\bot)$ where $H^\bot$ is the set of linear forms vanishing on $H$. The equation is
$$ \phi_1 x_1 + \cdots + \phi_n x_n = 0,$$
where $x = (x_1, \dots, x_n) \in V$ and $\phi = \phi_1 e_1^\ast + \cdots + \phi_n e_n^\ast$ (here $e_1^\ast, \dots, e_n^\ast$ is the dual basis of standard one of $\C^n$).
\end{example}

\begin{example}[Klein quadric] Let $V = \C^4$. The flag manifolds correspondences are gathered in the following diagram:
\begin{center}
\begin{tikzcd}[row sep= large, column sep= large]
 & \Flag_{123} \ar[dl] \ar[dr] \ar[d]& \\
\Flag_{12} \ar[d]  & \Flag_{13} \ar[dl] \ar[dr] &\Flag_{23} \ar[d] \ar[dl, crossing over]\\
\Flag_{1}  \ar[dr]& \Flag_{2} \ar[<-, ul, swap, crossing over] \ar[d]&\Flag_{3} \ar[dl]\\
& \{ \textup{pt} \} & 
\end{tikzcd}
\end{center}
Consider the Grassmannian $\Gr(2, 4)$ of planes in $V$. Through the Pl\"ucker embedding
\begin{eqnarray*}
\Gr(2, 4) & \too & \P(\textstyle \bigwedge^2 V) \\
W & \longmapsto & \textstyle \bigwedge^2 W,
\end{eqnarray*}
it can be seen as subvariety of $\P(\textstyle \bigwedge^2 V)$. The wedge product induces a non-degenerate quadratic form on $\bigwedge^2 V$,
\begin{eqnarray*}
Q\colon \textstyle \bigwedge^2 V &\too& \textstyle \bigwedge^4 V\\
w &\longmapsto& w \wedge w,
\end{eqnarray*}
and $\Gr(2, 4)$ is the quadric given by the equation $Q = 0$. It is classically known as \emph{Klein's quadric}. Let $e_1, \dots, e_4$ be the standard basis of $V$ and fix $e_1 \wedge e_2 \wedge e_3 \wedge e_4$ as basis of $\bigwedge^4 V$. Writing a vector $w \in \bigwedge^2 V$ as 
$$ w = \sum_{1 \le i < j \le 4} w_{ij} e_i \wedge e_j,$$
one obtains $$ Q(w) = w \wedge w = w_{12} w_{34} - w_{13}w_{24} + w_{14}w_{23} , $$ recovering the classical equation of Klein's quadric.
\end{example}

\subsection{Galois cohomology} Let $G$ be an affine real algebraic group. A  \emph{(right) principal homogenous $G$-space} is a non-empty real algebraic variety $P$ endowed with a (right) action of $G$ such that the map
\begin{eqnarray*}
P \times_\R G &\too& P \times_\R P \\
(p, g) & \longmapsto & (p, pg),
\end{eqnarray*}
is an isomorphism. A principal homogeneous $G$-space is \emph{trivial} if $P(\R) \neq \emptyset$.

The set of isomorphism classes of principal homogeneous $G$-spaces is denoted $\HH^1(\R, G)$. It can be computed also as the set of $1$-cocyles $\HH^1(\Gal(\C / \R), G(\C))$ (see \cite[Chapitre III]{SerreCohomologieGaloisienne}). The set $\HH^1(\R, G)$ is finite \cite[Chapitre III, Th\'eor\`eme 4]{SerreCohomologieGaloisienne} and is pointed by considering the class of trivial principal homogenous spaces.

\begin{theorem} \label{thm:Satz90} Let $n \ge 1$ be an integer. Then $\HH^1(\R, \GL_n) = 0$.
\end{theorem}

Let $f \colon G \to G'$ a map of affine real algebraic groups and $P$ a right principal homogenous $G$-space. Then the quotient $(P \times_\R G') / G$ of $P \times_\R G'$ by the (left) action $g(g',p) = (f(g) g', p g^{-1})$ of $G$ exists as a real algebraic variety and is a principal homogenous $G'$-space. This construction leads to a map of pointed sets
$$ \HH^1(\R, G) \too \HH^1(\R, G').$$

Suppose $G$ acts on affine variety $X$. Let $x \in X(\R)$ and $T = G \cdot x$ be its orbit (as a real algebraic variety). Then $T(\R)$ is not necessarily made of one single $G(\R)$ orbit. For $x' \in T(\R)$ consider the transporter
$$ \Trans(x, x') = \{ g \in G : gx = x'\}. $$
It is a principal homogeneous $G_x$-space where $G_x$ is the stabilizer of $x$. This leads to a map
$$ \delta \colon T(\R)  \too \HH^1(\R, G_x).$$

\begin{proposition} \label{Prop:RealPointsOrbitGaloisCohomology} The map $\delta$ induces a bijection
$$ T(\R) / G(\R) \stackrel{\sim}{\too} \ker( \HH^1(\R, G_x) \to \HH^1(\R, G) ).$$
In particular, if $\HH^1(\R, G) = 0$ then $T(\R) / G(\R) \iso \HH^1(\R, G_x)$.
\end{proposition}

\begin{proof} The maps $\delta$ fits into the following exact sequence of pointed sets (see \cite[Chapitre I \S 5.4 Proposition 36]{SerreCohomologieGaloisienne}): 
$$ 1 \too G_x(\R) \too G(\R) \too T(\R) \too \HH^1(\R, G_x) \too \HH^1(\R, G).$$
In particular,
\begin{equation*}T(\R) / G(\R) \stackrel{\sim}{\too} \ker(\HH^1(\R, G_x) \to \HH^1(\R, G)). \qedhere \end{equation*}
\end{proof}

\section{GIT over the complex and the real numbers}\label{section:GIT}

\subsection{The complex affine case} Let $G$ be a complex reductive group acting on an affine variety $X$. Geometric Invariant Theory (GIT) deals with the problem of constructing a quotient of $X$ by $G$ as an algebraic variety.

One cannot expect the underlying topological space of the putative quotient to be the quotient of $X$ by $G(\C)$ as a topological space, as the latter is too non separated in general to be an algebraic variety. For instance, looking at the action of $\G_{m, \C}$ on $\A^2_\C$ by $t(x, y) = (tx, ty)$, one sees that the origin lies in the closure of every orbit.

Instead, consider the $\C$-algebra $A$ of regular functions of $X$ and denote by $A^G$ the sub-$\C$-algebra of $G$-invariant elements.

\begin{theorem} The $\C$-algebra $A^G$ is finitely generated. Let $Y = \Spec A^G$. Then, the morphism $\pi \colon X \to Y$ induced by the inclusion $A^G \subset A$ enjoys the following properties:
\begin{enumerate}
\item $\pi$ is $G$-invariant et surjective;
\item if $F \subset X$ is a Zariski closed $G$-stable subset, then $\pi(F) \subset Y$ is closed;
\item let $x, x' \in X(\C);$ then $\pi(x) = \pi(x')$ if and only if
$$ \ol{G \cdot x} \cap \ol{G \cdot x'} \neq \emptyset;$$
\item if $U \subset X$ is a Zariski open $G$-saturated\footnote{That is, for every $x \in X(\C)$, the closure of the orbit $G(\C) \cdot x$ is contained in $U(\C)$.}, then $\pi(U)$ is Zariski open in $Y$;
\item the homomorphism of sheaves $\pi^\ast \colon \O_Y \to \pi_\ast \O_X$ induces an isomorphism
$$ \pi^\ast \colon  \O_Y \stackrel{\sim}{\too} (\pi_\ast \O_X)^G.$$
\end{enumerate}
\end{theorem}

This makes $Y$ the \emph{categorical quotient} of $X$ by $G$ in the category of complex algebraic varieties:  for every complex algebraic variety $Z$ and every $G$-invariant map $f \colon X \to Z$ there exists a unique map $\tilde{f} \colon Y \to Z$ such that $f = \tilde{f} \circ \pi$.

The quotient $Y$ is called the \emph{GIT quotient} of $X$ by $G$ and denoted by $X/G$.\footnote{As already mentioned in the Introduction, a competing symbol in the literature is $X /\!\!/ G$.}

\begin{example} Let $\G_{m, \C}$ on $\A^2_\C$ by $t(x, y) = (tx, ty)$. The action induced on the polynomial ring $\C[x, y]$ is given by
$$ \sum_{i, j} a_{ij} x^i y^j \longmapsto \sum_{i, j} a_{ij} t^{i+j} x^i y^j.$$
No non-constant polynomial invariant is left invariant. In other terms, the quotient of $\A^2_\C$ by $\Gm$ is a point.
\end{example}

\begin{example} \label{ex:GITQuotientConjugation} Let $n \ge 1$ be an integer and $\M_{n, \C}$ be the space of $n \times n$ matrices. The group $\GL_{n, \C}$ acts on it by conjugation. For a matrix $A \in \M_n(\C)$ let 
$$ P_A(T) := \det(T \cdot \id_n - A) = T^n - \sigma_1(A)T^{n-1} + \cdots + (-1)^{n} \sigma_n(A),$$
be the characteristic polynomial of $A$. The map
\begin{eqnarray*}
\sigma \colon \M_{n, \C} &\too& \A^n_\C\\
A &\longmapsto& (\sigma_1(A), \dots, \sigma_n(A)),
\end{eqnarray*}
is $\GL_{n, \C}$-invariant and makes $\A^n_\C$ the GIT quotient of $\M_{n, \C}$ by $\GL_{n, \C}$ (see \cite[Proposition 2]{MumfordOslo}).
\end{example}

\subsection{The real affine case}

Let $X$ be an affine real algebraic variety acted upon by a real reductive group $G$. Let $A$ denote de $\R$-algebra of regular functions on $X$. Then $G$ acts on the $\R$-algebra $A$ and we consider the sub-$\R$-algebra $A^G$ of $G$-invariants elements.

\begin{theorem} \label{Thm:AffineRealGIT}The $\R$-algebra $A^G$ is finitely generated. Let $Y = \Spec A^G$. Then, the morphism $\pi \colon X \to Y$ induced by the inclusion $A^G \subset A$ enjoys the following properties:
\begin{enumerate}
\item $\pi$ is $G$-invariant et surjective;
\item if $F \subset X$ is a Zariski closed $G$-stable subset, then $\pi(F) \subset Y$ is closed;
\item let $x, x' \in X(\C);$ then $\pi(x) = \pi(x')$ if and only if
$$ \ol{G(\C) \cdot x} \cap \ol{G(\C) \cdot x'} \neq \emptyset;$$
\item if $U \subset X$ is a Zariski open $G$-saturated, then $\pi(U)$ is Zariski open in $Y$;
\item we have  $ A^G \otimes_{\R} \C = (A \otimes_\R \C)^{G_{\C}}$.
\end{enumerate}
\end{theorem}

\begin{proof} See \cite[Theorem 1.1]{GIT}.
\end{proof}

The real algebraic variety $Y$ is called the \textit{GIT quotient} of $X$ by $G$ and is denoted $X/G$.

\begin{remark} \label{rmk:RemarksOnGIT} \
\begin{enumerate} 
\item While stating $(2)$ we did not specify whether the closure is taken with respect to the Zariski or the usual topology: this does not matter as those closures coincide (see \cite{SerreGAGA}). 

\item Statement $(2)$ implies that, given a point $x \in X(\C)$, there exists a unique closed orbit contained in $\ol{G(\C) \cdot x}$.

\item Statement (5) can be restated by saying that $Y_\C$ is the GIT quotient of $X_\C$ by $G_\C$.

\item Statements (3) and (4) still hold when we replace the Zariski topology by the usual one, \textit{e.g.} if $F \subset X(\C)$ is a closed $G(\C)$-stable subset, then $\pi(F)$ is closed in $Y(\C)$.
\end{enumerate}
\end{remark}

The task undertaken in this section is to understand what happens at the level of real points. Luna \cite{LunaRealGIT, LunaFonctionsDifferentiables}, Birkes \cite{Birkes}, Richardson-Slodowy \cite{RichardsonSlodowy} and others have extensively studied the problem and the following results will be of particular interest for us:

\begin{theorem} \label{thm:LunaRealGIT}With the notations introduced above, let $x \in X(\R)$. Then the following hold:
\begin{enumerate}
\item the unique closed orbit contained in the closure of $G(\C) \cdot x$ has a real point. In particular, it is defined over $\R$.
\item the orbit $G(\R)\cdot x$ is closed if and only if $G(\C) \cdot x$ is closed;
\item the closure of $G(\R) \cdot x$ in $X(\R)$ contains a unique closed orbit;
\item if $F \subset X(\R)$ is a closed $G(\R)$-stable subset then $\pi(F)$ is closed in $Y(\R)$.
\end{enumerate}
\end{theorem}

\begin{proof} 
(1) See \cite[Corollary 4.4]{KempfInstability}.  The argument ultimately relies on the existence of rational destabilizing one parameter subgroups, which is due to Birkes \cite[Theorem 5.2]{Birkes} over the real numbers whether, over an arbitrary perfect field, is due to Kempf \cite[Theorem 4.2]{KempfInstability} and Rousseau \cite{Rousseau}.

(2) Suppose that the orbit $G(\C) \cdot x$ is not closed. Then by \cite[Theorem 4.2]{KempfInstability} there exists a one parameter subgroup $\lambda \colon \G_{m, \R} \to G$ such that the point
$$ x_0 \df \lim_{t \to 0} \lambda(t) \cdot x \in X(\R),$$
belongs to the unique closed orbit contained in $\ol{G(\C) \cdot x}$. Thus $x_0$ lies in the closure of the map $\R^\times \to X(\R)$, $t \mapsto \lambda(t) \cdot x$ and $G(\R) \cdot x$ is not closed.

(3) To avoid the smoothness assumption in \cite[Th\'eor\`eme 2.7]{LunaRealGIT}, take a closed $G$-equivariant embedding of $X$ in $\A^n_\R$ and apply the cited result to $\A^n_\R$.

(4)  See \cite[Lemme 2.6]{LunaRealGIT} or \cite[Proposition 6.8]{RichardsonSlodowy}.
\end{proof}

In order to complete the analogy with the complex case we prove the following:

\begin{theorem} \label{Thm:ComplementRealGIT} With the notations introduced above, the following hold:
\begin{enumerate}
\item if $U \subset X(\R)$ is a $G(\R)$-saturated\footnote{That is, for every $x \in U(\R)$ the closure of the orbit $G(\R) \cdot x$ is contained in $U(\R)$.} open subset, then $\pi(U)$ is open in the image $\pi(X(\R))$;
\item if $V \subset Y$ is a Zariski open subset such that the induced map 
$$ \pi \colon U \df \pi^{-1}(V) \too V,$$
is smooth, then $\pi(U(\R))$ is a finite union of connected components of $V(\R)$. 
\end{enumerate}
In particular, if $W \subset V(\R)$ is a connected component, either it is contained in $\pi(X(\R))$ or it does not meet $\pi(X(\R))$.
\end{theorem}

This statement is probably known to experts in GIT over $\R$. Nonetheless we could not find it stated as it is, so we give a proof of it. Before doing this, let us illustrate the theorem through two examples.

\begin{example} Let $X = \A^1_\R$ be the affine line endowed with the action of $G = \{\pm \id\}$ defined by  $ \pm \id \cdot x = \pm x$.
Then the GIT quotient $Y$ is again $\A^1_\R$ and the quotient map is $\pi(x) = x^2$. In particular, we have
$$ \pi(\R) = \R_+,$$
so that $\pi(X(\R))$ cannot be open in $Y(\R)$ in general.

Remark that, if $V = Y \smallsetminus \{ 0\}$, then $V(\R)$ has exactly two connected components, 
$$ W_+ = \R_{> 0}, \quad W_- = \R_{<0},$$
one contained in the image and the other not meeting it.
\end{example}

\begin{example} One may wonder whether in the previous example the non-openness of the image of $X(\R)$ in $Y(\R)$ were due to the fact that $G$ is not connected. This is not true.

Indeed, let $G$ denote the real form of $\SL_2$ whose real points are
$$ G(\R) = \SU(2) = \left\{ 
\begin{pmatrix}
a & - \bar{b} \\ b & \bar{a}
\end{pmatrix} : a, b \in \C, |a|^2 + |b|^2 = 1
\right\}. $$

The GIT quotient of $X = G$ under the action of $G$ by conjugation is the affine $\A^1_\R$ and the quotient map is the trace (\emph{cf.} Example \ref{ex:GITQuotientConjugation}): for $a, b \in \C$ such that $|a|^2 + |b|^2 = 1$,
$$\pi
\begin{pmatrix}
a & - \bar{b} \\ b & \bar{a}
\end{pmatrix} = a + \bar{a} = 2 \Re(a),
$$
hence $\pi(\SU(2)) = [-2, 2]$. The points $\pm 2$ correspond to the points where the trace map is not smooth.
\end{example}

\begin{proof}[{Proof of Theorem \ref{Thm:ComplementRealGIT}}] Let $U \subset X(\R)$ be an open $G(\R)$-saturated subset. The naive argument would go as follows. 

Consider the complement $F = X(\R) \smallsetminus U$ of $U$: it is a closed $G(\R)$-stable subset of $X(\R)$, hence its image in $Y(\R)$ is closed. Nonetheless it may happen that $\pi(F)$ and $\pi(U)$ meet. Indeed, if $T \subset X$ is a closed $G$-orbit, then $T(\R)$ may be made of more than one $G(\R)$-orbit (see Example \ref{Ex:SL2(R)Conjugation} for instance). Then it can occur that one of these $G(\R)$-orbits lies in $U$ and another in the complement $F$. For this reason the argument is more involved.

(1) Let $U \subset X(\R)$ be an open $G(\R)$-saturated subset and let $x \in X(\R)$. One has to show that $\pi(U)$ contains an open neighborhood of $\pi(x)$ in $\pi(X(\R))$.

First of all, one may suppose that the orbit $G(\R)\cdot x$ is closed. If this is not the case, consider the unique closed orbit contained in $\ol{G(\R) \cdot x}$. The latter belongs to $U$ because is $G(\R)$-saturated.

From now on assume that the $G(\R)$-orbit of $x$ is closed. By Luna's Slice \'Etale Theorem \cite{LunaSlicesEtales}, there exists an affine locally closed complex subvariety $Z$ of $X_{\C}$ containing $x$, stable under the action of the stabilizer $G_x$ of $x$, and such that
$$ W = \{ g \cdot z \in X_\C : g \in G, z \in Z\},$$
is a $G_{\C}$-saturated Zariski open subset of $X_\C$. Moreover the induced map between GIT quotients $ \phi \colon V / G_x \to \pi(W)$ is \'etale. Clearly, up to taking $V$ smaller, one may suppose that it is defined over $\R$:  it suffices to consider $V \cap \bar{V}$ where $\bar{V} \subset X_\C$ is the conjugated variety.

It is now possible to conclude the argument. Let $p \colon V \to V/G_x$ be the quotient map. The open subset $U \cap V(\R)$ of $V(\R)$ is $G_x(\R)$-saturated. Remark that it suffices to prove that $p(U \cap V(\R))$ contains an open neighborhood of $p(x)$ in $p(V(\R))$. This is because the map $\phi$ is \'etale, hence an open map at the level of real points \cite[Proposition 2.2.1]{MoretBailly}.

The complement $F = V(\R) \smallsetminus U$ of $U$ in $V(\R)$ is a closed $G_x(\R)$-stable subset, thus its image in $(V/G_x)(\R)$ is closed by Theorem \ref{thm:LunaRealGIT}. It remains to show that $p(x)$ belongs to the open subset 
$$ p(V(\R)) \smallsetminus p(F)$$
of $p(V(\R))$. Suppose the contrary. Then there would exist a point $y \in F$ such that $p(y) = p(x)$, that is, $\{ x\}$ would be the unique closed $G_x(\R)$-orbit contained in the closure of $G_x(\R) \cdot y$. Since $F$ is closed, this would imply that $x$ belongs to $F$, leading to a contradiction.

(2) Let us consider a Zariski open subset $V$ of $Y$ such that the induced map
$$ \pi \colon U \df \pi^{-1}(V) \too V,$$
is smooth. Then, the induced map at the level of real points is open \cite[Proposition 2.2.1]{MoretBailly}. Moreover, $U(\R)$ is a $G(\R)$-saturated open subset by definition and if $F \subset U(\R)$ is a closed $G(\R)$-stable subset then its image in $V(\R)$ is closed (cover $V$ by affine open subsets to reduce to the case of Theorem \ref{thm:LunaRealGIT}). In particular, the image of $U(\R)$ is open and closed in $V(\R)$. 
\end{proof}

As said before for $x \in X(\R)$ there exists a unique closed $G(\R)$-orbit contained in the closure of $G(\R)\cdot x$. Therefore the space of $G(\R)$-orbits does not have that much interest. The notation $X(\R) / G(\R)$ will be reserved for the set of closed $G(\R)$-orbits. Consider the map
$$ p \colon X(\R) \too X(\R) / G(\R), $$
sending a point $x$ to the unique closed $G(\R)$-orbit contained in $\ol{G(\R)\cdot x}$. Endow the set $X(\R) / G(\R) $ with the finest topology such that $p$ is continuous: this is obtained declaring a subset  $V \subset X(\R) / G(\R) $ open if $p^{-1}(V)$ is an open subset of $X(\R)$. Since for a point $x \in X(\R)$ the orbit $G(\R)\cdot x$ is closed if and only if $G(\C)\cdot x$ is closed, one obtains a map
$$ \theta \colon X(\R) / G(\R) \too Y(\R),$$
whose image is contained in $\pi(X(\R))$.

\begin{proposition} \label{Prop:BiggestHausdorffQuotientAffine} With the notations introduced above:
\begin{enumerate}
\item the topological space $X(\R) / G(\R)$ is Hausdorff and locally compact;
\item the map $\theta$ is proper with finite fibers and open onto its image;
\item $X(\R) / G(\R)$ is the categorical quotient of $X(\R)$ by $G(\R)$ in the category of Hausdorff spaces: for all Hausdorff topological space $Z$ and all $G(\R)$-invariant continuous map $f \colon X(\R) \to Z$ there exists a unique continuous map $\tilde{f} \colon X(\R) / G(\R) \to Z$ such that $ f = \tilde{f} \circ p$.
\end{enumerate}
\end{proposition}

\begin{remark}\
\begin{enumerate}
\item Statement (3) characterizes $X(\R) / G(\R)$ up to a unique homeomorphism. The quotient $X(\R) / G(\R)$ is called the \emph{separated quotient} of $X(\R)$ by $G(\R)$. 

\item The Hausdorff topological space $X(\R) / G(\R)$ is more generally the quotient of $X(\R)$ by $G(\R)$ in the category of $\textup{T}_1$ topological spaces, that is, topological spaces whose points are closed (these are also called ``accessible'', ``Fr\'echet'' or ``Tikhonov'').
\end{enumerate}
\end{remark}

\begin{proof} (1) See \cite[Theorem 7.6]{RichardsonSlodowy}\footnote{The space $X(\R) / G(\R)$ being Hausdorff is due to Luna (see \cite[7.3.3]{RichardsonSlodowy}). However the techniques employed by Richardon and Slodowy permit to give a proof of this which does not have recourse to Luna's Slice \'Etale theorem (see \cite[\S 9]{RichardsonSlodowy}).} and \cite[Proposition 7.4]{RichardsonSlodowy}.

(2) Openness of the map $\theta$ is a reformulation of Theorem \ref{Thm:ComplementRealGIT}. If $x \in X(\R)$ has closed orbit $T = G \cdot x$ then, by Proposition \ref{Prop:RealPointsOrbitGaloisCohomology},
$$ \# \theta^{-1}(\pi(x)) = \# T(\R) / G(\R) = \# \ker(\HH^1(\R, G_x) \to \HH^1(\R, G)),$$
where $G_x$ stands for the stabilizer of $x$. In particular $f^{-1}(\pi(x))$ is a finite set.

(3) For $x \in X(\R)$ set $\tilde{f}(p(x)) = f(x)$. The map $\tilde{f}$ is well-defined: if $x'$ lies in the unique closed $G(\R)$-orbit contained in $\overline{G(\R) \cdot x}$ then $f(x) = f(x')$. To show that $\tilde{f}$ is continuous, remark that for an open subset $U \subset Z$,
$$ f^{-1}(U) = p^{-1}(\tilde{f}^{-1}(U)),$$
is an open $G(\R)$-saturated subset. In particular $\tilde{f}^{-1}(U)$ is open.
\end{proof}

\begin{example} \label{Ex:SL2(R)Conjugation} Consider the action of $G = \SL_{2, \R}$ by conjugation on itself. Then the GIT quotient of $X = \SL_{2, \R}$ by $G$ is $\A^1_\R$ and the quotient map is given by the trace (\emph{cf.} Example \ref{ex:GITQuotientConjugation}):
\begin{eqnarray*}
\Tr \colon \SL_{2, \R} & \too & \A^1_\R \\
A & \longmapsto & \Tr A.
\end{eqnarray*}
For a real matrix $A \in \SL_2(\R)$ three cases need to be distinguished:
\begin{itemize}
\item $|\Tr(A)| > 2$ : the matrix $A$ is diagonalisable over $\R$ with distinct eigenvalues. The $G$-orbit $G \cdot A$ is closed and its real points consists of one $G(\R)$-orbit.

\item $|\Tr(A)| = 2$ : the matrix $A$ has one real eigenvalue of multiplicity $2$ but may be not diagonalisable. Anyway the unique closed orbit in the closure of $G \cdot A$ is $\{\id \}$ if $\Tr(A) = 2$ and it is $\{ - \id \}$ if $\Tr(A) = -2$.

\item $|\Tr(A)| < 2$ : the matrix $A$ is a rotation of angle $\theta \not \equiv 0 \pmod{\pi \Z}$. Since elements in $\SL_2(\R)$ respect the orientation of $\R^2$, the matrices $A$ and $A^{-1}$ cannot lie on the same $G(\R)$-orbit. On the other hand it is easily seen that these are the unique two $G(\R)$-orbits.
\end{itemize}

The situation can be depicted as follows:

\begin{figure}[h]
\begin{tikzpicture}
\draw (-3,2)--(-1,2);
\draw (1,2)--(3,2);
\draw (-3,-0.25)--(3,-0.25);
\draw (0,2) circle (1cm);

\draw[dashed] (-1,2) -- (-1,-0.25);
\draw[dashed] (1,2) -- (1,-0.25);
\draw[->] (0,0.75) -- (0,0) node[midway, right] {$\pi$};
\node[anchor=east] at (-1, 2.25) {$-\id$};
\draw[fill] (-1,2) circle (1pt);
\node[anchor=west] at (1, 2.25) {$\id$};
\draw[fill] (1,2) circle (1pt);
\node[anchor=east] at (-3.25, 2) {$X(\R)/G(\R)$};
\node[anchor=east] at (-3.25, -0.25) {$Y(\R)$};
\node[anchor=west] at (4, -0.25) {};
\node[anchor=north] at (-1,-0.25) {-2};
\node[anchor=north] at (1,-0.25) {2};
  \end{tikzpicture}
\end{figure}

\end{example}

\subsection{The complex projective case} Let $G$ be a complex reductive group acting on a projective variety $X$. In order to be led back to the affine case, let $L$ be a $G$-linearized ample line bundle on $X$ (\cite[\S 1.3]{GIT}). A $G$-linearized very ample line bundle is equivalent to the a $G$-equivariant closed embedding $\iota \colon X \to \P(V)$, where $G$ acts linearly on $V$.\footnote{Under this correspondence $L = \iota^\ast \O(1)$ and $V = \HH^0(X, L)^\ast$.}

Even with this supplementary hypothesis, one cannot hope to obtain theorems analogous to the one in the affine case. For instance, letting $\G_{m, \C}$ acting on the projective line $\P^1_{\C}$ by $t[x:y] = [tx : t^{-1} y]$, both $0$ and $\infty$ lie in the closure of the orbit of $1$. 

To avoid such a situation one has to restrict itself to consider just an open subset of $X$.

\begin{definition} A point $x \in X(\C)$ is said to be \emph{semi-stable} with respect to $G$ and $L$ if there exists a positive integer $d \ge 1$ and a $G$-invariant section $s \in \H^0(X, L^{\otimes d})$ not vanishing at $x$. The subset of semi-stable points is denoted $X^\ss$; it is open in $X$.
\end{definition}

Consider the graded $\C$-algebra of finite type
$$ A \df \bigoplus_{d \ge 0} \HH^0(X, L^{\otimes d}), $$
and its sub-$\C$-algebra of $G$-invariant elements,
$$ A^G \df \bigoplus_{d \ge 0} \HH^0(X, L^{\otimes d})^G.$$

\begin{theorem} The graded $\C$-algebra $A^G$ is finitely generated. Consider the projective variety $Y = \Proj(A^G)$. 

Then, the inclusion $A^G \subset A$ induces a morphism $\pi \colon X^\ss \to Y$ satisfying the following properties:
\begin{enumerate}
\item $\pi$ is $G$-invariant and surjective;
\item let $x, x' \in X^\ss(\C)$; then $\pi(x) = \pi(x')$ if and only if
$$ \ol{G \cdot x} \cap \ol{G \cap x'} \cap X^\ss \neq \emptyset;$$
\item if $F \subset X^\ss$ is a Zariski closed $G$-stable subset, then $\pi(F)$ is Zariski closed in $Y$;
\item if $U \subset X$ is a Zariski open $G$-saturated subset, then $\pi(U)$ is Zariski open;
\item the homomorphism of sheaves $\pi^\ast \colon \O_Y \to \pi_\ast \O_{X^\ss}$ induces an isomorphism
$$ \pi^\ast \colon  \O_Y \stackrel{\sim}{\too} (\pi_\ast \O_{X^\ss})^G.$$
\end{enumerate}
\end{theorem}

The projective variety $Y$ is called the \textit{GIT quotient} of $X^\ss$ by $G$ and denoted $X^\ss / G$.\footnote{Again in the literature it can be denoted $X /\!\!/ G$ or $X^\ss /\!\!/ G$.}

\begin{example} Let $\G_{m, \C}$ act on $X = \P^1_{\C}$ by $t[x:y] = [tx : t^{-1} y]$. The set of semi-stable points is $X \smallsetminus \{ 0, \infty\}$, so that the orbit of $1$ is closed in $X^\ss$ and the quotient is a point.
\end{example}

\begin{example} Let $\M_{n, \C}$ be the space of $n \times n$ matrices. Then $\GL_{n, \C}$ acts on $\P(\M_{n}(\C))$ by conjugation and a matrix $A \in \M_{n}(\C)$ is semi-stable if and only if its characteristic polynomial is not $T^n$, that is, $A$ is not nilpotent (\emph{cf.} Example \ref{ex:GITQuotientConjugation}).
\end{example}

\subsection{The real projective case} Let $X$ be a real projective variety and let $G$ be a reductive group acting on $X$. Let $L$ be a $G$-linearized ample line bundle on $X$. We consider the open subset $X^\ss$ (resp. $X^\s$) of $X$ made of semi-stable (resp. stable) points with respect to the action of $G$ and the polarization $L$. Consider the graded $\R$-algebra
$$ A \df \bigoplus_{d \ge 0} \HH^0(X, L^{\otimes d}), $$
and its sub-$\R$-algebra of $G$-invariant elements,
$$ A^G \df \bigoplus_{d \ge 0} \HH^0(X, L^{\otimes d})^G.$$

\begin{theorem} The graded $\R$-algebra $A^G$ is finitely generated. Consider the real projective variety $Y = \Proj(A^G)$. 

Then, the inclusion $A^G \subset A$ induces a morphism $\pi \colon X^\ss \to Y$ satisfying the following properties:
\begin{enumerate}
\item $\pi$ is $G$-invariant and surjective;
\item let $x, x' \in X^\ss(\C)$; then $\pi(x) = \pi(x')$ if and only if
$$ \ol{G(\C) \cdot x} \cap \ol{G(\C) \cap x'} \cap X^\ss(\C) \neq \emptyset;$$
\item if $F \subset X^\ss$ is a Zariski closed $G$-stable subset, then $\pi(F)$ is Zariski closed in $Y$;
\item if $U \subset X$ is a Zariski open $G$-saturated subset, then $\pi(U)$ is Zariski open;
\item $A^G \otimes_\R \C = (A \otimes_\R \C)^{G_\C}$;
\item $V \df \pi(X^\s)$ is Zariski open in $Y$ and $\pi^{-1}(V) = X^\s$.
\end{enumerate}
\end{theorem}

The projective variety $Y$ is called the \textit{GIT quotient} of $X^\ss$ by $G$ and denoted $X^\ss / G$. Remarks analogous to the ones in the affine case (see Remark \ref{rmk:RemarksOnGIT}) apply also in this case.

Turning to the situation on real points, Theorems \ref{thm:LunaRealGIT} and \ref{Thm:ComplementRealGIT} imply the following:

\begin{theorem} \label{Thm:ProjectiveRealGIT} With the notations introduced above, let $x \in X(\R)$. Then the following hold:
\begin{enumerate}
\item the unique closed orbit contained in $\ol{G(\C) \cdot x} \cap X^\ss(\C)$ has a real point. In particular, it is defined over $\R$.
\item the orbit $G(\R) \cdot x$ is closed in $X^\ss(\R)$ if and only if the orbit $G(\C)\cdot x$ is closed in $X^\ss(\C)$;
\item the closure of $G(\R) \cdot x$ in $X^\ss(\R)$ contains a unique $G(\R)$-orbit;
\item if $F \subset X^\ss(\R)$ is a closed $G(\R)$-stable subset, then $\pi(F)$ is closed in $Y(\R)$;
\item if $U \subset X^\ss(\R)$ is an open $G(\R)$-saturated subset, then $\pi(U)$ is open in $\pi(X^\ss(\R))$;
\item if $V \subset Y$ is a Zariski open subset such that the induced map 
$$ \pi \colon U \df \pi^{-1}(V) \too V,$$
is smooth, then $\pi(U(\R))$ is a finite union of connected components of $V(\R)$.
\end{enumerate}
\end{theorem}

As for the affine case, because of statement (3) the set of $G(\R)$-orbits in $X^\ss(\R)$ is not very interesting. The notation $X(\R) / G(\R)$ is therefore reserved for the set of closed $G(\R)$-orbits in $X^\ss(\R)$ endowed with the strongest topology such that 
the natural map 
$$p \colon X^\ss(\R) \too X^\ss(\R) /G(\R) $$
is continuous. Similarly to Proposition \ref{Prop:BiggestHausdorffQuotientAffine}, an immediate consequence of Theorem \ref{Thm:ProjectiveRealGIT} is the following:

\begin{proposition} \label{Prop:ProjectionOfMaxSeparatedQuotientToGITQuotientProj} With the notations introduced above:
\begin{enumerate}
\item the topological space $X^\ss(\R) / G(\R)$ is Hausdorff and locally compact;
\item the map $\theta$ is proper with finite fibers and open onto its image;
\item $X^\ss(\R) / G(\R)$ is the categorical quotient of $X^\ss(\R)$ by $G(\R)$ in the category of Hausdorff spaces: for all Hausdorff topological space $Z$ and all $G(\R)$-invariant continuous map $f \colon X^\ss(\R) \to Z$ there exists a unique continuous map $\tilde{f} \colon X^\ss(\R) / G(\R) \to Z$ such that  $ f = \tilde{f} \circ p$.
\end{enumerate}
\end{proposition}

The topological space $X^\ss(\R) / G(\R)$ is determined (up to a unique homeomorphism) by statement (3) and is called the \emph{separated quotient} of $X^\ss(\R)$ by $G(\R)$.

\subsection{Configurations of real flags}\label{subsection:configurations} Let $G$ be a real reductive group and $G_\sp$ its split form. Let us identify their complexification: $G_\C = G_{\sp, \C}$. Let $P_\sp \subset G_\sp$ be a parabolic subgroup and $F_\sp = G_\sp / P_\sp$. 

\begin{definition}
We say that \emph{the type of the parabolic $P_\sp$ is defined over $\R$ for $G$} if there exists a parabolic subgroup $P$ of $G$ such that $P_\C$ is conjugated to $P_{\sp, \C}$. 
\end{definition}

The group $G_0 = G(\R)$ acts on the flag variety $F_\sp(\C)$.

\begin{theorem}[{\cite[Theorems 3.3 and 3.5]{W}}] With the notations introduced above,
\begin{enumerate}
\item there exists a unique closed $G_0$-orbit $T$ in $F_\sp(\C)$;
\item for $x \in X(\C)$, the orbit $T$ is contained in the closure of $G_0 x$ and
$$ \dim_\R G_0x \ge \dim_\R T \ge \dim_\C F_{\sp}(\C),$$
with equality if and only the type of the parabolic $P_\sp$ is defined over $\R$ for $G$. In this case $T$ is the set of real points of the associated real flag variety.
\end{enumerate}
\end{theorem}

The orbit $T$ has the following algebraic nature.  Let $W = \Res_{\C/\R}(F_{\sp, \C})$ the $\R$-algebraic variety obtained by Weil restriction of $F_{\sp, \C}$. Concretely, it is the complex variety $F_{\sp, \C} \times F_{\sp, \C}$ together with the anti-holomorphic involution
$$ (x_1, x_2) \longmapsto (\sigma(x_2), \sigma(x_1)),$$
where $\sigma$ is the anti-holomorphic involution on $F_{\sp, \C}$ given by $F_{\sp}$. 

The group $G$ acts naturally on $W$: if $\theta$ denotes the automorphism of $G_{\sp, \C}$ defining the form $G$, the action of $G$ corresponds to the action
$$ g(x_1, x_2) = (g x_1, \theta(g)x_2)$$
for $g \in G_{\sp}(\C)$ and $x_1, x_2 \in F_{\sp}(\C)$. 

For a point $t \in T$ let $S \subset W$ be the $G$-orbit of $x$. Let $F$ be the Zariski closure of $S$ in $W$: it is a projective variety.

\begin{lemma} With the notations introduced above:
\begin{enumerate}
\item $F(\R) = T$;
\item $S$ is projective if and only if the stabilizer of $t$ is a parabolic subgroup of $G$.
\end{enumerate}
\end{lemma}

\begin{proof} (1) The closure $F$ of the orbit $S$ in $W$ is made of $S$ and lower dimensional orbits \cite[\S 1.8]{Borel}. Therefore if $x \in F(\R)$ does not lie in $T$, then its $G(\R)$-orbit has dimension strictly smaller than $T$, contradicting the minimality of $T$.

(2) Clear by definition.
\end{proof}

Let $L_{\sp}$ be a $G_{\sp}$-equivariant line bundle on $F_{\sp}$. The line bundle  $$ \pr_1^\ast L_{\sp} \otimes \pr_2^\ast L_{\sp}$$ on $F_{\sp, \C} \times F_{\sp, \C}$ descends into a $G$-equivariant line bundle $L$ on $W$, and by restriction, on $F$.

Let $n \ge 1$ be an integer, $X_{\sp} = F_{\sp}^n$ and $X = F^n$. Let $X^\ss$ be the set of semi-stable points of $X$ with respect to $G$ and the polarization induced by $L$ on every copy of $F$. Similarly, let $X_{0}^\ss$ be the set of semi-stable points of $X_{\sp}$ with respect to $G_{\sp}$ and the polarization induced by $L_{\sp}$ on every copy of $F_{\sp}$.

\begin{proposition} With the notations introduced above:
\begin{enumerate}
\item A point $x \in T^n$ is semi-stable if and only if it is so for every one-parameter subgroup $\lambda \colon \mathbb{G}_{m, \R} \to G$ defined over $\R$. 
\item The quotient $T^{n, \ss}/G(\R)$ is compact.
\item The open subset $T^n \cap X_{\sp}(\C)^{n, \ss}$ of $T^{n, \ss}$ is $G(\R)$-saturated.
\item The inclusion $T \subset X_{\sp}(\C)$ induces a continuous map
$$ (T^n \cap X_{\sp}(\C)^{n, \ss}) / G(\R) \too X_{\sp}(\C)^\ss / G(\C)$$
and $\dim_\R T^{n, \ss}/G(\R) \ge \dim_\C X_{\sp}(\C)^\ss / G(\C)$. 
\end{enumerate}
\end{proposition}

\begin{proof} Clear.
\end{proof}

\begin{example} Let $V = \R^2$ and $h$ be the standard hermitian form on $\C^2$. Let $G_{\sp} = \SL_{2, \R}$ and $G$ be the real form of $\SL_{2, \C}$ whose points are $\SU(2)$.

Let $F_{\sp} = \P^1_{\R}$. The group $\SU(2)$ acts transitively on $\P^1(\C)$ thus, with the notations above, $F = \Res_{\C / \R}(\P^1_\C)$. There is no one-parameter subgroup of $G$ defined over $\R$, thus all points of $F(\R)^n = \P^1(\C)^n$ are semi-stable. Because of this, the superscript $\ss$ is reserved for semi-stable points of $\P^1(\C)^n$ with respect to $\SL_{2}(\C)$. The separated quotient $\P^1(\C)^n / \SU(2)$ is the quotient in the usual sense and has real dimension $2n - 3$ for $n \ge 2$ and it is a point for $n = 1$.

On the other hand, a $n$-tuple $(x_1, \dots, x_n)$ of points of $X_{\sp}(\C) = \P^1(\C)$ is semi-stable under the action of $\SL_{2, \C}$ if and only if none of the points is repeated more than $\tfrac{n}{2}$ times. The quotient $(\P^1)^{n, \ss} / \SL_{2, \C}$ has complex dimension $n - 3$ for $n \ge 3$, is a point for $n = 2$ and  is empty for $n = 1$.

\bigskip

\begin{center}
\begin{tabular}{c|c|c|c}
$n$ & $\P^1(\C)^n / \SU(2)$ & $\P^1(\C)^{n, \ss} / \SU(2)$ & $\P^1(\C)^{n, \ss} / \SL_2(\C)$ \\
\hline
 1 & point & $\emptyset$ & $\emptyset$ \\
 2 & $[0, 1]$ & $\left[0, 1 \right[$ & point \\
 3 & $\P^1(\C)^3 / \SU(2)$ &$ \{ \textup{hermitian forms on } \C^2\} / \R_{>0}$ & point
\end{tabular}
\end{center}

\bigskip
\end{example}

\begin{example} Let $G$ be a reductive group over $\R$ and $G_{\sp}$ its split form. Let $F_{\sp}, F_{\sp}'$ be flag varieties of $G_{\sp}$ with a $G_{\sp}$-equivariant surjection $\pi_{\sp} \colon F_{\sp} \to F_{\sp}'$. Let $T, T'$ be respectively the unique $G(\R)$-closed orbit contained in $F_{\sp}(\C)$ and $F_{\sp}'(\C)$. 

Clearly $\pi_{\sp}(T) = T'$. The $G_{\sp}$-equivariant projection $\pi_{\sp}$ induces a $G$-equivariant surjection
$$ \pi \colon W := \Res_{\C/\R}(F_{\sp, \C}) \too W' := \Res_{\C/\R}(F'_{\sp, \C}).$$
Denote by $F$ the Zariski closure of $T$ in $W$ and similarly $F'$ for that of $T'$. The map $\pi$ restricts to a projection $\pi \colon F \to F'$.

Let $L_{\sp}, L_{\sp}'$ be $G_{\sp}$-equivariant ample line bundles respectively on $F_{\sp}, F_{\sp}'$ and $L, L'$ the $G$-equivariant ample line bundles induced on $F, F'$. Let $n \ge 1$ be an integer and $X = F^n$, $X' = F'^n$. As in the general construction, consider semi-stable points with respect the polarizations induced by $L, L'$. This induces a continuous map between quotients
$$ T^{n, \ss} \cap \pi^{-1}(T'^{n, \ss}) / G(\R) \too T'^{n, \ss} / G(\R).$$

We will take profit of this construction in two situations. The first one is:
\begin{itemize}
\item $G_{\sp} = \SL_{n, \R}$;
\item $G(\R) = \SU(n-1, 1)$;
\item $F_{\sp} = \Flag(1, 2, \dots, n)$ the flag variety of complete flags;
\item $F_{\sp}' = \Flag(1, n-1)$ the variety of flags $(L, H)$ made of a line and hyperplane containing it.
\end{itemize}
The second one is:
\begin{itemize}
\item $G_{\sp} = \SL_{4, \R}$;
\item $G(\R) = \SU(2, 2)$ or $\SL_2(\H)$;
\item $F_{\sp} = \Flag(1, 2, 3)$ the flag variety of complete flags;
\item $F_{\sp}' = \Gr(2, 4)$ the grassmannian of planes in dimension $4$.
\end{itemize}
In both cases the stabilizer of a point $t \in T'$ is a parabolic subgroup of $G$.
\end{example}

\section{Sequences of flags line-hyperplane}\label{section:Sequences}

\subsection{Setup} Let $n, r \ge 3$ be integers and consider the variety $\Fl_{(1,n-1)}(\C^n)$ of couples $(L, H)$ of subspaces of $\C^n$ where $L$ is line and $H$ is a hyperplane containing it. Set:
\begin{align*}
V &= \C^n,
&G& = \SL_{n, \C}, 
&F& = \Fl_{(1, n-1)}(V),
&X& = F^r.
\end{align*}
Let $ \O(1, 1)$ be the line bundle $\pr_1^\ast \O(1) \otimes \pr_2^\ast \O(1)$ on the variety $\P(V) \times \P(V^\vee)$ (here $\pr_1, \pr_2$ are be respectively the projection onto the first and second factor).

The semi-simple group $G$ acts on $X$. Consider the linearization of this action given by the embedding
$$ \iota \colon X= F^r \too (\P(V) \times \P(V^\vee))^r.$$
This linearization corresponds to the $G$-linearized line bundle
$$ L = \iota^\ast (\pr_1^\ast \O(1,1) \otimes \cdots \otimes \pr_r^\ast \O(1, 1)),$$
where $\pr_i \colon X \to F$ denotes the projection onto the $i$-th factor.

\subsection{Invariants} We introduce some natural invariants for this action. For a permutation $\sigma \in \mathfrak{S}_r$ consider the linear form $ s_\sigma \colon \textstyle (V \otimes V^\vee)^{\otimes r} \to \C$ defined, for $v_1, \dots , v_r \in V$ and $\phi_1, \dots , \phi_r \in V^\vee$, by
$$
s_{\sigma}(v_1 \otimes \phi_1, \dots , v_r \otimes \phi_r) := \prod_{i = 1}^r \phi_{i}(v_{\sigma(i)}).
$$
The linear form $s_\sigma$ is $G$-invariant. The restriction to $X$ of the induced global section of $\bigotimes_{i = 1}^r \pr_i^\ast \O(1, 1)$ is a $G$-invariant global section of $L$:
$$ s_{\sigma} \in \HH^0(X, L)^{G}.$$
If the permutation $\sigma$ has a fixed point then $s_{\sigma}$ vanishes identically on $X$. Therefore it is interesting to consider the section only for \emph{derangements}, that is, fixed-points-free permutations. Denote by $\mathfrak{D}_r$ the set of derangements and by $!r$ its cardinality. One has the recurrence relation
$$ !r = (r-1)(!(r-1) + !(r-2)).$$ 
The global section $s_{\sigma}$ vanishes at a $r$-tuple $\{ (L_i, H_i) \}_{i = 1, \dots, r}$ if and only if there exists $i = 1, \dots, r$ such that $L_i$ is contained in $H_{\sigma(i)}$.

\subsection{The quotient} 

\begin{definition} Let $x = \{ (L_i, H_i) \}_{i = 1, \dots, r}$ be a point of $X$. 

For $n \ge r$, the point $x$ is said to be in \emph{general position} if the following conditions are satisfied:
\begin{enumerate}
\item $\dim L_1 + \cdots + L_r = r$;
\item  $\dim H_1 \cap \cdots \cap H_r = n- r$;
\item $L_i \cap H_j = 0$ for all $i \neq j$;
\item $ (L_1 + \cdots + L_r) \cap (H_1 \cap \cdots \cap H_r) = 0$.
\end{enumerate}

Suppose $n < r$. The point $x$ is said to be in \emph{general position} if, for every subset $I \subset \{ 1, \dots, n\}$ of $n$ elements, the $n$-tuple $\{ (L_i, H_i)\}_{i \in I}$ is in general position in the preceding sense.
\end{definition}

Let $U$ be the Zariski open subset of $X^\ss$ made of $r$-tuples in general position.

\begin{theorem} \label{Thm:Quotient(1,3)flags} With the notations introduced above:
\begin{enumerate}
\item The graded $\C$-algebra of $G$-invariant elements
$$ \bigoplus_{d \ge 0} \HH^0(X, L^{\otimes d})^G,$$
is generated by the sections $s_{\sigma}$ with $\sigma \in \mathfrak{D}_r$.
\item A $r$-tuple $\{ (L_i, H_i) \}_{i = 1, \dots, r}$ is semi-stable if and only if there exists a derangement $\sigma \in \mathfrak{S}_r$ such that $L_i$ is not contained in $H_{\sigma(i)}$ for all $i = 1, \dots, r$.
\item Let $\sigma_1, \dots, \sigma_{!r}$ be the derangements on $r$ elements and let $s_i = s_{\sigma_i}$ for $i =1, \dots, !r$. The map $\pi \colon X^\ss \to \P^{!r -1}_\C$,
$$ \pi(x) = [s_{1}(x) : \cdots : s_{!r}(x)]$$
is $G$-invariant and induces a closed embedding $X^\ss / G \to \P^{!r - 1}_\C$.
\item Let $x, x' \in U$ be such that $\pi(x) = \pi(x')$. Then, $x, x'$ lie in the same orbit. In particular, $U$ is a $G$-saturated open subset.

\item An element $ g \in G$ belongs to the stabilizer of a $r$-tuple $x = \{ (L_i, H_i)\}_{i = 1, \dots, r} $ in $U$ if and only if
\begin{enumerate}
\item $g_{\rvert L_1 + \cdots L_r} = t \cdot \id_{L_1 + \cdots + L_r}$ for some $t \in \C^\times$;
\item $H_1 \cap \cdots \cap H_r$ is stable under $g$.
\end{enumerate}
In particular the stabilizer of $x$ has dimension $\max \{ 0, n - r \}^2$.
\item The quotient $Y$ of $X^\ss$ by $G$ has dimension
$$ r^2 - 3r +1 - \min \{ 0, n - r\}^2.$$
\end{enumerate}
\end{theorem}

\begin{proof} (1) This is a consequence of the First Main Theorem of Invariant Theory. It states that, given a positive integer $N \ge 1$, the subspace of $\SL(V)$-invariants of $\End(V^{\otimes N})$ is generated by endomorphisms of the form
$$ \phi_\sigma \colon v_1 \otimes \cdots \otimes v_N \longmapsto v_{\sigma(1)} \otimes \cdots \otimes v_{\sigma(N)},$$
where $\sigma \in \mathfrak{S}_N$ and $v_1, \dots, v_N \in V$.  Let $d \ge 1$ be a positive integer. Since the restriction map
$$ \Sym^d \End(V^{\otimes r}) = \HH^0((\P(V) \times \P(V^\vee))^r, \textstyle \bigotimes_{i = 1}^r \pr_i^\ast \O(1,1)) \too \HH^0(X, L^{\otimes d}),$$
is surjective, the induced map 
$$ (\Sym^d \End(V^{\otimes r}))^G \too \HH^0(X, L^{\otimes d})^G$$
is surjective\footnote{This follows from the fact that the representations of $G$ are completely reducible.}. Identifying the $\Sym^d \End(V^{\otimes r})$ with the subspace of symmetric tensors of
$$ (V \otimes V^\vee)^{\otimes rd} = \End(V^{\otimes rd}),$$
the First Main Theorem of Invariant Theory implies that $\HH^0(X, L^{\otimes d})^G$ is generated by the restriction to elements of the sections $\phi_\sigma$.  Let $\sigma \in \mathfrak{S}_{rd}$ be a permutation. By evaluating $\phi_\sigma$ at a tensor of the form
$$ x= (v_1 \otimes \phi_1 \otimes \cdots \otimes v_r \otimes \phi_r)^{\otimes d}$$
for vectors $v_1, \dots, v_r \in V$ and linear forms $\phi_1, \dots, \phi_r \in V^\vee$, one obtains
$$ \phi_\sigma(x) = \prod_{i = 1}^r \prod_{j = 1}^d \phi_i(v_{\sigma((i-1)d + j)}).$$
For $i, j = 1, \dots, r$ let $a_{ij}$ be the number of times that the term $\phi_i(v_j)$ occurs in the preceding product. If $a_{ii} \neq 0$ for some $i$ then the section $\phi_\sigma$ vanishes identically on $X$. Suppose $a_{ii} = 0$ for all $i = 1, \dots, r$. 

The vector $v_i$ as well as the linear form $\phi_i$ have to appear exactly $d$ times in the product. That is,
\begin{align*}
\sum_{k = 1}^r a_{ik} &= d, &\sum_{k = 1}^r a_{kj} &= d,
\end{align*}
for all $i, j = 1, \dots, r$. Consider the matrix $A = (a_{ij})_{i,j = 1, \dots, r}$. It remains to show that $A$ is sum of permutation matrices associated to derangements.\footnote{One could argue by remarking that $\frac{1}{d} A$ is a doubly stochastic matrix. The Birkhoff-von Neumann theorem therefore states that $\frac{1}{d} A$ is a convex combination of permutation matrices. It follows that $A$ can be written as a sum
$$A =  \sum_{\sigma \in \mathfrak{S}_r} \lambda_\sigma A_\sigma,$$
where, for a permutation $\sigma \in \mathfrak{S}_r$, $A_\sigma$ is the associated permutation matrix and $\lambda_\sigma$ is a non-negative real number. Nonetheless it is not clear that the coefficients $\lambda_\sigma$ can be taken as integers (the decomposition is not unique).}
 It suffices to show that there is a permutation $\sigma \in \mathfrak{S}_r$ such that $a_{i \sigma(i)} \neq 0$ for all $i = 1, \dots, r$. Indeed, if such a permutation $\sigma$ exists, then it is fixed-point free because of the condition $a_{ii} = 0$ for all $i = 1, \dots, r$.

The fact that such a permutation exists is a consequence of Hall's Marriage Theorem \cite{HalmosMarriage}. Let $I = J = \{ 1, \dots, r\} $ and consider the bipartite graph $\Gamma$ on $I \sqcup J$ defined by the relation: for $i \in I$ and $j \in J$,  let $i \sim j$ if $a_{ij} \neq 0$. For a subset $S \subset I$ the neighbourhood $N_\Gamma(S)$ of $S$ in $\Gamma$ is the set of vertices in $Y$ adjacent to some element of $S$:
$$ N_\Gamma(S) = \{ j \in J : i \sim j\}. $$
Hall's Theorem states that there is a perfect matching for the graph (that is, in this case, a bijection $\sigma \in \mathfrak{S}_n$ such that $a_{i \sigma(i)} \neq 0$ for all $i= 1, \dots, r$) if and only if
$$\# S \le  \# N_\Gamma(S),$$
for all subsets $S$ of $I$.  The preceding condition is satisfied in this case. Indeed, let $S$ be a subset of $I$. Then,
$$
\# S = \frac{1}{d} \sum_{i \in S} \sum_{j \in N_\Gamma(S)} a_{ij} \le \frac{1}{d} \sum_{i \in I} \sum_{j \in N_\Gamma(S)} a_{ij}  = \# N_\Gamma(S). 
$$

\smallskip

(2) and (3) follow from (1).

\smallskip

(4) First of all remark that it suffices to find $g \in \GL_n(\C)$ such that $g x = x'$. Write $x = \{ (L_i, H_i)\}_{i = 1, \dots, r}$ and $x ' = \{ (L'_i, H_i')\}_{i = 1, \dots, r}$. For $i = 1, \dots, r$ let $v_i, v'_i$ be respectively a basis of $L_i, L_i'$ and $\phi_i, \phi_i'$ linear forms defining respectively $H_i, H_i'$.

\begin{lemma} \label{lem:BiratioImpliesConjugation} Let $t_1, \dots, t_r, u_1, \dots, u_r\in \C$. With the notations introduced above, there exists $g \in \GL_n(\C)$ such that, for all $i = 1, \dots, r$,
\begin{align*}
gv_i = t_i v_i, && \phi_i \circ g^{-1} = u_i \phi'_i,
\end{align*}
if and only if, for all $\alpha, \beta, i, j = 1, \dots, r$ with $\alpha, \beta \not \in \{ i, j\}$,
\begin{equation} \label{eq:BiratioCondition}
\frac{\phi_i(v_\alpha) \phi_j(v_\beta)}{\phi_j(v_\alpha) \phi_i(v_\beta)} = \frac{\phi'_i(v'_\alpha) \phi'_j(v'_\beta)}{\phi'_j(v'_\alpha) \phi'_i(v'_\beta)},
\end{equation}
and 
\begin{align*}
t_\alpha^{-1} t_\beta  = \frac{\phi'_i(v'_\alpha)}{\phi_i(v_\alpha)} \frac{\phi_i(v_\beta)}{\phi'_i(v'_\beta)}, &&
 t_\alpha u_i = \frac{\phi_i(v_\alpha)}{\phi'_i(v_\alpha)}.
 \end{align*}

\end{lemma}

\begin{proof}[Proof of the Lemma] ($\Rightarrow$) Suppose that such an element $g \in \GL_n(\C)$ exists. Then, for $ \alpha,i = 1, \dots, r$,
$$ u_i \phi'_i(v_\alpha) = \phi_i(g^{-1} v_\alpha') = t_\alpha^{-1} \phi_i(v_\alpha).$$
In particular, for $\alpha \neq i$,
$$ u_i = t_\alpha^{-1} \frac{\phi_i(v_\alpha)}{\phi'_i(v'_\alpha)},$$
is independent of $\alpha$. That is, for all $\alpha, \beta = 1, \dots, r$ different from $i$,
$$  t_\alpha^{-1} \frac{\phi_i(v_\alpha)}{\phi'_i(v'_\alpha)} =  t_\beta^{-1} \frac{\phi_i(v_\beta)}{\phi'_i(v'_\beta)}.$$
In particular, the quantity $t_\alpha^{-1} t_\beta$ is independent of $i$, whence \eqref{eq:BiratioCondition}.

($\Leftarrow$) Set $t_1 = 1$ and, for $\alpha = 2, \dots, r$,
$$ t_{\alpha} = \frac{\phi_i(v_\alpha) \phi'_i(v_1')}{\phi_i'(v_\alpha') \phi_i(v_1)},$$
where $i = 1, \dots, r$ is different from $1$ and $\alpha$. Condition \eqref{eq:BiratioCondition} implies that this does not depend on the chosen $i$. 

Suppose first $n \ge r$. Pick $g \in \GL_n(\C)$ such that 
\begin{align*}
g v_i = t_i v'_i, \textup{ for all } i = 1, \dots, r, && g(H_1 \cap \cdots \cap H_r) = H'_1 \cap \cdots \cap H'_r.
\end{align*}
It remains to show $[\phi_i \circ g^{-1}] = [\phi_i']$ for all $i = 1, \dots, r$. Write $v \in V$ as 
$$ v = \lambda_1 v_1' + \cdots + \lambda_r v_r' + w,$$ with $\lambda_1, \dots, \lambda_r \in \C$ and $w \in H_1' \cap \cdots \cap H_r'$. Then,
$$ \phi_i(g^{-1}v) = \sum_{\alpha = 1}^r \lambda_\alpha \phi_i(g^{-1} v_\alpha') = \sum_{\alpha = 1}^r \lambda_\alpha t_\alpha^{-1} \phi_i(v_\alpha),$$
because $g^{-1}w$ belongs to $H_1 \cap \cdots \cap H_r$. By construction, for all $\alpha, \beta = 1, \dots, r$ different from $i$,
$$ t_\alpha^{-1} \frac{\phi_i(v_\alpha)}{\phi'_i(v'_\alpha)} = t_\beta^{-1} \frac{\phi_i(v_\beta)}{\phi'_i(v'_\beta)}.$$
Call $\mu$ such a complex number. Then
$$ \sum_{\alpha = 1}^r \lambda_\alpha t_\alpha^{-1} \phi_i(v_\alpha) = \mu \sum_{\alpha = 1}^r \lambda_\alpha \phi_i'(v'_\alpha) + \mu \phi_i'(w) = \mu \phi'_i(v),$$
since $\phi_i'(w) = 0$ by assumption.

Suppose $n < r$. Applying the preceding case to the first $n$ couples $(L_i, H_i)$, there exists $g \in \GL_n(\C)$ such that, for all $i = 1, \dots, n$,
\begin{align*}
gv_i = t_i v_i, && \phi_i \circ g^{-1} = u_i \phi'_i.
\end{align*}
It remains to show the preceding property for $i = n+1, \dots, r$.  In order to verify the equality $\phi_i \circ g^{-1} = u_i \phi'_i$ it suffices to show
$$\phi_i (g^{-1} v'_\alpha) = u_i \phi'_i(v'_\alpha),$$
for $\alpha = 1, \dots, n$ as $v'_1, \dots, v'_n$ is a basis of $V$. On the other hand,
$$\phi_i (g^{-1} v'_\alpha) = t_\alpha^{-1} \phi_i (v_\alpha),$$
so that the preceding condition is equivalent to
$$ t_\alpha^{-1} \phi_i (v_\alpha) = u_i \phi'_i(v'_\alpha),$$
which is satisfied by hypothesis. 

 The argument to show $g v_i = t_i v_i$ for $i > n$ is similar.
\end{proof}

Suppose $r \ge 4$ and remark the following:

\begin{lemma} \label{lemma:ExistenceOfDerangements} Suppose $r \ge 4$ and let $i, j, \alpha, \beta = 1, \dots, r$ be pairwise distinct. Then there exists derangements $\sigma, \tau \in \mathfrak{D}_r$ such that $\sigma(k) = \tau(k)$ for all $k \neq i, j$ and 
\begin{align*}
\sigma(i) = \alpha,  && \sigma(j) = \beta, &&\tau(i) = \beta,   && \tau(j) = \alpha.
\end{align*}
\end{lemma}

Let $i, j, \alpha, \beta = 1, \dots, n$ be pairwise distinct and $\sigma, \tau$ derangements as in Lemma \ref{lemma:ExistenceOfDerangements}. If $s_\sigma, s_\tau$ denote the corresponding $G$-invariant sections, then
$$ \frac{s_\sigma(x)}{s_\tau(x)} = \frac{\phi_i(v_\alpha) \phi_j(v_\beta)}{\phi_j(v_\alpha) \phi_i(v_\beta)},$$
and similarly for $x'$. The condition $\pi(x) = \pi(x')$ implies $s_\sigma(x) / s_\tau(x) = s_\sigma(x')/ s_\tau(x')$ and one concludes thanks to Lemma \ref{lem:BiratioImpliesConjugation}. 

\smallskip

Suppose $r = 3$. Set
\begin{align*}
t_1 = 1,
&&t_2=\frac{\phi'_{3}(v'_{1}) \phi_{3}(v_{2})}{\phi_{3}(v_{1}) \phi'_{3}(v'_{2})},
&&t_3=\frac{\phi'_{2}(v'_{1}) \phi_{2}(v_{3})}{\phi_{2}(v_{1}) \phi'_{2}(v'_{3})}.
\end{align*}
Let $g \in \GL_n(\C)$ be such that $g v_i = t_i v_i$ for all $i = 1, 2, 3$. If $n > 3$ suppose moreover
$$ g(H_1 \cap H_2 \cap H_3) = H'_1 \cap H'_2 \cap H'_3.$$
The rest of the proof goes similarly to Lemma \ref{lem:BiratioImpliesConjugation}. Indeed, by definition,
\begin{align*}
t_2 t_1^{-1} = t_2 =\frac{\phi'_{3}(v'_{1}) \phi_{3}(v_{2})}{\phi_{3}(v_{1}) \phi'_{3}(v'_{2})},
&& t_3 t_1^{-1} = t_3 =\frac{\phi'_{2}(v'_{1}) \phi_{2}(v_{3})}{\phi_{2}(v_{1}) \phi'_{2}(v'_{3})},
\end{align*}
and $\pi(x) = \pi(x')$ translates into
$$ t_3 t_2^{-1} = \frac{\phi_1'(v_2') \phi_1(v_3)}{\phi_1(v_2) \phi_1'(v_3')}.$$

(5) For all $i = 1, \dots, r$ let $v_i \in V$ be a generator of the line $L_i$ and let $\phi_i \in V^\vee$ be a linear form defining $H_i$.

Suppose $g \in G_x$. Then, for all $i = 1, \dots, r$, there exists $t_i, u_i \in \C^\times$ such that $g(v_i) = t_i v_i$ and $\phi_i \circ g^{-1} = u_i \phi_i$. For all $i, j = 1, \dots, n$,
$$ \phi_i (v_j) = \phi_i(g^{-1} g(v_j)) = u_i \phi_i (t_j v_j) = u_i t_j \phi_i(v_j).$$
Suppose $i \neq j$. By assumption $\phi_i(v_j)$ does not vanish, thus $u_i t_j = 1$. In particular $ t_1 = \cdots = t_r = u_1^{-1} = \cdots = u_r^{-1}$. 

Conversely, suppose that $g$ satisfies conditions (a) and (b) in the statement. Suppose $n \ge r$. Since the point $x$ is in general position, a vector $v \in V$ can be written in a unique way as $ v = \lambda_1 v_1 + \cdots + \lambda_r v_r + w$, with $\lambda_1, \dots, \lambda_r \in \C$ and $w \in H_1 \cap \cdots \cap H_r$. Thus, for $i = 1, \dots, r$,
\begin{align*}
\phi_i (g^{-1}v) &= \phi_i (t^{-1} (\lambda_1 v_1 + \cdots + \lambda_r v_r) + g^{-1}w) \\
&= t^{-1} \phi_i(\lambda_1 v_1 + \cdots + \lambda_r v_r) \\
&= t^{-1} [ \phi_i(\lambda_1 v_1 + \cdots + \lambda_r v_r) + \phi_i(w)] \\
&= t^{-1} \phi_i(v),
\end{align*}
since $w$ and $g^{-1}w$ belong to $H_1 \cap \cdots \cap H_r$. In particular $g$ stabilizes $H_i$. The proof for $n <r$ goes similarly.

\smallskip

(6) Apply the formula $ \dim Y = \dim X - (\dim G - \dim G_x)$
for a point $x \in U$.  \qedhere
\end{proof}

\subsection{Triplets of flags} \label{par:TripletsLineHyperplane} Keep the notations introduced previously. In this section suppose $r = 3$ and $n \ge 3$. The situation in this is quite simple since there are two derangements in $\mathfrak{S}_3$: the $2$-cycles $(123)$ and $(132)$. Write $s_{123}$ and $s_{132}$ the corresponding invariant sections. The quotient $X^\ss/G$ is isomorphic to $\P^1$ and the map $\pi \colon X^\ss \to \P^1_\C$, 
$$ \pi(x) = [s_{123}(x) : s_{132}(x)],$$
is the quotient map.

The map $\pi \colon X^\ss \to \P^1$ has $3$ fibers that contain more than one orbit: those at $0 = [1:0]$, $\infty = [0:1]$ and $-1 = [1:-1]$. 

The fibers at $0$ and $\infty$ of $\pi$ correspond to points in $X^\ss$ where exactly one invariant vanishes. Up to action of $\mathfrak{S}_3$ on $X^\ss$ by permuting factors, it suffices to study the fiber at $\infty$.

\begin{definition} A semi-stable point $x = \{ (L_i, H_i) \}_{i = 1, 2, 3}$ such that $\pi(x) = \infty$ is said \emph{maximally degenerate} if 
$L_1 \subset H_2$, $L_2 \subset H_3$, $L_3 \subset H_1$.
\end{definition}

For such a maximally degenerate semi-stable point $x = \{ (L_i, H_i) \}_{i = 1, 2, 3}$ point there exists a basis $v_1, \dots, v_n$ of $\C^n$ such that
\begin{align} \label{eq:FormOfMaximallyDegeneratePoint}
L_1 &= \langle v_1 \rangle, & L_2 &= \langle v_2 \rangle, &L_3 &= \langle v_3 \rangle, \\
H_1 &: v_2^\vee = 0, & H_2 &: v_3^\vee = 0,  & H_3 &: v_1^\vee = 0, \nonumber
\end{align}
where $v_1^\vee, \dots, v_n^\vee \in V^\vee$ is the dual basis. To prove this, it suffices to take $v_i$ to be a basis of $L_i$ for $i =1, 2,3$ and $v_4, \dots, v_n$ to be a basis of $H_1 \cap H_2 \cap H_3$.

\begin{proposition} \label{Prop:QuotientLineHyperplaneFiberInfinity} With the notations introduced above:
\begin{enumerate}
\item An element $g \in G$ stabilizes a maximally degenerate point $x = \{ (L_i, H_i)\}_{i = 1, 2, 3} $ if and only if
\begin{enumerate}
\item for $i = 1, 2, 3$, $g_{L_i} = t_i \cdot \id_{L_i}$ for some $t_i \in \C^\times$;
\item $H_1 \cap H_2 \cap H_3$ is stable under $g$.
\end{enumerate}
In particular the stabilizer of $x$ has dimension $(n-3)^2 +2$.
\item The unique closed orbit contained in $\pi^{-1}(\infty)$ is the one given by maximally degenerate points. 
\end{enumerate}
\end{proposition}

\begin{proof} (1) Clear according to \eqref{eq:FormOfMaximallyDegeneratePoint}.

(2) It suffices to show that the orbit of $x$ is of minimal dimension among the orbits contained in $\pi^{-1}(\infty)$.  Let $x = \{ (L_i, H_i) \}_{i = 1, 2, 3}$ be a semi-stable point of $X$ such that $\pi(x) = \infty$ and which is not maximally degenerate. One has to show that the stabilizer of $x$ has dimension $\le (n-3)^2 + 1$. For $i = 1, 2, 3$ let $P_i$ be the stabilizer of the flag $(L_i, H_i)$.

Up to permutations, one may suppose that the flags $(L_1, H_1)$ and $(L_2, H_2)$ are opposite and $L_3 \subset H_1$. For $i = 1, 2, 3$ let $v_i$ be a generator of $L_i$. Since $L_3 \not \subset H_2$ one can take $v_3 = v_1 + w$ with $w \in H_1 \cap H_2$ non-zero.

The intersection $P_1 \cap P_2$ is made of matrices $g \in \GL_n(\C)$ such that
\begin{itemize}
\item for $i =1,2$ there is $t_i \in \C^\times $ such that $g v_i = t_i v_i$;
\item $g(H_1 \cap H_2) = H_1 \cap H_2$.
\end{itemize}
Therefore for $g \in P_1 \cap P_2$ one has $ gv_3 = t_1 v_1 + gw$ and $gw \in H_1 \cap H_2$. Thus $g$ stabilizes the line $L_3$ if and only if $g w = t_1 w$. 

This implies $\dim (P_1 \cap P_2 \cap P_3) \le (n-3)^2 + 1$.
\end{proof}

\begin{definition} A semi-stable point $x = \{ (L_i, H_i)\}_{i = 1, 2, 3}$ such that $\pi(x) = -1$ is said \emph{maximally degenerate} if $\dim L_1 + L_2 + L_3 = 2$ and $\dim H_1 \cap H_2 \cap H_3 = n - 2$.
\end{definition}

For $i = 1, 2, 3$ let $v_i$ be a basis of $L_i$ and $\phi_i$ a linear form defining $H_i$. If $x$ is maximally degenerate then one can take
\begin{align*} 
v_3 &= v_1 + v_2, & \phi_3 &= \phi_2(v_1) \phi_1 - \phi_1(v_2) \phi_2.
\end{align*}

\begin{proposition} \label{Prop:QuotientLineHyperplaneSingularFiber} With the notations introduced above:
\begin{enumerate}
\item An element $g \in G$ stabilizes a maximally degenerate point $x = \{ (L_i, H_i)\}_{i = 1, 2, 3} $ if and only if
\begin{enumerate}
\item $g_{\rvert L_1 + L_2} = t \cdot \id_{L_1 + L_2}$ for some $t \in \C^\times$;
\item $H_1 \cap H_2$ is stable under $g$.
\end{enumerate}
In particular the stabilizer of $x$ has dimension $(n-2)^2$.
\item The unique closed orbit contained in $\pi^{-1}(-1)$ is the one given by maximally degenerate points. 
\end{enumerate}
\end{proposition}

\begin{proof} (1) Similar to the proof of Theorem \ref{Thm:Quotient(1,3)flags} (5). 

(2) It suffices to show that the orbit of $x$ is of minimal dimension among the orbits contained in $\pi^{-1}(-1)$.  Let $x = \{ (L_i, H_i) \}_{i = 1, 2, 3}$ be a semi-stable point of $X$ such that $\pi(x) = -1$ and which is not maximally degenerate. One has to show that the stabilizer of $x$ has dimension $\le (n-2)^2 - 1$. 

For $i = 1, 2, 3$ let $v_i$  be a generator of $L_i$, $\phi_i$ a linear form defining $H_i$. If $g \in G$ belongs to the stabilizer of $x$ the there exists $t_i, u_i \in \C^\times$ ($i = 1, 2, 3$) such that $gv_i = t_i v_i$ and $\phi_i \circ g^{-1} = u_i \phi_i$. For $i \neq j$ one has
$$ \phi_i (v_j) = \phi_i(g^{-1} g(v_j)) = u_i \phi_i (t_j v_j) = u_i t_j \phi_i(v_j),$$
thus $u_i t_j = 1$. This implies $t_1 = t_2 = t_3 = u_1^{-1} = u_2^{-1} = u_3^{-1}$. In particular, 
\begin{enumerate}
\item $g_{\rvert L_1 + L_2 + L_3} = t \id_{L_1 + L_2 + L_3}$ for some $t \in \C^\times$;
\item $H_1 \cap H_2 \cap H_3$ is stable under $g$.
\end{enumerate}
If the vectors $v_1, v_2, v_3$ or the linear forms $\phi_1, \phi_2, \phi_3$ are linearly independent, then $\dim G_x \le (n - 3)^2$. This concludes the proof.
\end{proof}

\subsection{Quadruples of flags} \label{par:QuadruplesLineHyperplane} A permutation $\sigma \in \mathfrak{S}_4$ is written $(\sigma(1) \sigma(2) \sigma(3) \sigma(4))$. Put the derangements of $4$ elements in lexicographic order:
\begin{align*}
\sigma_1 &= (2143) & \sigma_2 &= (2341) &\sigma_3 &= (2413) \\ 
\sigma_4 &= (3142) &\sigma_5 &= (3412) &\sigma_6 &= (3421) \\
\sigma_7 &= (4123) &\sigma_8 &=(4312) &\sigma_9&= (4321).
\end{align*}
For $i =1, \dots, 9$ denote $s_i$ the $G$-invariant global section associated to the derangement $\sigma_i$.

\begin{theorem}\label{Theorem:linehyperplanengeq4} If $n \ge 4$ the map $\pi \colon X^\ss \to \P^8_\C$ identifies the GIT quotient $X^\ss / G$ with the subvariety of $\P^8_\C$ given by the equations
\begin{align*}
x_1 x_5 &= x_3 x_4, &
x_1 x_9 &= x_2 x_7, &
x_5 x_9 &= x_6 x_8,
\end{align*}
whereas for $n = 3$ the equation 
$$ x_1 + x_5 + x_9  = x_2 + x_3 + x_4 + x_6 + x_7 + x_8, $$
has to be added. 
\end{theorem}

\begin{remark} In each of the previous equations, the left-hand side is made of variables corresponding to derangements of order $2$. On the right-hand side, if the variable corresponding to a derangement $\sigma$ appears, then the variable corresponding to $\sigma^{-1}$ occurs too.
\end{remark}

\begin{proof} It is an elementary computation showing that the relations
\begin{align*}
s_1 s_5 &= s_3 s_4, & s_1 s_9 &= s_2 s_7, & s_5 s_9 &= s_6 s_8,
\end{align*}
hold. The equations in the statement give rise to a $5$-dimensional, irreducible, complete intersection in $\P^8$. On the other hand, the GIT quotient $X^\ss / G$ is an integral variety of dimension $5$ by Theorem \ref{Thm:Quotient(1,3)flags}.
\end{proof}

\subsection{Real forms}\label{subsection:realformsLH} Let $G$ be a real form of $\SL_{n, \C}$ admitting a parabolic subgroup of type $(1, n-1)$ defined over $\R$. 

Let $F$ be the corresponding flag variety and $X = F^r$ and $L$ the line bundle induced by the Pl\"ucker embedding of $X_\C$. Consider the Zariski open subset of semi-stable points $X^\ss$ with respect to $G$ and $L$. Because of the compatibility of construction of invariants with extension of scalars (cf. Theorem \ref{Thm:AffineRealGIT} (5)) the real vector space $\HH^0(X, L)^{G}$ is of dimension $!r$ and one has a well-defined map
$$ \pi \colon X^\ss \too \P(\HH^0(X, L)^{G}),$$
which induces a closed embedding $X^\ss / G \to \P(\HH^0(X, L)^{G})$.

\subsubsection{$\SL_n(\R)$} Let $G = \SL_{n, \R}$. Then $F$ is the algebraic variety $\Fl_{(1, n-1)}(\R^n)$ of flags of type $(1, n-1)$ of $\R^n$. For a derangement $\sigma$ the global section $s_\sigma$ is real. Let $\sigma_1, \dots, \sigma_{!r}$ be the derangement on $r$ elements and, for $i =1, \dots, !r$, $s_i$ be the associated $G$-invariant global section. The map $\pi$ is given by
$$ \pi(x) = [s_{1}(x) : \cdots : s_{!r}(x)].$$

The open subset $U_\C \subset X_\C$ of points in general position is stable under complex conjugation and therefore comes from an open subset $U \subset X$.

\begin{proposition} \label{prop:GeneralRealSplitFormLineHyperplane} The map $\theta \colon U(\R) / \SL_n(\R) \to \P^{!r - 1}(\R) $ is injective as soon as $n$ is odd or $n > r$. If $n \le r$ and $n$ is even then $\theta$ is two-to-one.
\end{proposition}

\begin{proof} This can be seen by reproducing the argument of Theorem \ref{Thm:Quotient(1,3)flags} (4) over $\R$. The difference is that, when $n \le r$, it is not always possible to have an element of the transporter of determinant $1$. 

An alternative way to prove the statement is computing Galois cohomology groups of the stabilizer of a point $x \in U(\R)$. The argument in the proof of Theorem \ref{Thm:Quotient(1,3)flags} (5) works over $\R$ and shows that an element $g \in \SL_n(\R)$ stabilizes $x$ if and only if
\begin{enumerate}
\item $g_{\rvert L_1 + \cdots L_r} = t \cdot \id_{L_1 + \cdots + L_r}$ for some $t \in \C^\times$;
\item $H_1 \cap \cdots \cap H_r$ is stable under $g$.
\end{enumerate}
Therefore for $n > r$ the stabilizer of $x$ is isomorphic to the algebraic group
$$ H = \{ (t, g) \in \mathbb{G}_{m, \R} \times_\R \GL_{n - r, \R} : t^r \det g = 1 \}.$$
Taking the Galois cohomology sequence associated to the short exact sequence
$$ 1 \too \mathbb{G}_{m, \R} \too H \too \GL_{n - r, \R} \too 1,$$
one obtains $\HH^1(\R, H) = 0$ thanks to Theorem \ref{thm:Satz90}. For $n \le r$ the stabilizer of $x$ is the center $\mu_n$ of $\SL_n$, thus
$$ \HH^1(\R, \mu_n) = 
\begin{cases}
\Z / 2\Z & \textup{if $n$ is even},\\
0 & \textup{otherwise}. 
\end{cases}$$
This concludes the proof.
\end{proof}

\begin{remark} The lack of injectivity disappears as soon as one quotients by $\GL_n(\R)$ instead of $\SL_n(\R)$.
\end{remark}

\begin{example} Suppose $r = 3$ and go back to the notations introduced in paragraph \ref{par:TripletsLineHyperplane}. The global sections $s_{123}$ and $s_{132}$ are real, thus the quotient map $\pi \colon X^\ss \to \P^1_\R$ is given by $\pi(x) = [s_{123}(x) : s_{132}(x)]$. It induces a natural map
$$\theta \colon X^\ss(\R) / \SL_n(\R) \to \P^1(\R). $$

\begin{proposition} The map $\theta$ is a homeomorphism.
\end{proposition}

\begin{proof} According to Proposition \ref{Prop:ProjectionOfMaxSeparatedQuotientToGITQuotientProj} it suffices to prove that $\theta$ is bijective. Surjectivity is easy: let $e_1^\vee, \dots, e_n^\vee$ be the dual basis and set
\begin{align*}
L_1 &= \langle e_1 \rangle, & L_2 &= \langle e_n \rangle, &L_3 &= \langle v \rangle, \\
H_1 &:  e_n^\vee= 0, & H_2 &: e_1^\vee = 0 , & H_2 &: \phi = 0,
 \end{align*}
where $v \in V$ and $\phi \in V^\vee$ are such that $e_n^\vee(v) = \phi(e_1) = 1$. Then,
$$ \pi(x) = [e_1^\vee(v) : \phi(e_n)], $$ 
whence the surjectivity of $\theta$.

Concerning the injectivity, according to Proposition \ref{prop:GeneralRealSplitFormLineHyperplane}, it suffices to prove it only for points lying on the closed orbits contained in the fibers at $0$, $\infty$ and $-1$. In each of these cases maximally degenerate points lie clearly in the same orbit under $\SL_n(\R)$.
\end{proof}
\end{example}

\subsubsection{$\SU(n-1, 1)$} \label{par:GeneralQuotientLineHyperplaneHermitian} Let $h$ be a hermitian form of signature $(n-1, 1)$ on $\C^n$ and $\SU(n-1, 1)$ the subgroup of $\SL_n(\C)$ of elements respecting the hermitian form $h$.

Let $G$ be the real form of $\SL_{n, \C}$ such that $G(\R) = \SU(n-1, 1)$. For a linear subspace $W \subset V$ denote by $W^\bot$ its orthogonal with respect to the hermitian form $h$. The flag variety $F$ is defined by the anti-holomorphic involution $\tau$ on $F$ associating to a complex flag $(L, H)$ the flag $(H^\bot, L^\bot)$.

The anti-holomorphic involution $\tau$ can be extended to $\P(V) \times_\C \P(V^\ast)$ in a way such that the Pl\"ucker embedding is compatible to these involutions.  For $v \in V$ denote by $\alpha_v$ the linear form on $V$ defined for every $w \in V$ by $ \alpha_v(w) \df h( v, w )$.
Conversely for a linear form $\alpha$ on $V$ write $v_\alpha \in V$ for the unique vector such that $\alpha_{v_\alpha} = \alpha$. Consider the anti-holomorphic involution $\tau$ on $\P(V) \times_\C \P(V^\ast)$ given by 
$$ \tau \colon (v, \alpha) \mapsto (v_\alpha, \alpha_v).$$

For $i = 1, \dots, r$ let $v_i \in V$ be an isotropic vector and $\alpha_i \df \alpha_{v_i}$. For a derangement $\sigma$, 
$$
s_\sigma((v_1 \otimes \alpha_1) \otimes \cdots \otimes (v_r \otimes \alpha_r)) = \prod_{i = 1}^r h(v_i, v_{\sigma(i)}).
$$
In particular $\tau^\ast s_{\sigma} = \bar{s}_{\sigma^{-1}}$. If a derangement $\sigma$ is order $2$, then
$$ s_\sigma((v_1 \otimes \alpha_1) \otimes \cdots \otimes (v_r \otimes \alpha_r)) = \prod_{i = 1}^r |h(v_i, v_{\sigma(i)})|, $$
which is a non-negative real number. If $\sigma$ is not of order $2$ the global sections
\begin{align*} \tfrac{1}{2}(s_{\sigma} + s_{\sigma^{-1}}), && \tfrac{1}{2 \sqrt{-1}} (s_{\sigma} - s_{\sigma^{-1}}), \end{align*}
are defined over $\R$ and take as values the real and imaginary part of $s_\sigma$.

\begin{proposition} \label{Prop:SemistableIsotropicLines} Let $x = (L_1, \dots, L_r)$ be a  $r$-tuple of isotropic lines.
\begin{enumerate}
\item $x$ is semi-stable if and only if, for every isotropic line $L$,
$$ \# \{ i : L = L_i \} \le \frac{r}{2};$$
\item if $r$ is even, then $x$ is semi-stable if and only if there is a derangement of order $2$ such that $s_\sigma(x) \neq 0$.
\end{enumerate}
\end{proposition}

\begin{proof}If an isotropic line is repeated more than $r/2$ times, then it is clear that all invariants $s_\sigma$ vanish at $x$.

The converse implication is proved separately for $r$ odd and even.

Suppose $r = 2r'$ with $r'$ integer. We prove statement (2) by induction on $r' \ge 2$. Suppose $s_\sigma(x) = 0$ for all derangements of order $2$. For $r' = 2$ one sees that there is a line repeated at least $3$ times. Assume $r' \ge 3$ and that the statement is true for $r'  - 1$. One may suppose $L_{r-1} \neq L_{r}$. Then, by the inductive hypothesis, among the lines $L_1, \dots, L_{r-2}$ there is a line $L$ repeated at least $\frac{r}{2}$ times. Up to permutation one may suppose $ L= L_1 = \cdots = L_{r/2}$. Consider the derangement of order $2$ given by $\sigma(i) = \frac{r}{2} + i$ and $\sigma(\frac{r}{2} + i) =  i$ for $i = 1, \dots, \frac{r}{2}$. The hypothesis $s_\sigma(x) = 0$ implies that there is $j = \frac{r}{2}, \dots, r$ such that $L_j = L$. This contradicts the semi-stability.

Suppose $r$ odd and write $r = 2r' + 1$. The statement is proved by induction on $r' \ge 1$. For $r' = 1$ the statement is clear. Suppose $r' \ge 2$. Up to permutation one may assume $L_1 \neq L_2$. Suppose $s_\sigma(x) = 0$ for all derangements $\sigma$ such that $\sigma_{\{1, 2\}} = (12)$. Then, by inductive hypothesis, there is a line $L$ such that 
$$ \# \{ i \in \{ 3, \dots, r\}: L = L_i \} \ge \frac{r-1}{2}.$$
By hypothesis $L_1, L_2 \neq L$. Let $k$ be such that $L_k = L$. Suppose $s_{\sigma}(x) = 0$ for all derangements $\sigma$ such that $\sigma_{\{ 1, 2, k\}} = (12k)$. By the case with even $r$, among the lines $L_i$ for $i \neq 1, 2, k$ there is a line $L'$ repeated at least $ \frac{r-1}{2}$ times. One has necessarily $L = L'$ thus the line $L$ is repeated at least $\frac{r+1}{2}$ times, contradicting the hypothesis.
\end{proof}

Let $\ell$ be the number of derangements of order $2$: it is $0$ if $r$ is odd and 
$$(r-1)!! := \prod_{i = 0}^{ \frac{r}{2}  - 1} (r - 2i),$$ if $r$ is even. Let $m$ the cardinality of the quotient of $\mathfrak{D}_r$ by the relation $\sigma \sim \tau$ if $\tau = \sigma^{-1}$: it is $!r$ if $r$ is odd and $\tfrac{!r - (r- 1)!!}{2}$ if $r$ is even.

 Order the derangements on $r$ elements in a way such that
\begin{itemize}
\item $\sigma_i$ is of order $2$ for $i = 1, \dots, \ell$;
\item if $i = \ell + 1, \dots, m$ then $\sigma_i^{-1} = \sigma_j$ for $j > m$.
\end{itemize}
Identify $\P^{!r -1}(\R)$ with the projective space associated to $\R^\ell \times \C^{m - \ell}$ (seen as a real vector space). For $x \in X^\ss(\R)$ the point $\pi(x)$ is the line generated by
$$ (s_1(x) , \dots, s_\ell(x),  s_{\ell + 1}(x), \dots, s_m(x)).$$

On the other hand, for isotropic vectors $v_1, \dots, v_r$, for non-zero complex numbers $\lambda_1, \dots, \lambda_r$ and for a derangement $\sigma$,
$$ \prod_{i = 1}^r h(\lambda_i v_i, \lambda_{\sigma(i)}v_{\sigma(i)}) = |\lambda_1|^2 \cdots |\lambda_r|^2 \prod_{i = 1}^r h(v_i, v_{\sigma(i)}).$$
This shows that the class $p(x)$ of the vector $(s_1(x) , \dots , s_m(x))$ in the sphere
$$ \SS(\R^\ell \times \C^m) \df (\R^\ell \times \C^m \smallsetminus \{ 0\})/\R_{>0},$$ only depends on the line generated by the vectors $v_1, \dots, v_r$. This induces well-defined maps  $p \colon X^\ss(\R) \to \SS(\R^\ell \times \C^m)$ and
$$ \epsilon \colon X^\ss(\R) / G(\R) \too \SS(\R^\ell \times \C^m).$$
Consider the  $G(\R)$-saturated open subset 
$$ W = \{ (L_1, \dots, L_r) \in X^\ss(\R) : \dim (L_1 + \cdots + L_r) \ge 3 \}.$$

\begin{theorem} \label{Thm:QuotientIsotropicLines} With the notations introduced above:
\begin{enumerate}
\item the map $W / \SU(1, n-1) \to \P(\R^\ell \times \C^m)$ is injective;
\item for $r$ odd the map $\epsilon \colon X^\ss(\R) / G(\R) \to \SS(\C^{!r/2})$ is injective;
\item for $r$ even and $x, x' \in X^\ss(\R)$, 
$$p(x) = p(x') \textrm{ if and only if } \pi(x) = \pi(x'). $$
\end{enumerate}
\end{theorem}

\begin{remark} When $r$ is odd, statement (1) implies that the image of $\epsilon$ in $\SS(\C^{!r/2})$ is contained in a ``hemisphere''.
\end{remark}

Before proving  Theorem \ref{Thm:QuotientIsotropicLines} let us illustrate it through two examples.

\begin{example} Let $r = 3$ and go back to the notation introduced in paragraph \ref{par:TripletsLineHyperplane}. Identify $\SS(\C)$ with $\U(1)$. For a semi-stable triple
$ x= ([v_1], [v_2], [v_3])$ of isotropic lines,
$$ p([v_1], [v_2], [v_3]) = \arg (h( v_1, v_3 ) h( v_3, v_2 ) h( v_2, v_1 )),$$
where $\arg(z) = z / |z|$ is the argument of a non-zero complex number $z$. 

Identify the real projective line $\P^1(\R)$ with $\U(1) / \{ \pm 1\}$ and in turn with $\U(1)$ through the map $u \mapsto u^2$. The map $\pi$ can be expressed as
$$ \pi([v_1], [v_2], [v_3]) =  \arg (h( v_1, v_3 ) h( v_3, v_2 ) h( v_2, v_1 ))^2.$$

\begin{proposition} \label{Prop:Quotient3IsotropicLines} With the notations introduced above, the map $p$ induces a homeomorphism of $X^\ss(\R) / G(\R)$ with
$$ H \df \{ u \in \U(1) : \Re(u) \le 0\}.$$
\end{proposition}

\begin{figure}[h]
\begin{tikzpicture}
\draw[dotted] (-2,0) arc (-90:90:1);
\draw (-2,2) arc (90:270:1);
\draw[->] (-0.35, 1) -- (0.35, 1) node[midway, above] {$\theta$};
\draw (2,1) circle (1cm);
\node[anchor=west] at (1, 1) {$-1$};
\node[anchor=south] at (-2, 2) {$i$};
\node[anchor=north] at (-2, 0) {$-i$};
\fill (-2, 2) circle (1pt);
\fill (-2, 0) circle (1pt);
\fill (1, 1) circle (1pt);
\draw (1, 1) circle (2pt);
\end{tikzpicture}
\caption{The map $\theta$ maps the half-circle onto the entire circle. It is injective outside $\pm i$. The Cartan invariant $\pm i$ corresponds to configurations of $3$ isotropic lines in $\C^n$ which, seen as points of $\P^{n-1}(\C)$, lie on a line.}
\end{figure}
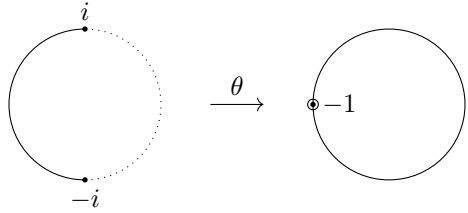

\begin{proof} One may suppose that the hermitian form $h$ is given by the matrix
$$ \begin{pmatrix}
0 &  0 & 1 \\
0 & \id  & 0 \\
1 & 0 & 0
\end{pmatrix}.
$$
In this case the vectors $e_1$, $e_n$ are isotropic. It is easy to prove that $p$ is surjective onto $H$: consider the isotropic vectors $v_1 =  e_1$, $v_2 = e_n$ and 
$$
v_3 = u e_1 + \sqrt{2 |\Re(u)|} e_2 + e_n,
$$
where $u \in \U(1)$ has non-positive real part. Then $h( v_1, v_3 ) h( v_3, v_2 ) h( v_2, v_1 ) = u$ whence the surjectivity. 

Let us prove that the image is contained in $H$. Let $v_1, v_2, v_3$ be isotropic vectors of $V$ such that  $ \lambda \df h( v_1, v_2 ) h( v_2, v_3 ) h( v_3, v_1 ) \neq 0$. One may assume $v_1 = e_1$ and $v_2 = e_n$. Write $ v_3 = x_1 e_1 + \cdots + x_n e_n$. Then $\lambda = \bar{x}_n x_1$ and 
$$ 0 = h( v_3, v_3 ) = 2 \Re(\lambda) + |x_2|^2 + \cdots + |x_{n-1}|^2 .$$
This implies $\Re(\lambda) \le 0$, that is, the image of $f$ is contained in $H$.  The injectivity outside $\{ \pm i \}$ follows immediately.
\end{proof}

\end{example}

\begin{example} Let $r = 4$ and suppose that the hermitian form $h$ is given by the matrix
$$
\begin{pmatrix} 
0 & 0  & 1 \\
0 & \id & 0 \\
1 & 0 & 0
\end{pmatrix}.
$$
Consider the following isotropic vectors:
\begin{align*}
v_1 &= e_1, & v_2&= e_n, & v_3 = v_4 &=  i e_1 + e_n, \\
v'_1 &= e_1, & v'_2&= e_n, & v'_3 = v'_4 &=  -i e_1 + e_n.
\end{align*}
One has
\begin{align*}
h(v_1, v_2) h(v_2, v_3) h(v_3, v_1) &= i, &
h(v'_1, v'_2) h(v'_2, v'_3) h(v'_3, v'_1) &= - i,
\end{align*}
thus the semi-stable quadruplets 
\begin{align*}
x &= ([v_1], [v_2], [v_3], [v_4]), &x' &= ([v_1'], [v'_2], [v'_3], [v'_4]), 
\end{align*} cannot be conjugated by $\SU(n-1,1)$. On the other hand one has $p(x) = p(x')$. 
\end{example}

The rest of this section is devoted to the proof of the Theorem \ref{Thm:QuotientIsotropicLines}: it is made of three lemmas. Before dipping into the proof, let us fix some notation. 

Let $ v = (v_1, \dots, v_n)$, $v' = (v'_1, \dots, v_n')$ be $n$-tuples of non-zero isotropic vectors. For $i = 1, \dots, r$ let $L_i$, $L'_i$ be the line generated respectively by $v_i$, $v'_i$ and denote by $x_i, x_i' \in F(\R)$ the associated points. The $n$-tuples $ x = (x_1, \dots, x_n)$, $x' = (x'_1, \dots, x'_n)$ are supposed to be semi-stable points of $X$.

\begin{definition} Let $i, j , k = 1, \dots, r$ be such that the lines generated by the vector $v_i$, $v_j$, $v_k$ are pairwise distinct. Set
$$ \lambda_{ijk}(v, v') := \frac{h(v_i, v_j)}{h(v_i', v_j')} \cdot \frac{h(v'_k, v'_j)}{h(v_k, v_j)} \cdot \frac{h(v'_i, v'_k)}{h(v_i, v_k)}. $$
\end{definition}

\begin{lemma} \label{lem:ExistenceOfIsometry} Suppose that there exists $k = 1, \dots, r$ with the following property: for $i, j, \alpha, \beta = 1, \dots, r$ such that the lines $L_i, L_j, L_k$ as well as the lines $L_\alpha, L_\beta, L_k$ are pairwise distinct,
$$ \lambda_{ijk}(v, v')= \lambda_{\alpha \beta k}(v, v'),$$
and such a complex number is real and positive. Then, there exists $g \in \SU(n-1, 1)$ such that $gx = x'$.
\end{lemma}

\begin{proof} First of all remark that it suffices to look for $g \in \U(n-1, 1)$ such that $gx = x'$: indeed, if $\det g \neq 1$ it suffices to divide $g$ by a $n$-th root of its determinant.

Let $s$ be the dimension of the vector space generated by $v_1, \dots, v_r$. The semi-stability hypothesis implies $s \ge 2$. Up to permutation, one may assume that $v_1, \dots, v_s$ are linearly independent.

Suppose $n \ge s$. Set $t_k = 1$ and
$$ \bar{t}_i = \frac{h(v_i, v_j)}{h(v_i', v_j')} \cdot \frac{h(v'_k, v'_j)}{h(v_k, v_j)},$$
for $j \neq i,k$ (it is independent of $j$ by hypothesis). Then, for $i \neq j,$
\begin{align*}
\frac{h(t_i v'_i, t_j v'_j)}{h(v_i, v_j)} &= \bar{t}_i t_j \frac{h(v'_i, v'_j)}{h(v_i, v_j)} = \lambda_{ijk}(v, v'),
\end{align*}
which is a positive real number independent of $i$ and $j$ by hypothesis. Call $\lambda$ such a number. Then there exists an isometry of the sub-vector space $\langle v_1, \dots, v_s \rangle$ onto $\langle v_1', \dots, v_s' \rangle$ sending the vector $v_i$ to $t_i / \sqrt{\lambda} v_i'$. By Witt's Theorem, such an isometry can be extended to $g \in \U(n-1, 1)$.

If $r = s$ there is nothing left to prove. Suppose $s < r$ and let $j > s$. Since $v_j$ belongs to the vector space generated by $v_1, \dots, v_s$, in order to prove $[g v_j] = [v_j']$ it suffices to prove that the complex number
$$ \frac{h(v_i', g v_j)}{h(v'_i, v'_j)}$$
does not depend on $i$. Since $g$ is an isometry,
$$ h(v_i', g v_j) = h(g^{-1}v_i',  v_j) = \sqrt{\lambda} \bar{t}_i^{-1} h(v_i, v_j).$$
In particular,
$$ \frac{h(v_i', g v_j)}{h(v'_i, v'_j)} =  \sqrt{\lambda} \frac{h(v_i', v_j')}{h(v_i, v_j)} \cdot \frac{h(v_k, v_j)}{h(v'_k, v'_j)} \cdot \frac{h(v_i, v_j)}{h(v'_i, v'_j)} = \sqrt{\lambda} \frac{h(v_k, v_j)}{h(v'_k, v'_j)},$$
which does not depend on $i$. This concludes the proof.
\end{proof}

\begin{lemma} \label{lemma:IndependencyTriratioHermitian} Suppose there exists $g \in \GL_n(\C)$ such that $gx = x'$. Then, 
\begin{enumerate} 
\item  For $i, j, k, \alpha, \beta = 1, \dots, r$ such that the lines $L_i, L_j, L_k$ as well as the lines $L_\alpha, L_\beta, L_k$ are pairwise distinct,
$$ \lambda_{ijk}(v, v')= \lambda_{\alpha \beta k}(v, v').$$
In particular both sides of the previous identity are real numbers.
\item Let $i, j, k = 1, \dots, r$ be such that $\dim(L_i + L_j + L_k) = 3$. Then,
$$ \lambda_{ijk}(v, v') > 0.$$
\item Suppose $r$ odd and $p(x) = p(x')$. Then, there exists $k$ with the following property: for all $i, j = 1, \dots, r$ such that $L_i, L_j, L_k$ are pairwise distinct
$$ \lambda_{ijk}(v, v') > 0.$$
\end{enumerate}
\end{lemma}

\begin{proof}Let us fix some notation. For simplicity write $ \lambda_{ij} = \lambda_{ijk}(v, v').$ For $\rho = 3, 4$ let $X_\rho = F^\rho$ and $\pi_\rho \colon X_\rho^\ss \to \P^{!\rho - 1}_\R$ be the quotient map.

(1) In order to prove $\lambda_{ij} = \lambda_{\alpha \beta}$ for  all $\alpha, \beta = 1, \dots, r$ such that the lines $L_\alpha, L_\beta, L_k$ are pairwise distinct, it suffices to show $\lambda_{ij} = \lambda_{ji}$ and $\lambda_{ij} = \lambda_{i\beta}$.  The semi-stable points 
\begin{align*}
y &= ([L_i], [L_j], [L_k]), &y' &= ([L'_i], [L'_j], [L'_k]),
\end{align*} 
of $X_3$ verify $gy = y'$ because of the hypothesis $gx = x'$. Thus $\pi_3(y) = \pi_3(y)$ which implies,
$$
\frac{h(v_i, v_j) h(v_j, v_k) h(v_k, v_i)}{h(v_i, v_k) h(v_k, v_j) h(v_j, v_i)} = \frac{h(v'_i, v'_j) h(v'_j, v'_k) h(v'_k, v'_i)}{h(v'_i, v'_k) h(v'_k, v'_j) h(v'_j, v'_i)},
$$
thus $\lambda_{ij} = \lambda_{ji}$.

Suppose $r \ge 5$. In order to show $\lambda_{ij} = \lambda_{\alpha \beta}$ for $\alpha, \beta = 1, \dots, r$ distinct, it suffices to prove the equalities $\lambda_{ij} = \lambda_{i\beta}$ for $i \neq \beta$. The latter is equivalent to
$$ \frac{h(v_i, v_j) h(v_k, v_\beta)}{h(v_i, v_\beta) h(v_k, v_j)} = 
\frac{h(v'_i, v'_j) h(v'_k, v'_\beta)}{h(v'_i, v'_\beta) h(v'_k, v'_j)}.
$$
Again the semi-stable points 
\begin{align*}
y &= ([L_i], [L_j], [L_k], [L_\beta]), &y' &= ([L'_i], [L'_j], [L'_k], [L'_\beta]),
\end{align*} of $X_4$ satisfy $\pi_4(y) = \pi_4(y')$ because of the assumption $gx = x'$. Let $\sigma, \tau$ be derangements such that $\sigma(\alpha) = \tau(\alpha)$ for all $\alpha \neq i, k$ and $\sigma(i) = j$,  $\tau(i) = \beta$,  $\sigma(k) = \beta$,  $\tau(k) = j$ (they exists thanks to Lemma \ref{lemma:ExistenceOfDerangements}). Let $s_\sigma, s_\tau$ the corresponding $G$-invariant sections. Then,
$$ \frac{s_\sigma(y)}{s_\tau(y)} = \frac{h(v_i, v_j) h(v_k, v_\beta)}{h(v_i, v_\beta) h(v_k, v_j)},$$
and similarly for $y'$. The condition $\pi_4(y) = \pi_4(y')$ implies 
$$s_\sigma(y) / s_\tau(y) = s_\sigma(y')/ s_\tau(y').$$
Since $\bar{\lambda}_{ij} = \lambda_{ji} $, then complex number $\lambda_{ij}$ is real because $\lambda_{ij} =  \lambda_{ji}$.

(2) Multiplying and dividing by $h(v'_j, v'_i)/h(v_j, v_i)$ one sees that the positivity of $\lambda_{ij}$ is equivalent to the positivity of the real number
$$ \mu = \frac{h(v'_i, v'_k)}{h(v_i, v_k)} \cdot   \frac{h(v'_k, v'_j)}{h(v_k, v_j)} \cdot \frac{h(v'_j, v'_i)}{h(v_j, v_i)}.$$

Suppose $\dim(L_i + L_j + L_k) = 3$. The semi-stable points 
\begin{align*}
y &= ([L_i], [L_j], [L_k]), & y' &= ([L'_i], [L'_j], [L'_k]),
\end{align*}
of $X_3$ are conjugated under $\SL_n(\C)$ thus $\pi_3(y) = \pi_3(y')$. According to Proposition \ref{Prop:Quotient3IsotropicLines} the hypothesis $\dim(L_i + L_j + L_k) = 3$ implies that the points $y, y'$ are actually conjugated under $\SU(1, n-1)$, thus $\mu >0$.

(3) This is based on the following:

\begin{lemma} \label{Lemma:ExistenceOfNonVanishing3Cycles} Suppose $r \ge 5$ odd. Let $x = (L_1, \dots, L_r)$ be a semi-stable $r$-tuple of isotropic lines.  There exists $k = 1, \dots, r$ with the following property: for all $i, j = 1, \dots, r$ such that the lines $L_i, L_j, L_k$ are pairwise distinct, there exists a derangement 
$\sigma \in \mathfrak{D}_r$ such that
\begin{enumerate}
\item $\sigma_{\rvert \{i, j, k\}} = (ijk)$, 
\item the induced derangement $\sigma_{\rvert \{ 1, \dots, r\} \smallsetminus \{ i, j, k\}}$ is of order $2$,
\item  $s_{\sigma}(x) \neq 0$. 
\end{enumerate}
\end{lemma}

The proof of the preceding Lemma is postponed to the end of the proof. Let $k$ as in the statement. Let $i, j = 1, \dots, r$ be such that the lines $L_i, L_j, L_k$ are pairwise distinct. Let $\sigma$ be a derangement such that $\sigma_{\{ i, j, k \}} = (ijk)$, $\sigma$ induces a derangement of order $2$ on  $ \{ 1, \dots, r\} \smallsetminus \{ i, j, k\}$ such that $s_{\sigma}(x) \neq 0$. Then,
$$ \prod_{\alpha \neq i, j, k} h(v_\alpha, v_{\sigma(\alpha)}) = \prod_{\alpha \neq i, j, k} |h(v_\alpha, v_{\sigma(\alpha)})| > 0,$$
and similarly for $v'_1, \dots, v'_r$. It follows that $\mu$ is positive if and only if $s_\sigma(x') / s_\sigma(x)$ is positive, which is implied by the hypothesis $p(x) = p(x')$.
\end{proof}

\begin{proof}[{Proof of Lemma \ref{Lemma:ExistenceOfNonVanishing3Cycles}}] If none of the lines $L_1, \dots, L_r$ is repeated exactly $\tfrac{r-1}{2}$ times, then every $k$ does the job. Indeed, let $i, j , k = 1, \dots, r$ be such that the lines $L_i, L_j, L_k$ are pairwise distinct. By contradiction suppose $s_{\sigma}(x) = 0$ for every derangement $\sigma$ such that $\sigma_{\rvert \{i, j, k\}} = (ijk)$ and $\sigma_{\rvert \{ 1, \dots, r\} \smallsetminus \{ i, j, k\}}$ is of order $2$. Then Proposition \ref{Prop:SemistableIsotropicLines} implies that there is a line repeated at least $\frac{r-1}{2}$ times contradicting the previous assumption.

Suppose there is a line $L$ repeated exactly $\tfrac{r-1}{2}$ times. Pick $k$ such that $L_k = L$ and let $i, j = 1, \dots, r$ be such that the lines $L_i, L_j, L_k$ are pairwise distinct. By contradiction suppose $s_{\sigma}(x) = 0$ for every derangement $\sigma$ such that $\sigma_{\rvert \{i, j, k\}} = (ijk)$ and $\sigma_{\rvert \{ 1, \dots, r\} \smallsetminus \{ i, j, k\}}$ is of order $2$. Then Proposition \ref{Prop:SemistableIsotropicLines} implies that there is a line $L'$ repeated at least $\tfrac{r-1}{2}$ times. Necessarily $L' = L$, thus the line $L$ is repeated at least $\frac{r+1}{2}$ times contradicting the semi-stability of $x$.
\end{proof}

\section{Quadruples of planes in $\C^{4}$} \label{sec:QuadruplesOfPlanes}

\subsection{Setup} Consider the Grassmannian variety $\Gr(2, 4)$ of $2$-dimensional subvector spaces of $\C^4$. Two planes $W_1$, $W_2$ are opposite if and only if their intersection is $0$. For $i = 1,2$ let $P_i$ denote the stabilizer of $W_i$. The subset $\Opp(W_i)$ of planes opposite to $W_i$ is Zariski open in $\Gr(2, 4)$. Set:
\begin{align*}
V &= \C^4, &
G &= \SL_{4, \C}, &
F &= \Gr(2, 4), &
X &= F^4.
\end{align*}
The semi-simple group $G$ acts on $X$. Consider the linearization of this action given by the Pl\"ucker embedding,
$$ j \colon X= F^4 \too \P(\textstyle \bigwedge^2 V)^4.$$
The considered linearization corresponds to the line bundle
$$ L = j^\ast( \pr_1^\ast \O(1) \otimes \cdots \otimes \pr_4^\ast \O(1)),$$
where $\pr_i \colon \P(\bigwedge^2 V)^4 \to F$ denotes the projection onto the $i$-th factor.

\subsection{Invariants} Let us introduce some natural invariants for this action. Identify the $1$-dimensional vector space $\bigwedge^4 V$ with $\C$ by choosing the basis $ e_1 \wedge e_2 \wedge e_3 \wedge e_4$. 
With this in mind, let $t = w_1 \wedge w_2$, $t' = w_1' \wedge w_2'$ for some vectors $w_1, w_2, w_1', w_2'$ of $V$. Then $t \wedge t'$ is the determinant of the matrix having $w_1, w_2, w_1', w_2'$ as columns,
$$ t \wedge t' = \det (w_1, w_2, w_1', w_2').$$
Consider the linear forms,
$$ s_{1234}, s_{1324}, s_{1423} \colon \textstyle (\bigwedge^2 V)^{\otimes 4} \too \C,$$
defined, for $t_1, t_2, t_3, t_4 \in \bigwedge^2 V$, by
\begin{align*}
s_{1234}(t_1 \otimes t_2 \otimes t_3 \otimes t_4) &:= t_1 \wedge t_2 \cdot t_3 \wedge t_4, \\
s_{1324}(t_1 \otimes t_2 \otimes t_3 \otimes t_4) &:= t_1 \wedge t_3 \cdot t_2 \wedge t_4, \\
s_{1423}(t_1 \otimes t_2 \otimes t_3 \otimes t_4) &:= t_1 \wedge t_4 \cdot t_2 \wedge t_3.
\end{align*}
They are $G$-invariant linear forms on $(\bigwedge^2 V)^{\otimes 4}$ hence they induce $G$-invariant global sections of $L$:
$$ s_{1234}, s_{1324}, s_{1423} \in \HH^0(X, L)^{G}.$$
For $i = 1, \dots, 4$ let $w_{i1}, w_{i2}$ be a basis of $W_i$.  The section $s_{1234}$ evaluated at the vector
$ w_{11} \wedge w_{12} \otimes \cdots \otimes w_{41} \wedge w_{42}$ takes the value
$$ \det(w_{11}, w_{12}, w_{21}, w_{22})  \cdot \det(w_{31}, w_{32}, w_{41}, w_{42}). $$
In particular $s_{1234}$ vanishes at a quadruple $(W_1, W_2, W_3, W_4)$ if and only if
$$ W_1 \cap W_2 \neq 0 \quad \textup{or} \quad W_3 \cap W_4 \neq 0.$$
Similar statements hold for $s_{1324}$ and $s_{1423}$.

\subsection{The quotient} 

\begin{theorem} \label{Thm:Quotient4planesC4} With notations introduced above:
\begin{enumerate}
\item The graded $\C$-algebra of $G$-invariant elements
$$ \bigoplus_{d \ge 0} \HH^0(X, L^{\otimes d})^G,$$
is generated by $s_{1234}$, $s_{1324}$  and $s_{1423}$.
\item A quadruple $(W_1, W_2, W_3, W_4)$ is semi-stable as a point of $X$ (with respect to the action of $G$ and to the linearization $L$) if and only if one of the following conditions holds:
\begin{enumerate}
\item $W_1 \cap W_2 = 0$ and $W_3 \cap W_4 = 0$;
\item $W_1 \cap W_3 = 0$ and $W_2 \cap W_4 = 0$;
\item $W_1 \cap W_4 = 0$ and $W_2 \cap W_3 = 0$.
\end{enumerate}
\item The map $ \pi \colon X^\ss \to \P^2_\C $ defined by $ \pi(x) := [s_{1234}(x) : s_{1324}(x) : s_{1423}(x)]$ is $G$-invariant and induces an isomorphism
$$ X^\ss / G \stackrel{\sim}{\too} \P^2_\C.$$
\end{enumerate}

\end{theorem}

\begin{proof} (2) and (3) follows from (1). 

(1) This is a consequence of the exceptional isogeny $\SL_{4, \C} \to \SO(6)$ and the invariant theory for $\SO(6)$. The wedge product induces a non-degenerate quadratic form on $\bigwedge^2 V$ which is invariant under the action of $\SL_{4, \C}$. This gives a map $ \SL_{4, \C} \to \SO(6)$ which is seen to be surjective.

The First Theorem of Invariant Theory for $\SO(6)$ states that, a  positive integer $d$ being fixed, the subspace of $\SO(6)$-invariant linear forms
$ \phi \colon (\textstyle \bigwedge^2 V)^{\otimes 2d} \to \C$
is generated by products of ``contractions'', \textit{i.e.} by products of the form
$$ \prod_{\alpha = 1}^d v_{\sigma(\alpha)} \wedge v_{\sigma(2d + \alpha)},$$
for some permutation $\sigma \in \mathfrak{S}_{2d}$ and $v_1, \dots, v_{2d} \in \bigwedge^2 V$. Refer to \cite[Appendix I]{AtiyahBottPatodi} for an elegant proof of this fact avoiding the Capelli Identity.

In order to compute the $G$-invariants of $\HH^0(X, L^{\otimes d})$ one has to evaluate the afore-mentioned contractions on vectors of the form $ (v_1 \otimes \cdots \otimes v_4)^{\otimes d}$,  where, for $i = 1, \dots, 4$, $W_i$ is a plane of $V$, $w_{i1}, w_{i2}$ a basis of $W_i$ and $v_i = w_{i1} \wedge w_{i2}$. Evaluated at this point a contraction $s$ is of the form
$$ \prod_{\alpha = 1}^d v_{i_\alpha} \wedge v_{j_\alpha}$$
with $i_\alpha, j_\alpha \in \{ 1, \dots, 4\}$. One may assume $i_\alpha < j_\alpha$ for every $\alpha$: indeed,  if $i_\alpha = j_\alpha$ the contraction $s$ vanishes. For $i, j =1, \dots, 4$ with $i < j$ let $a_{ij}$ be the number of times $v_i \wedge v_j$ appears in the product. Since $v_i$ occurs exactly $d$ times in $ (v_1 \otimes \cdots \otimes v_4)^{\otimes d}$ one has
\begin{align*}
a_{12} + a_{13} + a_{14} &= d, \\
a_{12} + a_{23} + a_{24} &= d,\\
a_{13} + a_{23} + a_{34} &= d, \\
a_{14} + a_{24} + a_{34} &=d.
\end{align*}
This readily implies
$$ a_{12} = a_{34}, \qquad a_{13} = a_{24}, \qquad a_{14} = a_{23}.$$
Thus $s$ is a product of $s_{1234}$, $s_{1324}$ and $s_{1423}$. 
\end{proof}

\subsection{Affine charts}
Consider the Zariski open subset $U$ of $X$ of quadruples of planes $(W_1, W_2, W_3, W_4)$ such that $W_1$ is opposite to $W_2$ and $W_4$, and $W_2$ is also opposite to $W_3$. Write a point in $\P^2_\C$ as $[t_0 : t_1: t_2]$ so that for a semi-stable point $x \in X^\ss$ one has
$$ \pi(x) = [s_{1234}(x) : s_{1324}(x) : s_{1423}(x)] = [t_0 : t_1 : t_2].$$
Set $ V \df \P^2_\C \smallsetminus \{ t_2 = 0\}$, $U := \pi^{-1}(V)$ and
\begin{align*}
W_1 &= \langle e_1, e_2 \rangle, 
&W_2 &= \langle e_3, e_4 \rangle.
\end{align*}
For $i = 1, 2$ let $P_i$ be the stabilizer of $W_i$. Since $W_1$ is opposite to $W_2$, for $i = 1, 2$ the action of $\rad^u(P_i)$ on $\Opp(W_i)$ yields an isomorphism
\begin{eqnarray*}
\phi_i \colon \rad^u(P_i) &\stackrel{\sim}{\too}& \Opp(W_i), \\
u &\longmapsto& u \cdot W_{3 - i}.
\end{eqnarray*}
The map $\rad^u(P_1) \times \rad^u P_2 \to U$,
$$ (u_1, u_2)  \longmapsto (W_1, W_2, u_2 \cdot W_1, u_1 \cdot W_2), $$
is $P_{12}$-equivariant and induces an isomorphism
$$( \rad^u P_1 \times \rad^u P_2 ) / P_{12} \stackrel{\sim}{\too} V. $$

Let us make more explicit the quotient $( \rad^u P_1 \times \rad^u P_2 ) / P_{12}$. An element of $P_{12} \df P_1 \cap P_2$ can be written as
$$
\begin{pmatrix}
M_1 & 0 \\
0 & M_2
\end{pmatrix},
$$
where $M_1, M_2$ are $2 \times 2$ invertible matrices such that $\det M_1 M_2 = 1$. Such an element will be denoted by the couple $(M_1, M_2)$. The action of $P_{12}$ on $\Opp(W_i)$ induces through $\phi_i$ the action of $P_{12}$ on $\rad^u(P_i)$ by conjugation.\footnote{Given $g \in P_1 \cap P_2$ and $h \in \rad^u(P_1)$ one has to find the unique element $h' \in \rad^u(P_1)$ such that $ gh \cdot W_2 = h' \cdot W_2$. 
Since $g$ belongs to stabilizer of $W_2$ by hypothesis, $ gh \cdot W_2  = ghg^{-1} \cdot W_2$. On the other hand $ghg^{-1}$ belongs to $\rad^u(P_1)$ because the unipotent radical is a normal subgroup of $P_1$. Thus $h ' = ghg^{-1}$.}

The Lie algebra of $\rad^u P_1$ and $\rad^u P_2$ are respectively made of block matrices of the form
\begin{align*}
& \begin{pmatrix} 
0 & A_1 \\ 0 & 0
\end{pmatrix},
&& \begin{pmatrix} 
0 & 0 \\ A_2 & 0
\end{pmatrix},
\end{align*}
where $A_1, A_2$ are $2 \times 2$ matrices. Therefore in what follows $\Lie \rad^u P_i$ is identified with $M_{2}(\C)$. For $i = 1, 2$ the exponential map $\exp \colon \Lie \rad^u P_i \to \rad^u P_i$,
\begin{align*}
\exp(A_1) &= 
\begin{pmatrix} 
\id & A_1 \\ 0 & \id
\end{pmatrix},
&\exp(A_2) &= 
\begin{pmatrix} 
\id & 0 \\ A_2 & \id
\end{pmatrix},
\end{align*}
is an isomorphism. The adjoint action of $P_{12}$ on $\Lie \rad^u P_1 \times \Lie \rad^u P_2$ is given by
$$ (M_1, M_2) \cdot (A_1, A_2) = (M_1 A_1 M_2^{-1}, M_2 A_2 M_1^{-1}).$$

The exponential map is $P_{12}$-equivariant and induces isomorphisms between the quotients
$$ (\Lie \rad^u P_1 \times \Lie \rad^u P_2) / P_{12} \stackrel{\sim}{\too} (\rad^u P_1 \times \rad^u P_2) / P_{12} \stackrel{\sim}{\too} V = \A^2_\C.$$

\begin{proposition} \label{Prop:4PlanesAffineChart}With the notations introduced above, let $A_i \in \Lie \rad^u P_i$ for $i = 1, 2$ and $x= (W_1, W_2, \exp(A_2)W_1, \exp(A_1)W_2)$. Then:
\begin{enumerate}
\item  the quotient map  $
\pi \colon U  \to  V$ is given by
$$ \pi(x) = [\det(A_1 A_2 - \id) :  \det(A_1A_2) : 1]. $$
\item the orbit of $x$ is closed if and only if the matrices $A_1, A_2$ are semi-simple and $\rk A_1 = \rk A_2$.
\end{enumerate}
\end{proposition}

\begin{proof}
(1) By definition the invariant section $s_{1423}$ does not vanish on $U$. The rational functions $s_{1234}/s_{1423}$, $s_{1324}/s_{1423}$ are regular on $U$ and generate the algebra of $\SL_{4, \C}$-invariant elements. Evaluating at $x$ one obtains:
\begin{align*}
\frac{s_{1234}}{s_{1423}}(x) &= \det \begin{pmatrix} \id & A_1 \\ A_2 & \id\end{pmatrix} = \det(A_1 A_2 - \id),\\
\frac{s_{1324}}{s_{1423}}(x) &= \det(A_1 A_2).
\end{align*}

(2) If $A_1 A_2$ or $A_2 A_1$ is zero, then the orbit of $x$ is closed if and only if $A_1$ and $A_2$ are zero. 

Suppose $A_1A_2, A_2 A_1$ non-zero and set $A = A_1 A_2$.

If $A_1$ is invertible, then $(\id, A)$ belongs to the orbit of $(A_1, A_2)$: it suffices to take $M_1 = A^{-1}$ and $M_2 = \id$. The stabilizer of $(\id, A)$ in $P_{12}$ is given by couples $(M_1, M_2)$ with $M_1 = M_2$ and $M_1 A = AM_1$. If the eigenvalues of $A$ are distinct, the stabilizer has dimension $\le 1$; on the contrary if they are equal the dimension is $\le 2$. In both cases the equality is attained if and only if $A$ is semi-simple. This concludes the proof when $A$ is invertible.

Suppose $A_1, A_2$ are of rank $1$. Then $A_1$, $A_2$ are automatically semi-simple (their eigenvalues are distinct). It suffices to show that the stabilizer of the couple $(A_1, A_2)$ has dimension $2$. 

Up to conjugating $A_1$ by an invertible matrix $N$ (\emph{i.e.} taking $M_1 = M_2 = N$) and up to scalar factors, one may suppose
$$ A_1 = \begin{pmatrix} 1 & 0 \\ 0 & 0 \end{pmatrix}.$$
Since $\rk(A_2) = 1$ and $A_1 A_2\neq 0$ one may write
$$ A_2 = 
\begin{pmatrix}
a_1 & a_2 \\ \lambda a_1 & \lambda a_2
\end{pmatrix},
$$
with $(a_1, a_2) \in \C^2$ non-zero and $\lambda \in \C$. The hypothesis $A_2 A_1 \neq 0$ implies $a_1 \neq 0 $ therefore up to dividing by $a_1$ one may suppose
$$ A_2 = 
\begin{pmatrix}
1 & a \\ \lambda & \lambda a
\end{pmatrix},
$$
for some $a \in \C$. For $\alpha = 1, 2$ write
$$ 
M_\alpha = \begin{pmatrix}
m^{(\alpha)}_{11} & m^{(\alpha)}_{12} \\ m^{(\alpha)}_{21} & m^{(\alpha)}_{22}
\end{pmatrix}.
$$
The condition $M_1 A_1 = A_1 M_2$ implies
\begin{align*}
m^{(1)}_{11} &= m^{(2)}_{11}, & m^{(1)}_{21} & = 0, & m^{(2)}_{12} & = 0,
\end{align*}
whereas $M_2 A_2 = A_2 M_1$ infers
\begin{align*}
m^{(1)}_{12} &= a (m^{(1)}_{11} - m^{(1)}_{22}) & m^{(2)}_{21} &= \lambda (m^{(2)}_{11} - m^{(2)}_{22}).
\end{align*}
Imposing the condition $\det M_1 M_2 = 1$ one obtains that $M_1, M_2$ are of the form
\begin{align*}
M_1 & = \begin{pmatrix} e & a (e - f) \\ 0 & f\end{pmatrix}, &
M_2 &=  \begin{pmatrix} e & 0 \\ \lambda (e - g) & g\end{pmatrix},
\end{align*}
with $e,f, g \in \C^\times$ such that $e^2fg = 1$. In particular the stabilizer of $(A_1, A_2)$ has dimension $2$.
\end{proof}

\subsection{Real forms} Let $G$ be a real form of $\SL_{4, \C}$ admitting a parabolic subgroup of type $(2, 2)$ defined over $\R$. Let $F$ the associated flag variety and $X = F^4$ and $L$ the line bundle induced by the Pl\"ucker embedding of $X_\C$. Let $X^\ss$ be the set of semi-stable points of $X$ with respect to $G$ and $L$. The GIT quotient of $X^\ss$ by $G$ is the projective space $ \P(\HH^0(X, L)^{G})$ together with the natural map
$$ \pi \colon X^\ss \too \P(\HH^0(X, L)^{G}). $$

\subsubsection{$\SL_4(\R)$} Let $G(\R) = \SL_4(\R)$ and $F$ be the grassmannian $\Gr(2, 4)$ over the real numbers. Consider the Pl\"ucker embedding
$$ i \colon \Gr(2, 4) \too \P(\textstyle \bigwedge^2 \R^4).$$
The global sections $s_{1234}, s_{1324}$, $s_{1423}$ are defined over $\R$, that is, they are $G$-invariant global sections of $L$. The quotient map $\pi$ is thus given by
\begin{eqnarray*}
\pi \colon X^\ss &\too& \P^2_\R \\
x &\longmapsto &[s_{1234}(x) : s_{1324}(x) : s_{1423}(x)].
\end{eqnarray*}

\begin{proposition} \label{Prop:Quotient4PlanesSplitRealForm} The map $\theta\colon X^\ss(\R) / \SL_4(\R) \to \P^2(\R)$ is a homeomorphism.
\end{proposition}

\begin{proof} It suffices to prove that $\theta$ is bijective. This can be checked on affine charts, so that, thanks to isomorphism given by Proposition \ref{Prop:4PlanesAffineChart}, this boils down to proving that the map
\begin{eqnarray*} 
\M_2(\R) \times \M_2(\R) &\too& \R^2, \\
(A_1,A_2) &\longmapsto & (\det(A_1 A_2 -\id), \det(A_1 A_2)),
\end{eqnarray*}
induces a bijection $ (\M_2(\R) \times \M_2(\R)) / H(\R) \to \R^2$ where
$$ H = \{ (M_1, M_2) \in \GL_{2, \R}^2 : \det M_1M_2 = 1\}$$
is acting on $M_{2, \R} \times M_{2, \R}$ by  $ (M_1, M_2) \cdot (A_1, A_2) = (M_1 A_1 M_2^{-1}, M_2 A_2 M_1^{-1})$. 

The map is clearly surjective, thus only the injectivity is left to show. By Proposition \ref{Prop:4PlanesAffineChart} (2) the orbit of $(A_1, A_2)$ is closed if and only if $A_1, A_2$ are semi-simple and $\rk(A_1) = \rk(A_2)$. 

If $A_1 = A_2 = 0$ there is nothing to prove.

If $\rk(A_1) = \rk(A_2) = 1$ then it follows from the proof of Proposition \ref{Prop:4PlanesAffineChart} that the stabilizer in $H$ of the couple $(A_1, A_2)$ isomorphic to $\mathbb{G}_{m, \R}^2$. Denote by $T$ the orbit of $(A_1, A_2)$ under $H$ (as real algebraic variety). Then, Theorem \ref{thm:Satz90} implies $\HH^1(\R, \mathbb{G}_{m, \R}^2) = 0$ thus $T(\R) / H(\R) = \{ \ast \}$ by Proposition \ref{Prop:RealPointsOrbitGaloisCohomology}.

If $A_1, A_2$ are invertible, then the couple $(A_1, A_2)$ belongs to the orbit of $(\id, A_1 A_2)$. Setting $A = A_1 A_2$ and $M_1 = M_2$, one is led back to understand the separated quotient of $\M_2(\R)$ by conjugation by the subgroup of matrices in $\GL_2(\R)$ of determinant $\pm 1$. The orbit of a matrix $A \in \M_2(\R)$ under $\GL_2(\R)$ is closed if and only if $A$ is semi-simple. Moreover, semi-simple matrices are classified up to $\GL_2(\R)$-conjugacy by their characteristic polynomial. Therefore, if $A, B \in \M_2(\R)$ verify  $\det A = \det B$, $\Tr A = \Tr B$ then there exists $g \in \GL_2(\R)$ such that $B = g A g^{-1}$. 

Set $g' \df |\det g|^{- 1/2} g$. Then $B = g' A g'^{-1}$ and $|\det g'| = 1$. This concludes the proof.
\end{proof}

\subsubsection{$\SU(2, 2)$} Let $h$ be a hermitian form of signature $(2, 2)$. Up to changing coordinates we may assume that it is given by the matrix
$$
\begin{pmatrix}
0 & 0 & 0 & 1 \\
0 & 0 & 1 & 0 \\
0 & 1 & 0 & 0 \\
1 & 0 & 0 & 0 
\end{pmatrix}.
$$

Let $\SU(2, 2)$ be the subgroup of $\SL_4(\C)$ respecting this hermitian form and let $G$ be the real form of $\SL_{4, \C}$ such that $G(\R) = \SU(2, 2)$. Let $F$ be the flag variety of $G$ corresponding to the grassmannian $\Gr(2, 4)$. The involution $\sigma$ defining $F$ associates to a complex plane $W$ in $\C^4$ its orthogonal $W^\bot$ with respect the hermitian form $h$. The fixed points are the isotropic planes of $\C^4$.  Consider the isotropic planes
$$ W_1 = \langle e_1, e_2 \rangle, \qquad W_2 = \langle e_3, e_4\rangle. $$
For $\alpha = 1, 2$ let $P_{\alpha}$ the parabolic subgroup of $G$ stabilizing $x_\alpha$. The Lie algebra of $\rad^u P_\alpha$ can be identified with the real vector space of $2 \times 2$ matrices of the form
$$ \begin{pmatrix}
a & ib \\ ic & - \bar{a}
\end{pmatrix}, $$
with $a \in \C$, $b, c \in \R$ and $i^2 = -1$. 

\begin{lemma} The global sections $s_{1234}$, $s_{1324}$ and $s_{1423}$ are defined over $\R$.
\end{lemma}

\begin{proof} It suffices to verify the following statement: for $\alpha = 1, 2$, $A_\alpha \in \Lie \rad^u P_\alpha$ and $x = (W_1, W_2, \exp(A_2)W_1, \exp(A_1)W_2)$ then the complex numbers 
\begin{align*}
\frac{s_{1234}}{s_{1423}}(x) &= \det (A_1 A_2 - \id) = \det(A_1A_2) - \Tr(A) + 1, \\
\frac{s_{1324}}{s_{1423}}(x)&= \det(A_1A_2),
\end{align*}
are real. For $\alpha =1, 2$ write
$$ A_\alpha = \begin{pmatrix}
a_\alpha & ib_\alpha \\ ic_\alpha & - \bar{a}_\alpha
\end{pmatrix}, $$
with $a_\alpha \in \C$, $b_\alpha, c_\alpha \in \R$. Then,
$$
A_1 A_2 = \begin{pmatrix}
a_1 a_2 - b_1 c_2 & i(a_1 b_2 - b_1 \bar{a}_2) \\
i(c_1 a_2 - \bar{a}_1 c_2) & \bar{a}_1 \bar{a}_2 - b_2 c_1
\end{pmatrix},
$$
and
\begin{align*}
\Tr(A_1A_2) &= 2 \Re(a_1 a_2) - b_1 c_2 - b_2 c_1, \\
 \det(A_1 A_2) &= (|a_1|^2 - b_1 c_1)(|a_2|^2 - b_2c_2), 
\end{align*}
which are real.
\end{proof}

Let $H$ be the real form of $\GL_{2, \C}$ such that $H(\R)$ is the subgroup of elements respecting the hermitian form on $\C^2$ given by the matrix
$$ 
\begin{pmatrix}
0 & 1 \\ 1 & 0
\end{pmatrix}.
$$

\begin{proposition} \label{Prop:Quotient4PlanesSU(2,2)} The map $\theta\colon X^\ss(\R) / \SU(2, 2) \to \P^2(\R)$ is a homeomorphism.
\end{proposition}

\begin{proof} It suffices to verify the statement on the chart $\{ t_2 \neq 0\}$ in $\P^2(\R)$. Surjectivity is easy. Indeed, take
\begin{align*}
A_1 &= \begin{pmatrix} 1 & 0 \\ 0 & -1\end{pmatrix},
&A_2 &= \begin{pmatrix} a & ib \\ ic & - \bar{a}\end{pmatrix},
\end{align*}
with $a \in \C$ , $b, c \in \R$ and $i^2 = -1$. Then by Proposition \ref{Prop:4PlanesAffineChart} the projection of the semi-stable point
$x = (W_1, W_2, \exp(A_2) W_1, \exp(A_1)W_2)$ is
$$ \pi(x) = [|a|^2 - bc - 2 \Re(a) +1 : |a|^2 - bc : 1]. $$
Since there are no constraints on $a, b, c$, the map is seen to be surjective.

For the injectivity one is  led back to the situation discussed in the proof of Proposition \ref{Prop:Quotient4PlanesSplitRealForm} because the subgroup 
$$ \{ g \in H(\R) : \det g = \pm 1\}, $$
is conjugated in $\GL_2(\C)$ to the subgroup
$$ \{ g \in \GL_2(\R) : \det g = \pm 1\}. $$
The remaining details are left to the reader.
\end{proof}

\subsubsection{$\SL_2(\H)$} Let $\H$ be the quaternions and let $i, j, k$ be generators satisfying the usual relations. Let $\P^1(\H)$ be the set of rank $1$ right $\H$-modules of $\H^2$. For a non-zero $(x, y) \in \H^2$ the right $\H$-module that it generates is denoted $[x : y]$. The group $\GL_2(\H)$ acts on $\P^1(\H)$ on the left: for
$$g = \begin{pmatrix} a & b \\ c & d\end{pmatrix} \in \GL_2(\H),$$
and $(x, y) \in \H^2$ non-zero,
$$ g \cdot [x : y] = [ax + by : cx + dy].$$

Let $G$ the real form of $\SL_4(\C)$ such that $G(\R) = \SL_2(\H)$. Let $F$ the flag variety of $G$ corresponding to the Grassmannian of planes in $\C^4$. The anti-holomorphic involution $\sigma$ defining $F$ associates to a complex plane $ \langle x,  y \rangle$ the plane $\langle \tau(x), \tau(y)  \rangle,$
where
$$ \tau(x_1, x_2, x_3, x_4) = (- \bar{x}_2, \bar{x}_1, - \bar{x}_4, \bar{x}_3).$$
Fixed points under the anti-holomorphic involution $\sigma$ are planes of $\C^4$ of the form $ \langle x, \tau(x) \rangle$ for a non-zero $x \in \C^4$. The correspondence between real points of $F$ and points of $\P^1(\H)$ is given by
\begin{eqnarray*}
F(\R) & \too & \P^1(\H) \\
(x_1, x_2, x_3, x_4) & \longmapsto & [ x_1 +  j x_2 : x_3 + j x_4].
\end{eqnarray*}

\begin{proposition} The sections $s_{1234}, s_{1324}, s_{1423}$ are defined over $\R$.
\end{proposition}

\begin{proof} This follows from the identity
$$
\left|
\begin{matrix}
- \bar{x}_2 & - \bar{y}_2 & - \bar{x}_2' & - \bar{y}_2' \\
 \bar{x}_1 & \bar{y}_1 & \bar{x}_1' & \bar{y}_1' \\
- \bar{x}_4 & - \bar{y}_4 & - \bar{x}_4' & - \bar{y}_4' \\
 \bar{x}_3 & \bar{y}_3 & \bar{x}_3' & \bar{y}_3'  
\end{matrix}
\right| =
\ol{
\left|
\begin{matrix}
x_1 & y_1 & x_1' & y_1' \\
x_2 & y_2 & x_2' & y_2' \\
x_3 & y_3 & x_3' & y_3' \\
x_4 & y_4 & x_4' & y_4' 
\end{matrix}
\right|
},
$$
for $x, y, x', y' \in \C^4$.
\end{proof}

The map $\pi \colon \P^1(\H)^{4, \ss} \to \P^2(\R)$ is therefore given by
$$ \pi(x) = [s_{1234}(x) : s_{1324}(x) : s_{1423}(x)]. $$

\begin{proposition} A quadruple of points in $\P^1(\H)$ is semi-stable if and only if no point is repeated more than twice.
\end{proposition}

\begin{proof} Let $L_1, L_2 \subset \H^2$ be rank $1$ right $\H$-modules and let $W_1, W_2 \subset \C^4$ be the corresponding complex planes. If $W_1, W_2$ meet then $L_1 \cap L_2$ is a right $\H$-module of rank $1$. Therefore $L_1$ and $L_2$ coincide.

The previous observation implies that the section $s_{1234}$ vanishes at a quadruple $x = (L_1, \dots, L_4)$ of points in $\P^1(\H)$ if and only if $L_1 = L_2$ or $L_3 = L_4$. The statement follows immediately.
\end{proof}
Consider the points $x_1 = [1: 0]$ and $x_2 = [0: 1]$.  For $\alpha = 1, 2$ let $P_{\alpha}$ the parabolic subgroup of $G$ stabilizing $x_\alpha$. Identify $\Lie \rad^u P_\alpha$ with $\H$. Then the quotient map given by Proposition \ref{Prop:4PlanesAffineChart} induces the following map on real points:
\begin{eqnarray} \label{eq:ExplicitExpressionQuotient4PlanesHyperbolic} \H \times \H &\too& \P^2(\R), \\
(q_1, q_2) &\longmapsto& [|q_1 q_2 - 1|^2 : |q_1 q_2|^2 :1]. \nonumber
\end{eqnarray}

\begin{proposition} \label{Prop:Quotient4planesSL(2,H)} With the notations introduced above:
\begin{enumerate}
\item the image of $\pi$ is the subset
$$ S = \{ [t_0 : t_1 : t_2] \in \P^2(\R) : t_0^2 + t_1^2 +t_2^2 \le 2 (t_0 t_1 + t_0 t_2 +  t_1 t_2) \};$$
\item the induced map $\theta \colon \P^1(\H)^{4, \ss} / \SL_2(\H) \to S$ is a homeomorphism;
\item a quadruple $(x_1, \dots, x_4) \in X^\ss(\R)$ is made of pairwise distinct points if and only if 
$$ \pi(x_1, \dots, x_4) \in \{ [t_0 : t_1 : t_2] \in \P^2(\R) : t_0 t_1 t_2 \neq 0\}.$$
\end{enumerate}
\end{proposition}

\begin{figure}[h]

\begin{tikzpicture}
\def\r{4.5};
    \node[anchor=south west,inner sep=0] at (-4.5,-4.5) {\includegraphics[height=9cm]{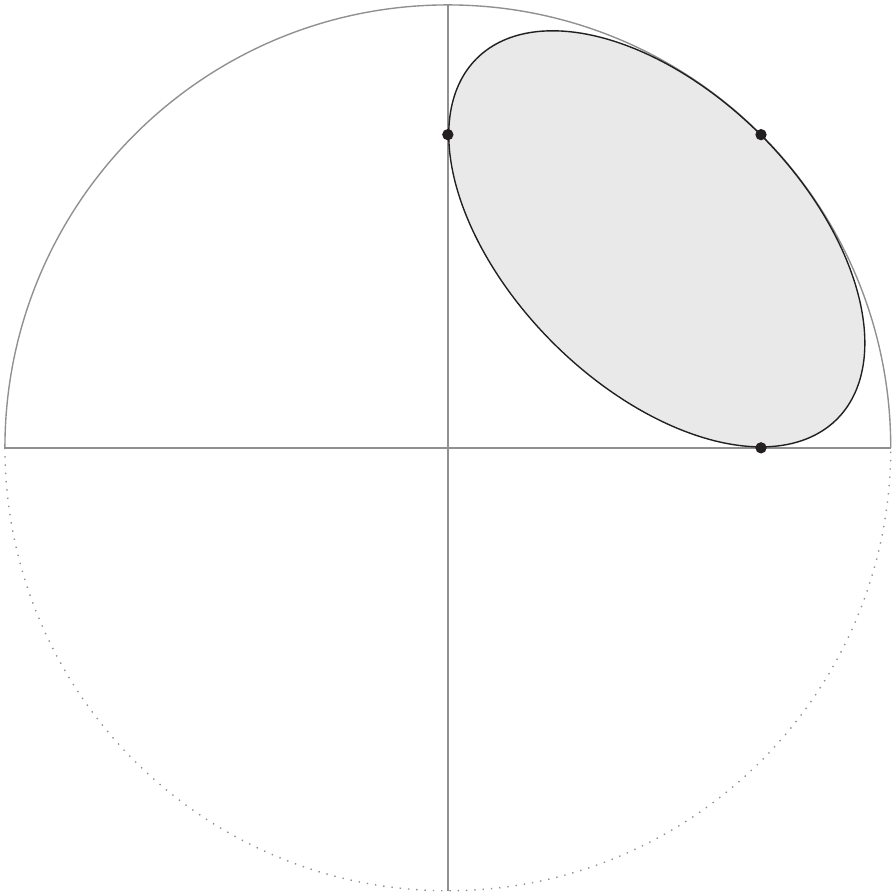}};
    
    \node[anchor=north] at (\r/1.4142,0) {$[1:0:1]$};
    \node[anchor=east] at (0,\r/1.4142) {$[1:1:0]$};
    \node[anchor=west] at (\r/1.35,\r/1.35) {$[0:1:1]$};
    
    \node at (\r/2,\r/2) {$ \P^1(\H)^{4, \ss} / \SL_2(\H)$};
\node[anchor=south] at (0,\r) {$\P^2(\R)$};
    \end{tikzpicture}
    
\caption{Through the map $\theta$ one can see the separated quotient $\P^1(\H)^{4, \ss} / \SL_2(\H)$ as the filled conic in the figure. It is a parabola in the complement of each coordinate axis. The three points of tangency correspond to the degenerate quadruples. Points on the boundary of the conic correspond to configurations whose quaternionic cross-ratio is real.}
  
  \end{figure}
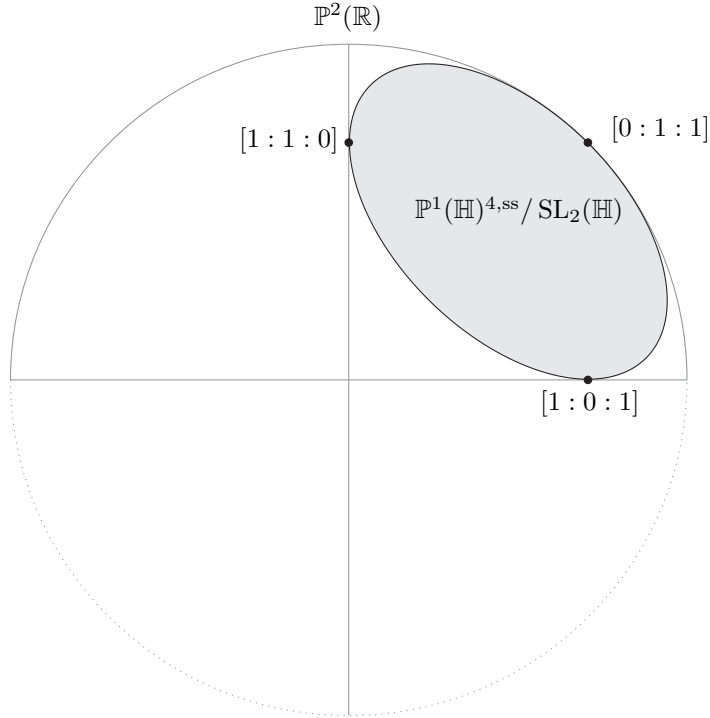

\begin{proof} It suffices to verify the assertions (1) and (2) on the affine chart $\{ t_2 \neq 0\}$. 

(1) For $q_1, q_2 \in \H$ set $q = q_1 q_2$ and  $ [t_0 : t_1 : t_2] = [|q-1|^2:|q|^2:1].$ Then,
$$
\left( \frac{t_1 + t_2 - t_0}{t_0} \right)^2 = 4 \Re(q)^2 \le 4 |q|^2 = \frac{t_1}{t_2}.
$$
In order to see that $\pi_0$ is onto the subset $S$ it suffices to consider $(q_1, q_2) =(1, q)$.

(2) It suffices to prove that $\pi$ is injective.  If $q_1 = 0$ the unique closed orbit contained in the closure of the orbit $(q_1, q_2)$ in $\H^2$ is $(0, 0)$. If $q_1 \neq 0$, then $(q_1, q_2)$ lies in the same orbit of $(1, q_1q_2)$. Set $q = q_1 q_2$. There exists $h \in \H^\times$ and $\lambda \in \C$ (unique up to complex conjugation) such that $h \lambda h^{-1} = q$.\footnote{This can be seen as follows. Write $q$ as a matrix 
$$ A= \begin{pmatrix} a & - \bar{b} \\ b & \bar{a}\end{pmatrix}, $$
with $a, b \in \C$. The matrix $A$ commutes with its adjoint, thus there is an orthonormal basis of $\C^2$ made of eigenvectors. In other words there exists $U \in \U(2)$ and $\lambda \in \C$ such that
$$ U A U^{-1} = \begin{pmatrix} \lambda & 0 \\ 0 & \bar{\lambda}\end{pmatrix}.$$} On the other hand,
\begin{align*} 
|q|^2 &= |\lambda|^2, &|q-1|^2 &= |\lambda|^2 - 2 \Re(\lambda) + 1,
\end{align*}
thus $\pi(q_1, q_2)$ determines $\lambda$ up to complex conjugation.

(3) is clear.
\end{proof}

The map $\pi$ is related to the notion of quaternionic cross-ratio. Let $x_1, \dots, x_4$ be four pairwise distinct points of $\P^1(\H)$. Then there exists $g \in \GL_2(\H)$ such that 
\begin{align*}
 g \cdot x_1 &= [0:1], & g \cdot x_2 &= [1:1], & g \cdot x_3 &= [1:0].
 \end{align*}
Let $q \in \H^\times$ be such that $ g \cdot x_4 = [q:1]$. Over the complex numbers one defines the cross-ratio of $x_1, \dots, x_4$ as $q$. However in this situation $q$ is not well-defined. Indeed, the subgroup  $\{ g \in \GL_2(\H) : g x_i = x_i \textup{ for } i = 1, 2, 3\}$ is the subgroup of matrices $h \cdot \id$ with $h \in \H^\times$ and
$$ [hq : h] = [hqh^{-1}:1]$$
as right $\H$-modules of rank $1$. The quaternion $q$ is thus well-defined only up to conjugation.

\begin{definition} \label{definition:quaternionicCross}
The \emph{quaternionic cross-ratio} of $x_1, \dots, x_4$ is the conjugacy class $[q]$ of the quaternion $q$.
\end{definition}

In the literature the quaternionic cross-ratio is defined in various manners, depending on the way the quotient $\H^\times / \H^\times$ is presented ($\H^\times$ acting on itself by conjugation). For instance, the map
\begin{eqnarray*}
\H^\times / \H^\times &\too & \R^2,\\
\left[q\right] &\longmapsto & (|q|, \Re q),
\end{eqnarray*}
is well-defined (the modulus and the real part of a quaternion are invariant under conjugation) and injective\footnote{This is essentially the proof of (2) in the preceding Proposition.}. In \cite{GwynneLibine} the quaternionic cross-ratio is defined as the image of the quaternionic cross-ratio in the present sense through the preceding map.

\begin{proposition} Let $x_1, \dots, x_4$ be four distinct points of $\P^1(\H)$ and let $[q]$ their quaternionic cross-ratio. Then,
$$ \pi(x_1, \dots, x_4) = [|q-1|^2: |q|^2:1]. $$
The quaternionic cross-ratio is real if and only if
$$ \pi(x_1, \dots, x_4) \in \{ [t_0 : t_1 : t_2] \in \P^2(\R) : t_0^2 + t_1^2 +t_2^2 = 2 (t_0 t_1 + t_0 t_2 +  t_1 t_2) \}.
$$
\end{proposition}

This shows in particular that the notion of quaternionic cross-ratio can be extended as a function taking values in $\P^2(\R)$ and defined on the set of quadruples of points in $\P^1(\H)$ whose entries are not repeated more than twice.

\begin{proof} Since the quaternionic cross-ratio is invariant under the action of $\GL_2(\H)$ it suffices to verify the statement for the quadruple
\begin{align*}
 x_1 &= [0:1], & x_2 &= [1:1], & x_3 &= [1:0], & x_4 &= [q:1].
 \end{align*}
In this case the formula follows from the expression given in \eqref{eq:ExplicitExpressionQuotient4PlanesHyperbolic}.

For the second part remark that $q$ is real if and only if $|q| = |\Re(q)|$. Then the statement follows from the equality $|q - 1|^2 = |q|^2 - 2 \Re(q) + 1$.
\end{proof}

The injectivity of $\theta \colon X^\ss(\R) / \SL_2(\H) \to \P^2(\R)$ and the previous expression for $\pi$ prove:

\begin{corollary}[{\cite[Theorem 8]{GwynneLibine}}] Let $(x_1, \dots, x_4)$, $(y_1, \dots, y_4)$ be quadruplets of pairwise distinct points in $\P^1(\H)$. 

Then there exists $g \in \SL_2(\H)$ such that $g \cdot x_\alpha = y_\alpha$ for $\alpha = 1, 2, 3, 4$ if and only if their quaternionic cross-ratio coincide.
\end{corollary}

\section{Coordinates for generic configurations}\label{section:coordiantesgeneric}

Cross ratios and triple ratios for complete flags are classical invariants of flags and have been used by Fock and Goncharov (\cite{FG}) to describe higher Teichm\"uller spaces for representations of surface groups. They have also been used to study representations of fundamental groups of three manifolds in (see \cite{BFG}, \cite{GGZ}, \cite{DGG}).

In this section cross ratios for complete flags are related to the invariants introduced in the preceding sections of the associated partial flags. Explicit computations with these coordinates allow us to obtain  representations of fundamental groups of three manifolds with values in real forms of $\SL(n,\C)$ or $\PGL(n,\C)$. 

\subsection{Generic flags}
Let $n, r \ge 1$ be positive integers, $\Fl(V)$  the variety of complete flags of $V =\C^n$ and $X = \Fl(V)^r$. 

Let $ \O(1, \dots, 1)$ be the line bundle $\pr_1^\ast \O(1) \otimes \cdots \otimes \pr_{n-1}^\ast \O(1)$ on the variety 
$$\prod_{i = 1}^{n-1}\P(\textstyle \bigwedge^i V),$$
where $\pr_i$ is the projection onto the $i$-th factor. The semi-simple group $G = \SL(V)$ acts on $X$. Consider the linearization of this action given by the embedding
$$ \iota \colon X \too \left(\prod_{i = 1}^{n-1}\P(\textstyle \bigwedge^i V) \right)^r.$$
This linearization corresponds to the $G$-linearized line bundle
$$ L = \iota^\ast (\pr_1^\ast \O(1, \dots, 1) \otimes \cdots \otimes \pr_r^\ast \O(1, \dots, 1)),$$
where $\pr_i \colon X \to F$ denotes the projection onto the $i$-th factor.

\begin{definition} Let $r \ge 1$ be a positive integer and for $i = 1, \dots, r$ let 
$$F^{(i)} : F^{(i)}_0 = 0 \subset F^{(i)}_1 \subset F^{(i)}_2 \subset \cdots \subset F^{(i)}_{n - 1} \subset F^{(i)}_n = V,$$
be a complete flag of $V$. The configuration $(F^{(1)}, \dots, F^{(r)})$ is said to be \emph{generic} if, for all $\alpha_1, \dots, \alpha_r = 0, \dots, n$ such that $\alpha_1 + \cdots + \alpha_r = n$,
$$ \dim_\C ( F^{(1)}_{\alpha_1} + \cdots + F^{(r)}_{\alpha_r}) = n.$$
The subset of $X$ made of generic $r$-tuples is denoted $X^{\gen}$.
\end{definition}

\subsection{Generic invariants} Let $A_{n, r}$ be the set of $r$-tuples $\alpha = (\alpha_1, \dots, \alpha_r)$ made of integers $\alpha_i = 0, \dots, n$ such that $\alpha_1 + \cdots + \alpha_r = n$. Consider the product of Grassmannians
$$ X_\alpha = \Gr_{\alpha_1}(V) \times \cdots \times \Gr_{\alpha_r}(V),$$
and the Pl\"ucker embedding $ \iota_\alpha \colon X_\alpha \too \P(\textstyle \bigwedge^{\alpha_1} V) \times \cdots \times \P(\bigwedge^{\alpha_r} V)$. Denote by $L_\alpha$ the line bundle
$$ \iota_\alpha^\ast( \pr_1^\ast \O(1) \otimes \cdots \otimes \pr_r^\ast \O(1))$$
on $X_\alpha$. The semi-simple group $G = \SL(V)$ acts naturally on $X_\alpha$, the Pl\"ucker embedding $\iota_\alpha$  is $G$-equivariant and $L_\alpha$ is an ample $G$-linearized line bundle. The wedge product induces a linear map
\begin{eqnarray*}
s_\alpha \colon \textstyle \bigwedge^{\alpha_1} V \otimes \cdots \otimes \bigwedge^{\alpha_r} V &\too & \textstyle \bigwedge^n V \\
t_{\alpha_1} \otimes \cdots \otimes t_{\alpha_r} &\longmapsto & t_{\alpha_1} \wedge \cdots \wedge t_{\alpha_r} .
\end{eqnarray*}
Identify $\bigwedge^n V$ with $\C$ by fixing $e_1 \wedge \cdots \wedge e_n$ as basis. Then $s_\alpha$ can be seen as $G$-invariant linear form on $\textstyle \bigwedge^{\alpha_1} V \otimes \cdots \otimes \bigwedge^{\alpha_r} V$, thus a $G$-invariant global section of the line bundle $L_\alpha$. For $\alpha \in A_{n, r}$ let $\pr_\alpha \colon X \to X_\alpha$ be the projection
$$ \pr(F) = (F^{(1)}_{\alpha_i}, \dots, F^{(r)}_{\alpha_r}).$$

Let $X^\ss$ be the open subset of semi-stable points of $X$ with respect the action of $G$ and the linearization $L$ and $\pi \colon X^\ss \to Y$ be the GIT quotient of $X^\ss$ by $G$.

\begin{proposition} A generic $r$-tuple of complete flags $F$ is a semi-stable point of $X$ with respect the polarization $L$.
\end{proposition}

\begin{proof} We prove the statement when $r$ is even, leaving the odd case to the reader. Write $r = 2r'$ and, for $i = 1, \dots, r'$ and $\ell = 1, \dots, n-1$ let
$$
\alpha_{i \ell} = (0, \dots, \ell, 0, \dots, n-\ell, 0, \dots, 0),
$$
where the $\ell$ is the $i$-th position and $n- \ell$ in the $(r' + i)$-th position. Set also
$$ 
\beta_{i \ell} = (0, \dots, n- \ell, 0, \dots, \ell, 0, \dots, 0),
$$
with the same conventions. Consider the global section
$$ s_{i \ell} = \pr_{\alpha_{i \ell}}^\ast s_{\alpha_{i \ell}} \otimes \pr_{\beta_{i \ell}}^\ast s_{\beta_{i \ell}},$$
of the line bundle $\pr_{\alpha_{i \ell}}^\ast L_{\alpha_{i \ell}} \otimes \pr_{\beta_{i \ell}}^\ast L_{\beta_{i \ell}}$ on $X$. The global section $s_{i \ell}$ does not vanish at $F$ as
\begin{align*}
F^{(i)}_\ell \cap F^{(r' +i)}_{n - \ell} = 0, && F^{(i)}_{n -\ell} \cap F^{(r' +i)}_{\ell} = 0.
\end{align*}
Therefore
$$ s = \bigotimes_{\ell = 1}^{n-1} \bigotimes_{i = 1}^{r'} s_{i \ell},$$
is a $G$-invariant global section of $L$ which does not vanish at $F$. In particular, $F$ is semi-stable.
\end{proof}

Let $\Fl_{1, n-1}(V)$ be the variety of couples $(L, H)$ made of a line $L$ and a hyperplane $H$ containing it and $X_{1, n-1} = \Fl_{1, n-1}(V)^r$. Let $X_{1, n-1}^\ss$ be the open subset of semi-stable points of $X_{1, n-1}$ under the action of $G$ and with respect to the Pl\"ucker embedding (\emph{cf.} Section \ref{par:QuadruplesLineHyperplane}) and $\pi_{1, n-1} \colon X_{1, n-1}^\ss \to Y_{1, n-1}$ the GIT quotient of $X_{1, n-1}^\ss$ by $G$.

\subsection{Triplets of flags} Suppose $r = 3$. Let $F = (F^{(1)},F^{(2)}, F^{(3)}$) be a generic triple of complete flags. For $\alpha \in A_{n, 3}$ let $\Delta_{\alpha} := \pr_\alpha^\ast s_\alpha$.

\begin{definition} (see \cite{FG} pg. 133) For a triple $\alpha = (\alpha_1, \alpha_2, \alpha_3)$ of integers $\alpha_i = 1, \dots, n-1$ such that $\alpha_1 + \alpha_2 + \alpha_3 = n$ define
$$
 t_{\alpha}:=\frac{\Delta_{\alpha + (0, -1, 1)} \otimes \Delta_{\alpha +(1, 0, -1)} \otimes \Delta_{\alpha + (-1, 1, 0)}}{\Delta_{\alpha + (1, -1, 0)} \otimes  \Delta_{\alpha +(0, 1, -1)} \otimes \Delta_{\alpha + (-1, 0, 1)}}.
$$
Remark that both the numerator and the denominator of $t_\alpha$ are global sections of the line bundle $L_{\alpha}^{\otimes3}$,  so that $t_\alpha$ is a well-defined rational function $X \dashrightarrow \P^1$ called the \emph{triple ratio} relative to the index $\alpha$.
\end{definition}

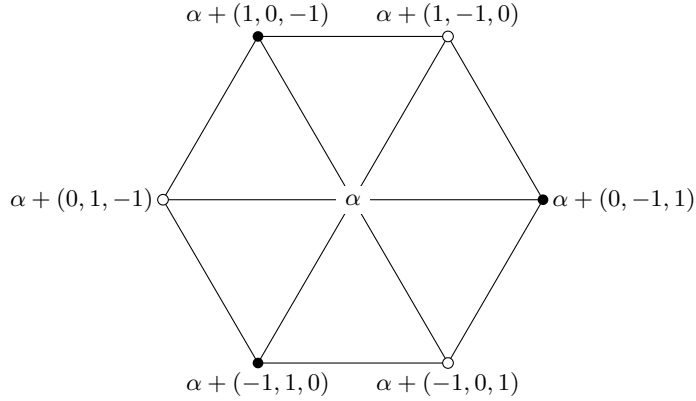
\begin{figure}[h]
\begin{tikzpicture}

\def\sqrtthree{1.7320508075688773};
\def\r{2.5};

\draw (\r, 0) -- (\r/2, \sqrtthree/2*\r) 
                   -- (-\r/2, \sqrtthree/2*\r) 
                   -- (-\r, 0) 
                   -- (-\r/2, -\sqrtthree/2*\r) 
                   -- (\r/2, -\sqrtthree/2*\r)
                   -- cycle;
\draw (-\r, 0) -- (\r, 0);
\draw (\r/2, \sqrtthree/2*\r) -- (-\r/2, - \sqrtthree/2*\r);
\draw (-\r/2, \sqrtthree/2*\r) -- (\r/2, -\sqrtthree/2*\r);

\fill (\r, 0) circle (2pt);

\fill[white] (\r/2, \sqrtthree/2*\r) circle (2pt);
\draw (\r/2, \sqrtthree/2*\r) circle (2pt);

\fill (-\r/2, \sqrtthree/2*\r) circle (2pt);

\fill[white] (-\r, 0) circle (2pt);
\draw (-\r, 0) circle (2pt);

\fill (-\r/2, -\sqrtthree/2*\r) circle (2pt);

\fill[white] (\r/2, -\sqrtthree/2*\r) circle (2pt);
\draw (\r/2, -\sqrtthree/2*\r) circle (2pt);

\node[fill=white] at (0,0) {\small $\alpha$};
\node[anchor=west] at (\r, 0) {\small $\alpha + (0, -1, 1)$};
\node[anchor=south] at (\r/2, \sqrtthree/2*\r) {\small $\alpha + (1, -1, 0)$};
\node[anchor=south] at (-\r/2, \sqrtthree/2*\r) {\small $\alpha + (1, 0, -1)$};
\node[anchor=east] at (-\r, 0) {\small $\alpha + (0, 1, -1)$};
\node[anchor=north] at (-\r/2, -\sqrtthree/2*\r) {\small $\alpha + (-1, 1, 0)$};
\node[anchor=north] at (\r/2, -\sqrtthree/2*\r) {\small $\alpha + (-1, 0, 1)$};

\end{tikzpicture}
\caption{The triple ratio $T_\alpha$.} \label{tripleratio}
\end{figure}

For a triple $\alpha = (\alpha_1, \alpha_2, \alpha_3)$ of integers $\alpha_i = 1, \dots, n-1$ such that $\alpha_1 + \alpha_2 + \alpha_3 = n$ the map $t_\alpha \colon X \dashrightarrow \P^1$ is $G$-invariant, thus factors in a unique way through a rational map on the quotient $t_\alpha \colon Y \dashrightarrow \P^1$. Since there are $\frac{(n-1)(n-2)}{2}$ triple ratios, they induce a rational map
\begin{eqnarray*}
t \colon Y & \dashrightarrow & (\P^1)^{(n -1)(n-2)/2} \\
\left[F \right] & \longmapsto & (t_\alpha(F))_{\alpha}.
\end{eqnarray*}

One can state the fact that the triple ratio parametrizes configurations of three complete flags (see \cite{FG}) as follows:

\begin{proposition}With the notations introduced above, 
\begin{enumerate}
\item $\pi(X^{\gen})$ is an open subset of $Y$;
\item the rational map $t$ is regular on $\pi(X^{\gen})$ and induces an isomorphism 
$$ t \colon \pi(X^{\gen}) \stackrel{\sim}{\too} (\P^1 \smallsetminus \{ 0, \infty\})^{(n-1)(n-2)/2}.$$
\end{enumerate}
\end{proposition}

The GIT quotient of $X_{1, n-1}^\ss$ by $G$ is $\P^1$ and the projection on the quotient  $\pi_{1, n-1}$ is given, for a semi-stable triplet $x = \{ (L_i, H_i)\}_{i = 1, 2, 3}$, by
$$ \pi_{1, n-1}(x) = [\phi_1(v_2)\phi_2(v_3)\phi_3(v_1) : \phi_1(v_3)\phi_3(v_2)\phi_2(v_1)].$$

The natural projection $p_{1, n-1}\colon X \to X_{1, n-1}$ is $G$-equivariant, therefore it induces a rational map between quotients $\tilde{p}_{1, n-1} \colon Y \dashrightarrow \P^1$.

\begin{proposition} With these notations, the rational map $\tilde{p}_{1, n-1} \colon  Y \dashrightarrow \P^1$  is the product of all triple ratios.
\end{proposition}

\begin{proof} It is easy to see that  the product of all triple ratios simplifies to
$$
\frac{\Delta_{(n-1,1,0)} \otimes \Delta_{(0,n-1,1)} \otimes \Delta_{(1,0,n-1)}}{\Delta_{(n-1,0,1)} \otimes \Delta_{(1,n-1,0)} \otimes \Delta_{(0,1,n-1)}}.
$$
Indeed each factor $\Delta_{(i,j,k)}$ appears in the numerator and the denominator the same number of times except for 
the terms above 
which correspond to the nearest points at the three vertices of the triangle as in Figure.

\begin{figure}[h]
\begin{tikzpicture}
\def\sqrtthree{1.73205080757};
\def\r{3};

\draw[dashed] (0,\r) -- (-\sqrtthree/2*\r, -\r/2);
\draw[dashed] (-\sqrtthree/2*\r, -\r/2) -- (\sqrtthree/2*\r, -\r/2);
\draw[dashed]  (\sqrtthree/2*\r, -\r/2) -- (0, \r);

\draw[fill=magenta, fill opacity=0.2]
(\sqrtthree/2*\r - 3*\sqrtthree/8*\r, -\r/2 +3/8*\r)--
(0, -\r/2 +6/8*\r) --
(-\sqrtthree/2*\r + 2*\sqrtthree/8*\r, -\r/2 +6/8*\r) --
(-\sqrtthree/2*\r + \sqrtthree/8*\r, -\r/2 +3/8*\r)--
(-\sqrtthree/2*\r + \sqrtthree/4*\r, -\r/2) --
(-\sqrtthree/2*\r + 2*\sqrtthree/4*\r, -\r/2);

\draw[fill=cyan, fill opacity=0.2]
(-\sqrtthree/2*\r + 2*\sqrtthree/4*\r, -\r/2) --
(-\sqrtthree/2*\r + 3*\sqrtthree/4*\r, -\r/2) --
(\sqrtthree/2*\r - \sqrtthree/8*\r, -\r/2 +3/8*\r) --
(\sqrtthree/2*\r - 2*\sqrtthree/8*\r, -\r/2 +6/8*\r) --
(0, -\r/2 +6/8*\r)--
(-\sqrtthree/2*\r + 3*\sqrtthree/8*\r, -\r/2 +3/8*\r);

\draw[fill=yellow, fill opacity=0.2]
(-\sqrtthree/2*\r + 3*\sqrtthree/8*\r, -\r/2 +9/8*\r) -- 
(\sqrtthree/2*\r - 3*\sqrtthree/8*\r, -\r/2 +9/8*\r) --
(\sqrtthree/2*\r - 2*\sqrtthree/8*\r, -\r/2 +6/8*\r) --
(\sqrtthree/2*\r - 3*\sqrtthree/8*\r, -\r/2 +3/8*\r) --
(-\sqrtthree/2*\r + 3*\sqrtthree/8*\r, -\r/2 +3/8*\r) --
(-\sqrtthree/2*\r + 2*\sqrtthree/8*\r, -\r/2 +6/8*\r) -- 
(-\sqrtthree/2*\r + 3*\sqrtthree/8*\r, -\r/2 +9/8*\r);

\draw (-\sqrtthree/2*\r + \sqrtthree/4*\r, -\r/2) -- (\sqrtthree/2*\r - 3*\sqrtthree/8*\r, -\r/2 +9/8*\r);

\draw (\sqrtthree/2*\r - \sqrtthree/4*\r, -\r/2) -- (-\sqrtthree/2*\r + 3*\sqrtthree/8*\r, -\r/2 +9/8*\r);

\draw (-\sqrtthree/2*\r + 2*\sqrtthree/4*\r, -\r/2) -- (\sqrtthree/2*\r - 2*\sqrtthree/8*\r, -\r/2 +6/8*\r);

\draw (\sqrtthree/2*\r - 2*\sqrtthree/4*\r, -\r/2) -- (-\sqrtthree/2*\r + 2*\sqrtthree/8*\r, -\r/2 +6/8*\r);

\draw (-\sqrtthree/2*\r + 3*\sqrtthree/4*\r, -\r/2) -- (\sqrtthree/2*\r - \sqrtthree/8*\r, -\r/2 +3/8*\r);

\draw (\sqrtthree/2*\r - 3*\sqrtthree/4*\r, -\r/2) -- (-\sqrtthree/2*\r + \sqrtthree/8*\r, -\r/2 +3/8*\r);

\draw (-\sqrtthree/2*\r + 3*\sqrtthree/8*\r, -\r/2 +9/8*\r) -- (\sqrtthree/2*\r - 3*\sqrtthree/8*\r, -\r/2 +9/8*\r);

\draw (-\sqrtthree/2*\r + 2*\sqrtthree/8*\r, -\r/2 +6/8*\r) -- (\sqrtthree/2*\r - 2*\sqrtthree/8*\r, -\r/2 +6/8*\r);

\draw (-\sqrtthree/2*\r + \sqrtthree/8*\r, -\r/2 +3/8*\r) -- (\sqrtthree/2*\r - \sqrtthree/8*\r, -\r/2 +3/8*\r);

\draw[fill=black] (-\sqrtthree/2*\r + 3*\sqrtthree/8*\r, -\r/2 +3/8*\r) circle (2pt);
\draw[fill=black] (\sqrtthree/2*\r - 3*\sqrtthree/8*\r, -\r/2 +3/8*\r) circle (2pt);
\draw[fill=black] (0, -\r/2 +6/8*\r)  circle (2pt);

\end{tikzpicture}
\caption{}
\end{figure}

Observe than that $\Delta_{(n-1,1,0)} = \phi_1(v_2)$ where $\phi_1$ is a linear form defining the hyperplane $F^{(1)}_{n-1}$ and $v_2$ is a generator of the line $F^{(2)}_{1}$. The analogous formulas for the other terms yield the desired result.
\end{proof}

\begin{example}[{$n = 3$, \cite{FG,BFG}}] \label{ex:GenericVsGeneralPosition} In this case there is just one triple ratio (corresponding to the index $(1,1,1)$, see Figure \ref{pic:OneTripleRatio}) and the map $\tilde{p}$ induces an isomorphism $Y \iso \P^1$. 
\begin{figure}[h]
\begin{tikzpicture}
\def\sqrtthree{1.73205080757};
\def\r{3};

\draw[dashed] (0,\r) -- (-\sqrtthree/2*\r, -\r/2);
\draw[dashed] (-\sqrtthree/2*\r, -\r/2) -- (\sqrtthree/2*\r, -\r/2);
\draw[dashed]  (\sqrtthree/2*\r, -\r/2) -- (0, \r);

\draw[pattern=north west lines, pattern color=gray] 
(-\sqrtthree/2*\r + \sqrtthree/3*\r, -\r/2) --
(\sqrtthree/2*\r - \sqrtthree/3*\r, -\r/2) --
(\sqrtthree/3*\r, 0) --
(\sqrtthree/2*\r - \sqrtthree/3*\r, \r/2) --
(-\sqrtthree/2*\r + \sqrtthree/3*\r, \r/2) --
(-\sqrtthree/3*\r, 0) --
(-\sqrtthree/2*\r + \sqrtthree/3*\r, -\r/2);

\draw (-\sqrtthree/2*\r + \sqrtthree/3*\r, -\r/2) -- (\sqrtthree/2*\r - \sqrtthree/3*\r, \r/2);

\draw (-\sqrtthree/2*\r + 2*\sqrtthree/3*\r, -\r/2) -- (\sqrtthree/3*\r, 0);

\draw (\sqrtthree/2*\r - \sqrtthree/3*\r, -\r/2) -- (-\sqrtthree/2*\r + \sqrtthree/3*\r, \r/2);

\draw (\sqrtthree/2*\r - 2*\sqrtthree/3*\r, -\r/2) -- (-\sqrtthree/3*\r, 0);

\draw (-\sqrtthree/3*\r, 0) -- (\sqrtthree/3*\r, 0);

\draw (\sqrtthree/2*\r - \sqrtthree/3*\r, \r/2) -- (-\sqrtthree/2*\r + \sqrtthree/3*\r, \r/2);

\node[anchor=south] at (0,\r) {\footnotesize (3,0,0)};
\node[anchor=east] at (-\sqrtthree/2*\r, -\r/2) {\footnotesize (0,3,0)};
\node[anchor=west] at (\sqrtthree/2*\r, -\r/2) {\footnotesize (0,0,3)};
\node[anchor=north] at (-\sqrtthree/2*\r + \sqrtthree/3*\r, -\r/2) {\footnotesize (0,2,1)};
\node[anchor=west] at (\sqrtthree/2*\r - \sqrtthree/3*\r, \r/2) {\footnotesize (2,0,1)};
\node[anchor=north] at (-\sqrtthree/2*\r + 2*\sqrtthree/3*\r, -\r/2) {\footnotesize (0,1,2)};
\node[anchor=west] at (\sqrtthree/3*\r, 0) {\footnotesize (1,0,2)};
\node[anchor=east] at (-\sqrtthree/2*\r + \sqrtthree/3*\r, \r/2) {\footnotesize (2,1,0)};
\node[anchor=east] at (-\sqrtthree/3*\r, 0) {\footnotesize (1,2,0)};
\node[fill=white] at (0,0) {\footnotesize $(1,1,1)$};

\draw[fill=black] (\sqrtthree/3*\r, 0) circle (2pt);
\draw[fill=black] (-\sqrtthree/2*\r + \sqrtthree/3*\r, -\r/2) circle (2pt);
\draw[fill=black] (-\sqrtthree/2*\r + \sqrtthree/3*\r, \r/2) circle (2pt);
\draw[fill=white] (-\sqrtthree/3*\r, 0) circle (2pt);
\draw[fill=white] (\sqrtthree/2*\r - \sqrtthree/3*\r, \r/2) circle (2pt);
\draw[fill=white] (-\sqrtthree/2*\r + 2*\sqrtthree/3*\r, -\r/2) circle (2pt);
\end{tikzpicture}

\caption{} \label{pic:OneTripleRatio}
\end{figure}

However, for a complete flag being generic is a weaker condition than being in general position. More precisely, for $i = 1, 2, 3$, let  $F^{(i)} = (L_i, H_i)$ be a complete flag of $V = \C^3$. Then the triplet $F = (F^{(1)}, F^{(2)}, F^{(3)})$ is
\begin{itemize}
\item \emph{generic} if and only if $L_1 + L_2 + L_3 = V$ and $L_i \cap H_j = 0$ for all $i \neq j$;
\item in \emph{general position} if and only if it is generic and $H_1 \cap H_2 \cap H_3 = 0$.
\end{itemize}
For $i = 1, 2, 3$ let $v_i$ be a generator of $L_i$ and $\phi_i$ a linear form defining $H_i$. If $F$ is generic, then the triple ratio
$$ t_{(1, 1, 1)}(F) =\frac{\phi_1(v_2)\phi_2(v_3)\phi_3(v_1)}{\phi_1(v_3)\phi_3(v_2)\phi_2(v_1)},$$ 
is a non-zero complex number. If $\dim(L_1 + L_2 + L_3) = 2$ then $t_{(1, 1, 1)}(F) = -1$. For instance consider:
\begin{align*}
L_1 &= \langle e_1\rangle, & L_2 &= \langle e_2\rangle, & L_3& = \langle e_1 + e_2\rangle, \\
H_1 &= \{ x_2 + x_3 = 0\}, & H_2 &= \{ x_1 + x_3 = 0\}, & H_3 &= \{ x_1 - x_2 = 0 \}.
\end{align*}
Then $F$ is generic but not in general position. Moreover, the orbit of $F$ is not closed (\emph{cf.} Proposition \ref{Prop:QuotientLineHyperplaneSingularFiber}).

\end{example}

\subsection{Quadruples of flags} Suppose $r = 4$. For a quadruple $F = (F^{(1)}, \dots, F^{(4)})$ of complete flags of $V$ and $\alpha \in A_{n-2, 4}$ consider the following subspace of $V$:
$$
E_\alpha(F) = F^{(1)}_{\alpha_1}+ F^{(2)}_{\alpha_2} + F^{(3)}_{\alpha_3} + F^{(4)}_{\alpha_4}.
$$
If the quadruple $F$ is generic, then
$$ \dim_\C E_\alpha(F) = \sum_{i = 1}^4 \dim_\C F^{(i)}_{\alpha_i} = \alpha_1 + \cdots + \alpha_4 = n-2.$$
For $i = 1, \dots, 4$ the vector subspace of $V / E_\alpha(F)$,
$$ L_i = (F^{(i)}_{\alpha_i + 1} + E_\alpha(F)) / E_\alpha(F),$$
is  of dimension $1$. Let $\chi_\alpha(F) \in \P^1(\C)$ be the cross-ratio of the points 
$$ [L_1], \dots, [L_4] \in \P(V / E_\alpha(F)).$$
This defines a rational map $\chi_\alpha \colon X \dashrightarrow \P^1$. Since $\chi_\alpha$ is $G$-invariant it induces a rational map on the quotient $\chi_\alpha \colon Y \dashrightarrow \P^1$. Putting together all these rational maps, one obtains a rational map
$$ \chi \colon Y \dashrightarrow (\P^1)^{\# A_{n-2, 4}}.$$
In the literature the functions $\chi_\alpha$ are often denote by $z_\alpha$. Here we adopted the letter $\chi$ in order to avoid confusion in the case of complete flags in a space of dimension $3$ (Example \ref{ConfigurationsTripletsGeneric}). We refer to \cite{FG} and its references for the following: 

\begin{proposition} With the notations introduced above, 
\begin{enumerate}
\item $\pi(X^{\gen})$ is an open subset of $Y$;
\item the rational map $\chi$ is regular on $\pi(X^{\gen})$ and induces an embedding
$$ \chi \colon \pi(X^{\gen}) \too (\P^1 \smallsetminus \{ 0, \infty\})^{\# A_{n-2, 4}}.$$
\end{enumerate}
\end{proposition}

The quotient $Y'$ is embedded in $\P^8$ through derangements. Going back to the notations introduced in Section \ref{par:QuadruplesLineHyperplane}, set
\begin{align*}
w_{12}=\frac{x_5}{x_8},
&& w_{13}=\frac{x_9}{x_2},
&& w_{14}=\frac{x_3}{x_5},
&& w_{23}=\frac{x_1}{x_3},
&& w_{24}=\frac{x_2}{x_1},&&
 w_{34}=\frac{x_5}{x_6}.
\end{align*}
Let  $x' = \{ (L_i, H_i)\}_{i = 1, \dots, 4}$ be a quadruple of flags line-hyperplane such that $L_i$ is not contained in $H_j$ for $i \neq j$. Then, for all $i < j$,
$$ w_{ij}(F) = \frac{\phi_i(v_\alpha) \phi_j(v_\beta)}{\phi_i(v_\alpha) \phi_j(v_\beta)},$$
where $\{ i, j , \alpha, \beta\} = (1, 2, 3, 4)$.

\begin{proposition} \label{propositionz>w} Let $F$ be a generic quadruple of complete flags. Write
\begin{align*}
\chi_\alpha := \chi_\alpha(F) \textup{ for } \alpha \in A_{n-2, 4}, && w_{ij} := w_{ij}(p(F)) \textup{ for } i<j.
\end{align*}
Then,
\begin{align*}
w_{12}&=\prod_{i+j=n-2} \chi_{ij00}, & w_{34}&=\prod_{i+j=n-2} \chi_{00ij}, \\
 w_{14}&=\prod_{i+j=n-2} 1 - \frac{1}{\chi_{i00j}} , &
w_{23}&=\prod_{i+j=n-2} 1 - \frac{1}{\chi_{0ij0}}  , \\
w_{24}&= \prod_{i+j=n-2} \frac{1}{1 - \chi_{0i0j}}, &w_{13}&=\prod_{i+j=n-2} \frac{1}{1 - \chi_{i0j0}}.
\end{align*}
\end{proposition}

\begin{proof} We prove the first equality, the others being obtained analogously.  For non-negative integers $\alpha, \beta$ such that $\alpha + \beta = n-1$ let $\phi_{\alpha \beta}$ be a linear form on $V$ defining the hyperplane $F^{(1)}_\alpha + F^{(2)}_\beta$. Let $v_3, v_4$ be respectively generators of the lines $F^{(3)}_1$ and $F^{(4)}_1$. With these notations and setting $E = F^{(1)}_{i} + F^{(2)}_{j}$, the cross-ratio of the lines
\begin{align*}
(F^{(1)}_{i+1} + E) / E, &&
(F^{(2)}_{j+1} + E)/ E, &&
(F^{(3)}_{1} + E) / E, &&
(F^{(4)}_{1} + E) / E,
\end{align*}
is
$$ \frac{\phi_{i+1,j}(v_3) \phi_{i,j+1}(v_4)}{\phi_{i+1,j}(v_4) \phi_{i,j+1}(v_3)}.$$
Therefore,
\begin{align*}
\prod_{i+j=n-2} \chi_{ij00} &= \prod_{i = 0}^{n-2} \frac{  \phi_{i,n-1-i}(v_4)\phi_{i+1,n-2-i}(v_3)}{ \phi_{i,n-1-i}(v_3) \phi_{i+1, n-2 - i}(v_4)} \\
&= \frac{  \phi_{0,n-1}(v_4) \phi_{n-1,0}(v_3)}{ \phi_{0,n-1}(v_3) \phi_{n-1,0}(v_4)}=w_{ij},
\end{align*}
as $\phi_{n-1, 0}, \phi_{0, n-1}$ are respectively linear forms defining $F^{(1)}_{n-1}$, $F^{(2)}_{n-1}$. \end{proof}

\begin{example}[$n = 3$] \label{ConfigurationsTripletsGeneric} In this case $N = 4$ and the elements of $A_{1, 4}$ are
\begin{align*}
(1, 0, 0, 0), && ( 0,1, 0, 0), && (0, 0, 1, 0), && ( 0, 0, 0, 1),
\end{align*}
and $\binom{4}{2} =6$. In particular,
\begin{align*}
w_{12}&= \chi_{1000} \chi_{0100}, & w_{34}&= \chi_{0010} \chi_{0001}, \\
 w_{14}&= \left( 1 - \frac{1}{\chi_{1000}} \right) \left( 1 - \frac{1}{\chi_{0001}} \right) , &
w_{23}&= \left( 1 - \frac{1}{\chi_{0100}} \right) \left( 1 - \frac{1}{\chi_{0010}} \right), \\
w_{24}&=  \frac{1}{1 - \chi_{0100}} \cdot \frac{1}{1 - \chi_{0001}}, &w_{13}&=\frac{1}{1 - \chi_{1000}} \cdot \frac{1}{1 - \chi_{0010}}.
\end{align*}
With the notations of \cite{BFG},
\begin{align*}
z_{12} &= \chi_{1000} & z_{21} &= \chi_{0100}, \\
z_{13} &= \frac{1}{1 - \chi_{1000}} & z_{31} &= \frac{1}{1 - \chi_{0010}}, \\
z_{14} &= 1 - \frac{1}{\chi_{1000}}, &z_{41} &= 1 - \frac{1}{\chi_{0001}} ,\\
z_{23} &= 1 - \frac{1}{\chi_{0100}}, &z_{32}&=  1 - \frac{1}{\chi_{0010}}, \\
z_{24} &= \frac{1}{1 - \chi_{0100}}, & z_{42}&=  \frac{1}{1 - \chi_{0001}}, \\
z_{34} &= \chi_{0010}, & z_{43}&=  \chi_{0001}.
\end{align*}
so that,
\begin{align*}
w_{12}&=z_{12}z_{21},
&w_{13}&=z_{13}z_{31},
&w_{14}&=z_{14}z_{41},\\
w_{23}&=z_{23}z_{32},
&w_{24}&=z_{24}z_{42},
&w_{34}&=z_{34}z_{43},
\end{align*}
A geometric way to remember this equations is shown in Figure \ref{pic:zCoordinatesTetrahedra}.

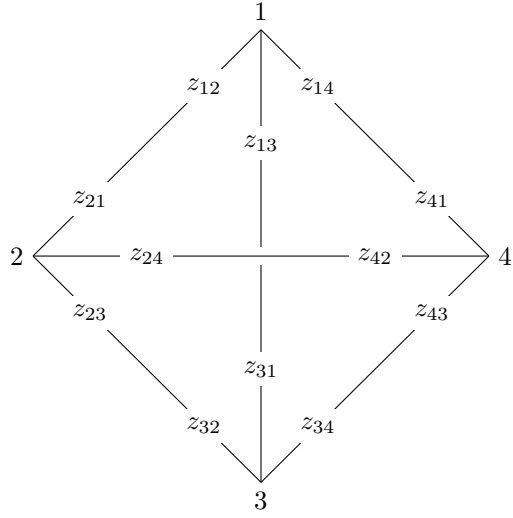
\begin{figure} 

\begin{tikzpicture}

\def\r{3};

\draw (0, \r) -- (\r, 0) node[near start, fill=white] {$z_{14}$} node[near end, fill=white] {$z_{41}$};

\draw (0, \r) -- (-\r, 0) node[near start, fill=white] {$z_{12}$} node[near end, fill=white] {$z_{21}$};

\draw (0, -\r) -- (\r, 0) node[near start, fill=white] {$z_{34}$} node[near end, fill=white] {$z_{43}$};

\draw (0, -\r) -- (-\r, 0) node[near start, fill=white] {$z_{32}$} node[near end, fill=white] {$z_{23}$};

\draw (0,-\r) -- (0, \r) node[midway, fill=white] {} node[near start, fill=white] {$z_{31}$} node[near end, fill=white] {$z_{13}$};

\draw (-\r, 0) -- (\r, 0) node[near start, fill=white] {$z_{24}$} node[near end, fill=white] {$z_{42}$};

\node[anchor=south] at (0, \r) {$1$};
\node[anchor=east] at (-\r, 0) {$2$};
\node[anchor=north] at (0, -\r) {$3$};
\node[anchor=west] at (\r, 0) {$4$};

\end{tikzpicture}

\caption{The $z$-coordinates.} \label{pic:zCoordinatesTetrahedra}

\end{figure}

\end{example}

\begin{example}[$n=4$] Let $\Gr_2(V)$ be the Grassmannian of planes in $V = \C^4$ and $X_2 = \Gr_2(V)^4$. Let $X_{2}^\ss$ be the open subset of semi-stable points of $X_2$ under the action of $G$ and with respect to the Pl\"ucker embedding (\emph{cf.} Section \ref{sec:QuadruplesOfPlanes}). The GIT quotient of $X_2^\ss$ by $G$ is isomorphic to $\P^2$ and the projection $\pi_2 \colon X_2^\ss \to \P^2$ is given by non trivial elements of the Klein group (see Theorem \ref{Thm:Quotient4planesC4}). Write $p_2 \colon X \to X_2$ the canonical projection.

\begin{proposition}\label{proposition:generic4toinvariants} Let $F$ be a generic quadruple of complete flags of $V = \C^4$. For $\alpha \in A_{2, 4}$ write $ \chi_\alpha := \chi_\alpha(F)$. Then,
\begin{align*}
\frac{s_{1234}}{s_{1423}}(p(F)) &= \chi_{0110} \chi_{0101} \chi_{1010} \chi_{1001},\\
\frac{s_{1324}}{s_{1423}}(p(F)) &=\left( 1- \frac{1}{\chi_{0011}} \right) \left( 1- \frac{1}{\chi_{1100}} \right) \left( 1- \frac{1}{\chi_{0101}} \right) \left( 1- \frac{1}{\chi_{1010}} \right).
\end{align*}
\end{proposition}

\begin{proof} For $i = 1, \dots, 4$ let $v_i$ be  a generator of the line $F^{(i)}_1$ and $\phi_i$ a linear form defining the hyperplane $F^{(i)}_3$. For $i, j = 1, \dots, 4$ distinct let $\phi_{ij}$ be a linear form for the hyperplane  $ F^{(i)}_2 + F^{(j)}_1$. Writing $E = F^{(1)}_1 + F^{(2)}_1$, the cross-ratio $\chi_{1100}$ of the lines
\begin{align*}
(F^{(1)}_2+E)/E, && (F^{(2)}_{2}+E)/E, && (F^{(3)}_1+E)/E, &&(F^{(4)}_1+E)/E,
\end{align*}
is
$$
\frac{\phi_{12}(v_3)\phi_{21}(v_4)}{\phi_{12}(v_4)\phi_{21}(v_3)}.
$$

Analogously,

\begin{align*}
 \frac{1}{1-\chi_{1010}}&=\frac{\phi_{13}(v_4)\phi_{31}(v_2)}{\phi_{13}(v_2)\phi_{31}(v_4)}, &
1 - \frac{1}{\chi_{1001}}&=\frac{\phi_{14}(v_2)\phi_{41}(v_3)}{\phi_{14}(v_3)\phi_{41}(v_2)}, \\
1 - \frac{1}{\chi_{0110}}&=\frac{\phi_{23}(v_1)\phi_{32}(v_4)}{\phi_{23}(v_4)\phi_{32}(v_1)}, &
\frac{1}{1-\chi_{0101}}&=\frac{\phi_{42}(v_1)\phi_{24}(v_3)}{\phi_{42}(v_3)\phi_{24}(v_1)}, \\
\chi_{0011}&=\frac{\phi_{34}(v_1)\phi_{43}(v_2)}{\phi_{34}(v_2)\phi_{43}(v_1)}.
\end{align*}

Up to action of $\SL_4(\C)$ the complete flags $F^{(1)}, \dots, F^{(4)}$ can be taken as:
\begin{align*}
F^{(1)}_1 = \langle \begin{pmatrix} 0 \\ 0 \\ 1 \\ \lambda_1 \end{pmatrix} \rangle, && 
F^{(1)}_2 = \langle \begin{pmatrix} 0 \\ 0 \\ 1 \\ 0 \end{pmatrix}, \begin{pmatrix} 0\\ 0 \\ 0 \\ 1\end{pmatrix} \rangle, && 
F^{(1)}_3 = \langle \begin{pmatrix} 0 \\ 0 \\ 1 \\ 0 \end{pmatrix}, \begin{pmatrix} 0 \\ 0 \\ 0 \\ 1\end{pmatrix},
                               \begin{pmatrix} 1 \\ \mu_1 \\ 0 \\ 0\end{pmatrix} \rangle,\\
F^{(2)}_1 = \langle \begin{pmatrix} 1 \\ \lambda_2 \\ 0 \\ 0 \end{pmatrix} \rangle, && 
F^{(2)}_2 = \langle \begin{pmatrix} 1 \\ 0 \\ 0 \\ 0 \end{pmatrix}, \begin{pmatrix} 0\\ 1 \\ 0 \\  0\end{pmatrix} \rangle, && 
F^{(2)}_3 = \langle \begin{pmatrix} 1 \\ 0 \\ 0 \\ 0 \end{pmatrix}, \begin{pmatrix} 0\\ 1 \\ 0 \\  0\end{pmatrix},
                               \begin{pmatrix} 0 \\ 0 \\ \mu_2 \\ 1\end{pmatrix} \rangle,\\
F^{(3)}_1 = \langle \begin{pmatrix} 1 \\ \lambda_3 \\ 1 \\ \lambda_3 \end{pmatrix} \rangle, && 
F^{(3)}_2 = \langle \begin{pmatrix} 1 \\ 0 \\ 1 \\ 0 \end{pmatrix}, \begin{pmatrix} 0\\ 1 \\ 0 \\ 1\end{pmatrix} \rangle, && 
F^{(3)}_3 = \langle \begin{pmatrix} 1 \\ 0 \\ 1 \\ 0 \end{pmatrix}, \begin{pmatrix} 0 \\ 1 \\ 0 \\ 1\end{pmatrix},
                               \begin{pmatrix} 1 \\ \mu_3 \\ 0 \\ 0\end{pmatrix} \rangle,\\
F^{(4)}_1 = \langle \begin{pmatrix} a \\ \lambda_4 b \\ 1 \\ \lambda_4 \end{pmatrix} \rangle, && 
F^{(4)}_2 = \langle \begin{pmatrix} a \\ 0 \\ 1 \\ 0 \end{pmatrix}, \begin{pmatrix} 0\\ b \\ 0 \\ 1\end{pmatrix} \rangle, && 
F^{(4)}_3 = \langle \begin{pmatrix} a \\ 0 \\ 1 \\ 0 \end{pmatrix}, \begin{pmatrix} 0 \\ b \\ 0 \\ 1\end{pmatrix},
                               \begin{pmatrix} 1 \\ \mu_4 \\ 0 \\ 0\end{pmatrix} \rangle,
\end{align*}
where $\lambda_1, \dots, \lambda_4, \mu_1, \dots, \mu_4 \in \C$ and $a, b \in \C^\times$. According to Proposition \ref{Prop:4PlanesAffineChart},
\begin{align*}
\frac{s_{1234}}{s_{1423}} = \frac{1}{ab}, &&  \frac{s_{1324}}{s_{1423}} = (a -1)(b -1).
\end{align*}
We leave the remaining computations to the reader.
\end{proof}  

\begin{figure}[h]
\begin{tikzpicture}
\path[pattern=north west lines, pattern color=gray] (1,4) -- (2.75,0.75) -- (4.75,11) -- (1,4);
\draw (2.75,0.75) -- (11,3);
\draw  (11,3) -- (4.75,11);
\draw (1,4) -- (2.75,0.75);

\foreach \i in {1,2,3} {
            
            \pgfmathtruncatemacro{\a}{3-\i)};
            \pgfmathtruncatemacro{\b}{\i-1)};
            \node[anchor=east, fill=white] at (2.75 + \i*2/4,0.75 + \i*10.25/4) {\footnotesize \b\a00};
            \draw[fill=black] (2.75 + \i*2/4,0.75 + \i*10.25/4) circle (1pt);
        }

\foreach \i in {1,2,3} {
            \draw[fill=black] (11 - \i*6.25/4,3 + \i*8/4) circle (1pt);
            \pgfmathtruncatemacro{\a}{3-\i)};
            \pgfmathtruncatemacro{\b}{\i-1)};
            \node[anchor=west] at (11 - \i*6.25/4,3 + \i*8/4) {\footnotesize \b0\a0'};
        }
        
\foreach \i in {1,2,3} {
            \draw[fill=black] (1 + \i*3.75/4,4 + \i*7/4) circle (1pt);
            \pgfmathtruncatemacro{\a}{3-\i)};
            \pgfmathtruncatemacro{\b}{\i-1)};
            \node[anchor=east] at (1 + \i*3.75/4,4 + \i*7/4) {\footnotesize \b00\a''};
}

\foreach \i in {1,2,3} {
            \draw[fill=black] (1 + \i*1.75/4,4 - \i*3.25/4) circle (1pt);
            \pgfmathtruncatemacro{\a}{3-\i)};
            \pgfmathtruncatemacro{\b}{\i-1)};
            \node[anchor=east] at (1 + \i*1.75/4,4 - \i*3.25/4) {\footnotesize 0\a0\b'};
}

\foreach \i in {1,2,3} {
            
            \pgfmathtruncatemacro{\a}{3-\i)};
            \pgfmathtruncatemacro{\b}{\i-1)};
            \node[anchor=south, fill=white] at (1 + \i*10/4,4  -\i*1/4) {\footnotesize 00\a\b};
            \draw[fill=black] (1 + \i*10/4,4  -\i*1/4) circle (1pt);
}

\foreach \i in {1,2,3} {
            \draw[fill=black] (2.75 + \i*8.25/4,0.75  +\i*2.25/4) circle (1pt);
            \pgfmathtruncatemacro{\a}{3-\i)};
            \pgfmathtruncatemacro{\b}{\i-1)};
            \node[anchor=north] at (2.75 + \i*8.25/4,0.75  +\i*2.25/4) {\footnotesize 0\a\b0''};
}

\foreach \i in {1,2,3} {
            \draw (2.75 + \i*2/4,0.75 + \i*10.25/4) -- (11 - \i*6.25/4,3 + \i*8/4);
            \draw (2.75 + \i*2/4,0.75 + \i*10.25/4) -- (2.75 + \i*8.25/4,0.75  +\i*2.25/4);
            \draw (11 - \i*6.25/4,3 + \i*8/4) -- (11 - \i*8.25/4,3  -\i*2.25/4);
        }        

\draw (4.75,11) --  (2.75,0.75) ;
\draw (4.75,11) -- (1,4);
\draw[dashed] (1,4) -- (11,3);

\end{tikzpicture}

\caption{The $n=4$ case of the cross-ratio coordinates.  Each coordinate is associated to a quadruple $\alpha = (\alpha_1,\alpha_2,\alpha_3,\alpha_4)$ of non-negative integers $\alpha_i$ whose sum is $n-2$,  attached to a point in one edge.}
\end{figure}
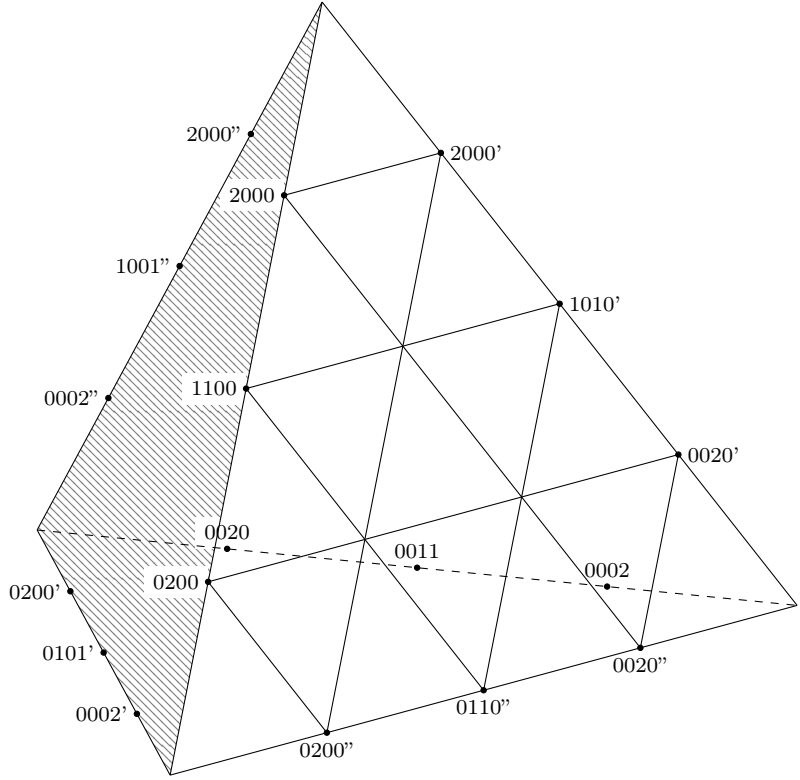

\end{example}

\section{Representations}\label{section:Representations}

In this section we recall the definition of decorated triangulations of three manifolds following \cite{BFG,GGZ,DGG}.  The definition is related to the case of decorated triangulated surfaces (see \cite{FG}).
The decoration corresponds to an assignment of a configuration of four complete flags in $\C^n$ to each  3-simplex in the 3-manifold satisfying certain compatibility conditions.  The flags in the configuration can be thought as associated to each of the four vertices
of a simplex.
We will be interested mainly in the case $n=4$.  See \cite{BFG,FKR} for the case $n=3$.  The decorated triangulations will give rise to representations of the fundamental group into $\PGL(4,\C)$ or $\SL(4,\C)$ and  the goal is to find, among these representations, those with values in one of the real forms.  Recall that the non-compact real forms of $\SL(4,\C)$ are $\SL(4,\R)$, $\SU(3,1)$, $\SU(2,2)$ and $\SL(2,\H)$.   

We use coordinates on generic configurations of the configuration space (modulo $\PGL(n,\C)$) introduced in the previous section and  recall that
these coordinates can be expressed using the invariant functions described using GIT for $\SL(n,\C)$.  This allows one to use  the description in Proposition \ref{propositionz>w} and Example \ref{ConfigurationsTripletsGeneric} to identify configurations in the 
closed orbit of certain real forms of $\PGL(n,\C)$.

\subsection{Decorated triangulations of 3-manifolds and representations}

We consider an ideal triangulated compact oriented 3-manifold $M$ with boundary $\partial M$. That is, there exists a finite number of oriented 3-simplices $T_\nu$ ($1\leq \nu \leq N$) and an orientation reversing matching of their faces so that if one truncates all the simplices at their vertices one obtains, after identifying the truncated faces, a manifold homeomorphic 
to $M$ with $\partial M$ identified to a triangulated 2-surface.  

Let $V = \C^n$ and  the flag variety $F$ of complete flags.  We let $G= \PGL(n,\C)$ be the projective group acting on the space of flags. 

\begin{definition}
A decoration of an ideal triangulation of a manifold $M$ is a collection of maps from tetrahedra $T_\nu$ 
to the space of configurations $F^{4,\gen}/G$ compatible with the matchings.  More precisely, to the 0-simplex of $T_\nu$ one associates 
a configuration of flags up to conjugation such that to each matching face one associates the same point in $F^{3, \gen}/G$.
 \end{definition}

We will impose in the following that  the decorations are chosen such that the configurations in $F^{3, \gen}$ have trivial stabilizer. This condition is clearly satisfied for generic configurations as defined in the previous section.
We also assume that the \emph{holonomy} around the 1-skeleton be trivial.
That is,  around each edge one choses an abutted face and consider a configuration in $F^{3, \gen}$ associated to it (a lift of the element in $F^{3, \gen}/G$ associated to the face).  Turning around the edge one constructs 
a sequence of flags corresponding to the vertices (not contained in the edge) of the faces of each of the simplices around the edge.  After completing a turn around the edge we suppose that
the we recover the same original configuration  in $F^{3, \gen}$.

One can then obtain a holonomy representation (defined up to conjugation)
$$
\rho: \pi_1(M)\too \PGL(n,\C).
$$

Indeed, chose a base face and consider a configuration in $F^{3, \gen}$ associated to it (a lift of the element in $F^{3, \gen}/G$ associated to the face).  Any closed path $\gamma$ (based at the chosen face and not intersecting any edge) determines a sequence of configurations in  $F^{3, \gen}$ corresponding to the configurations for each successive face crossed by the path.
The last face configuration of flags along the path is the image of the reference configuration of flags by a unique element in $G$.  This element is called the \emph{holonomy} 
along $\gamma$.  Clearly, if the holonomy around all edges are trivial, the paths may cross edges without modifying the holonomy and therefore the holonomy representation is well defined on  $\pi_1(M)$.

\subsubsection{Unipotent decorations}

An important class of decorations is obtained by restricting the boundary holonomy, that is, the holonomy  along paths on the boundary of $M$, to be 
in a subgroup which is conjugate to a unipotent radical of a Borel subgroup of $G$.  In this paper, we consider only the case of $\SL(n,\C)$ or $\PGL(n,\C)$ and therefore, for a boundary component, the boundary holonomy may be considered to consist of upper triangular matrices with equal entries in the diagonal.   

Computations of decorations with complete flags were implemented in \cite{C}.   Unipotent decorations were computed for
$\PGL(2,\C)$, $\PGL(3,\C)$ in many cases and for one manifold, the complement of the figure eight knot, for $\PGL(4,\C)$.

\subsubsection{$G_0$-decorations and representations}

\begin{definition}
Let $G_0$ be a real form of $G$ and $F_0$ the closed $G_0$-orbit in $F$.  A $G_0$-decoration of an ideal triangulation of $M$ is a decoration obtained through a family of maps from each tetrahedron to $F_0^{4, \gen}/G_0$ (here $F_0^{4, \gen} = F^{4, \gen} \cap F_0$).
\end{definition}

The holonomy representation representation
$
\rho: \pi_1(M)\rightarrow G_0
$ 
will be $G_0$-valued in this case.

In order to obtain explicit $G_0$-decorations we use coordinates for generic flags, excluding certain degenerate configurations.
For a given triangulation of a 3-manifold this might result in a loss of certain representations of its fundamental group but
passing to a barycentric subdivision one  obtains all representations.  In generic coordinates, the $G_0$-decorations are described as a constructible set given by algebraic equations
implying face matchings and trivial holonomy around edges.  The condition that the decoration be unipotent is described by additional algebraic equations.  We will not reproduce
these equations here but refer to references \cite{BFG,GGZ,DGG}.

Unfortunately, computations are difficult and were completed only for triangulations with less than or equal to four tetrahedra in the case of $\PGL(3,\C)$ or
2 tetrahedra in the case of  $\PGL(4,\C)$.  
The case $\PGL(3,\C)$, for manifolds obtained with less than 3 simplices, was treated in \cite{FKR} (see also \cite{C} for a census of manifolds triangulated by four tetrahedra).  $G_0$-decorations give rise to representations into the real forms $\PU(2,1)$ and $\SL(3,\R)$.  In both cases, one can, in certain cases, associate $(G_0,F_0)$-structures to the 3-manifold such that its holonomy coincides with the representation.  We refer to \cite{F,FS} for more details. 

In the case $\PGL(4,\C)$, a complete computation for unipotent decorated representations of the standard triangulation of the complement of the figure eight knot $M_8$ by two tetrahedra
was completed by Matthias Goerner (\cite{C}).  Using those  data, Antonin Guilloux checked, using the description in Proposition \ref{propositionz>w} and Example \ref{ConfigurationsTripletsGeneric}
that there exists, up to conjugation, only one family of unipotent decorated triangulations which gives rise to a unique representation in a real form.  Namely, the  representation  obtained from the geometric representation $\pi_1(M_8)\rightarrow \SO(3,1)$ which realizes a lift of the holonomy of the complete hyperbolic structure on the isometry group of hyperbolic space.  Observe that the image of the representation is contained in both $\SL(4,\R)$ and  $\SU(3,1)$.  

\begin{proposition}\label{proposition:figureeight}
The only $G_0$ decorations associated to the standard triangulation of the figure eight knot by two tetrahedra correspond to a unique representation, up to conjugation, $\pi_1(M_8)\rightarrow \SO(3,1)$ identified to the geometric representation.
\end{proposition}
 We believe that this is the only irreducible boundary unipotent representation in a real form up to conjugation.

It is interesting to note that  (see Theorem 1.4 in \cite{BPW})  this unipotent representation is contained in a one parameter family of discrete and faithful boundary parabolic representations with values in $\SU(3,1)$.

\small

\providecommand{\bysame}{\leavevmode\hbox to3em{\hrulefill}\thinspace}
\providecommand{\MR}{\relax\ifhmode\unskip\space\fi MR }
\providecommand{\MRhref}[2]{%
  \href{http://www.ams.org/mathscinet-getitem?mr=#1}{#2}
}
\providecommand{\href}[2]{#2}


\begin{thebibliography}{BPW17}

\bibitem[ABP73]{AtiyahBottPatodi}
M.~Atiyah, R.~Bott, and V.~K. Patodi, \emph{On the heat equation and the index
  theorem}, Invent. Math. \textbf{19} (1973), 279--330. \MR{0650828}

\bibitem[BFG14]{BFG}
N.~Bergeron, E.~Falbel, and A.~Guilloux, \emph{Tetrahedra of flags, volume and
  homology of {${\rm SL}(3)$}}, Geom. Topol. \textbf{18} (2014), no.~4,
  1911--1971. \MR{3268771}

\bibitem[BG09]{BisiGentili}
C.~Bisi and G.~Gentili, \emph{M\"obius transformations and the {P}oincar\'e
  distance in the quaternionic setting}, Indiana Univ. Math. J. \textbf{58}
  (2009), no.~6, 2729--2764. \MR{2603766}

\bibitem[Bir71]{Birkes}
D.~Birkes, \emph{Orbits of linear algebraic groups}, Ann. of Math. (2)
  \textbf{93} (1971), 459--475. \MR{0296077}

\bibitem[Bor91]{Borel}
A.~Borel, \emph{Linear algebraic groups}, second ed., Graduate Texts in
  Mathematics, vol. 126, Springer-Verlag, New York, 1991.

\bibitem[BPW17]{BPW}
S.~A. Ballas, J.~Paupert, and P.~Will, \emph{Rank 1 deformations of
  non-cocompact hyperbolic lattices}, \texttt{arXiv:1702.00508} (2017).

\bibitem[C]{C}
\emph{Curve (computing and understanding representation varieties
  efficiently)}.

\bibitem[DGG16]{DGG}
T.~Dimofte, M.~Gabella, and A.~B. Goncharov, \emph{K-decompositions and 3d
  gauge theories}, J. High Energy Phys. (2016), no.~11, 151, front matter+144.
  \MR{3594814}

\bibitem[Fal08]{F}
E.~Falbel, \emph{A spherical {CR} structure on the complement of the figure
  eight knot with discrete holonomy}, J. Differential Geom. \textbf{79} (2008),
  no.~1, 69--110. \MR{2401419}

\bibitem[FG07]{FG}
V.~V. Fock and A.~B. Goncharov, \emph{Moduli spaces of convex projective
  structures on surfaces}, Adv. Math. \textbf{208} (2007), no.~1, 249--273.
  \MR{2304317}

\bibitem[FKR15]{FKR}
E.~Falbel, P.-V. Koseleff, and F.~Rouillier, \emph{Representations of
  fundamental groups of 3-manifolds into {$\textup{PGL}(3,\mathbb{C})$}: exact
  computations in low complexity}, Geom. Dedicata \textbf{177} (2015),
  229--255. \MR{3370032}

\bibitem[FST15]{FS}
E.~Falbel and R.~Santos~Thebaldi, \emph{A flag structure on a cusped hyperbolic
  3-manifold}, Pacific J. Math. \textbf{278} (2015), no.~1, 51--78.
  \MR{3404666}

\bibitem[GIT]{GIT}
D.~Mumford, J.~Fogarty, and F.~Kirwan, \emph{Geometric invariant theory}, third
  ed., Ergebnisse der Mathematik und ihrer Grenzgebiete (2) [Results in
  Mathematics and Related Areas (2)], vol.~34, Springer-Verlag, Berlin, 1994.
  \MR{1304906}  

\bibitem[GGZ15]{GGZ}
S.~Garoufalidis, M.~Goerner, and C.~K. Zickert, \emph{Gluing equations for
  {$\textup{PGL}(n,\mathbb{C})$}-representations of 3-manifolds}, Algebr. Geom.
  Topol. \textbf{15} (2015), no.~1, 565--622. \MR{3325748}

\bibitem[GL12]{GwynneLibine}
E.~Gwynne and M.~Libine, \emph{On a quaternionic analogue of the cross-ratio},
  Adv. Appl. Clifford Algebr. \textbf{22} (2012), no.~4, 1041--1053.
  \MR{2994138}

\bibitem[HV50]{HalmosMarriage}
P.~R. Halmos and H.~E. Vaughan, \emph{The marriage problem}, Amer. J. Math.
  \textbf{72} (1950), 214--215. \MR{0033330}

\bibitem[Kem78]{KempfInstability}
G.~R. Kempf, \emph{Instability in invariant theory}, Ann. of Math. (2)
  \textbf{108} (1978), no.~2, 299--316. \MR{506989}

\bibitem[Lun73]{LunaSlicesEtales}
D.~Luna, \emph{Slices \'etales}, Sur les groupes alg\'ebriques, Soc. Math.
  France, Paris, 1973, pp.~81--105. Bull. Soc. Math. France, Paris, M\'emoire
  33. \MR{0342523}

\bibitem[Lun75]{LunaRealGIT}
\bysame, \emph{Sur certaines op\'erations diff\'erentiables des groupes de
  {L}ie}, Amer. J. Math. \textbf{97} (1975), 172--181. \MR{0364272}

\bibitem[Lun76]{LunaFonctionsDifferentiables}
\bysame, \emph{Fonctions diff\'erentiables invariantes sous l'op\'eration d'un
  groupe r\'eductif}, Ann. Inst. Fourier (Grenoble) \textbf{26} (1976), no.~1,
  ix, 33--49. \MR{0423398}

\bibitem[MB12]{MoretBailly}
L.~Moret-Bailly, \emph{Un th\'eor\`eme de l'application ouverte sur les corps
  valu\'es alg\'ebriquement clos}, Math. Scand. \textbf{111} (2012), no.~2,
  161--168. \MR{3023520}

\bibitem[MS72]{MumfordOslo}
D.~Mumford and K.~Suominen, \emph{Introduction to the theory of moduli},
  Algebraic geometry, {O}slo 1970 ({P}roc. {F}ifth {N}ordic {S}ummer-{S}chool
  in {M}ath.), Wolters-Noordhoff, Groningen, 1972, pp.~171--222. \MR{0437531}

\bibitem[Rou81]{Rousseau}
G.~Rousseau, \emph{Instabilit\'e dans les espaces vectoriels}, Algebraic
  surfaces ({O}rsay, 1976--78), Lecture Notes in Math., vol. 868, Springer,
  Berlin-New York, 1981, pp.~263--276. \MR{638603}

\bibitem[RS90]{RichardsonSlodowy}
R.~W. Richardson and P.~J. Slodowy, \emph{Minimum vectors for real reductive
  algebraic groups}, J. London Math. Soc. (2) \textbf{42} (1990), no.~3,
  409--429. \MR{1087217}

\bibitem[Ser56]{SerreGAGA}
J.-P. Serre, \emph{G\'eom\'etrie alg\'ebrique et g\'eom\'etrie analytique},
  Ann. Inst. Fourier, Grenoble \textbf{6} (1955--1956), 1--42. \MR{0082175}

\bibitem[Ser94]{SerreCohomologieGaloisienne}
\bysame, \emph{Cohomologie galoisienne}, fifth ed., Lecture Notes in
  Mathematics, vol.~5, Springer-Verlag, Berlin, 1994. \MR{1324577}

\bibitem[Wol69]{W}
J.~A. Wolf, \emph{The action of a real semisimple group on a complex flag
  manifold. {I}. {O}rbit structure and holomorphic arc components}, Bull. Amer.
  Math. Soc. \textbf{75} (1969), 1121--1237. \MR{0251246}

\end{thebibliography}
\end{document}